\documentclass[preprint,3p,times]{elsarticle}

\usepackage{amssymb}
\usepackage{amsthm}

\usepackage{lineno}

\usepackage{amsmath, amsfonts}
\usepackage{mathabx}   
\usepackage[makeroom]{cancel}    
\usepackage{tipa}      
\usepackage{textcomp}  
\usepackage{latexsym}  
\usepackage{mathrsfs}  

\usepackage{graphicx}
\usepackage{tikz}
\usepackage{pgfplots}
\pgfplotsset{compat=1.6}
\usepgfplotslibrary{groupplots}
\usepackage{pgf}
\usetikzlibrary{positioning,shapes,arrows,fit,calc,automata,perspective,3d}
\usepackage{neuralnetwork}
\usepackage{tikz-network}

\usepackage{hyperref}
\hypersetup{
	colorlinks=true,       
	linkcolor=blue,
	citecolor=red,
	filecolor=magenta,      
	urlcolor=cyan           
}
\usepackage{cleveref}

\usepackage{hhline}

\usepackage{multirow}
\usepackage{colortbl} 

\usepackage{algpseudocode}
\usepackage{algorithm}

\usepackage{subcaption} 

\newtheorem{defin}{Definition}[section]%
\newtheorem{prop}{Proposition}[section]%
\newtheorem{lem}{Lemma}[section]%
\newtheorem{teo}{Theorem}[section]%
\newtheorem{alg}{Algorithm}[section]%

\theoremstyle{definition}
\newtheorem{rem}{Remark}[section]
\newtheorem{notaz}{Notation}[section]

\newcommand{\norm}[1]{\Vert{#1}\Vert}	

\newcommand{\N}{\ensuremath{\mathbb{N}}}    
\newcommand{\R}{\ensuremath{\mathbb{R}}}	

\DeclareMathOperator{\argmin}{arg\,min}

\renewcommand{\v}{\ensuremath{\boldsymbol}}	

\journal{arXiv}

\begin{document}

\begin{frontmatter}



\title{A New Linear Programming Approach and a New Backtracking Strategy for Multiple-Gradient Descent in Multi-Objective Optimization}


\author[inst1,inst3]{Francesco Della Santa\corref{cor1}}
\cortext[cor1]{Corresponding author}

\affiliation[inst1]{organization={Department of Mathematical Sciences, Politecnico di Torino},
            addressline={Corso Duca degli Abruzzi 24}, 
            postcode={10129}, 
            state={Turin},
            country={Italy}}

\affiliation[inst3]{organization={Gruppo Nazionale per il Calcolo Scientifico INdAM},
            addressline={Piazzale Aldo Moro 5}, 
            postcode={00185}, 
            state={Rome},
            country={Italy}}



\begin{abstract}
In this work, the author presents a novel method for finding descent directions shared by two or more differentiable functions defined on the same unconstrained domain space. Then, the author illustrates an alternative Multiple-Gradient Descent procedure for Multi-Objective Optimization problems that is based on this new method. In particular, the proposed method consists in finding the shared descent direction solving a relatively cheap Linear Programming (LP) problem, where the LP's objective function and the constraints are defined by the gradients of the objective functions of the Multi-Objective Optimization problem. More precisely, the formulation of the LP problem is such that, if a shared descent direction does not exist for the objective functions, but a non-ascent direction for all the objectives does, the LP problem returns the latter. Moreover, the author defines a new backtracking strategy for Multiple-Gradient Descent methods such that, if the proposed LP is used for computing the direction, the ability to reach and/or explore the Pareto set and the Pareto front is improved.
A theoretical analysis of the properties of the new methods is performed, and tests on classic Multi-Objective Optimization problems are proposed to assess their goodness.
\end{abstract}







\end{frontmatter}



\section{Introduction}\label{sec:intro}

Multi-Objective Optimization (MOO) is the area in optimization that deals with problems that involve multiple, and often competing, objectives that must be optimized simultaneously. The importance of such a kind of optimization problems is evident from the wide MOO applications in real-world scenarios, ranging from mechanical engineering and fluid dynamics \cite{Schiffmann2006, zadeh2009, Colombo2019}, to energy-saving strategies \cite{CUI2017681,WU20181613}, to task allocation strategies \cite{Chen2022}, and many other topics. These applications have in common the need to find optimal trade-offs between different objectives, typically characterized by competing behaviors with respect to the optimization variables. The trade-off solutions typically are represented by the so-called Pareto set, while their images in the objectives' space by the so-called Pareto front. For more details about the theory of MOO problems, we refer to \cite{Miettinen1999book, Ehrgott2005book, Ehrgott2005}.

Often, MOO problems are solved using derivative-free methods \cite{Larson_Menickelly_Wild_2019}, where the most used ones are the nature-inspired methods \cite{NatInspired_Opt_Yang2020}. In particular, Genetic Algorithms \cite{GA_Mitchell1998,Deb2002,McCall2005,GA_YANG202191,Katoch2021} and Particle Swarm Optimization algorithms \cite{PSO_1995, PSO_YANG2021111, Xu2015}, are the most popular. The main advantages of these methods are their derivative-free nature and their efficiency in exploring the domain through ``populations'' of solutions; on the other hand, they lack ``realistic'' theoretical convergence properties (e.g., see \cite{GA_YANG202191}) or they have difficulties in finding accurate solutions, typically due to premature convergence (e.g., see \cite{Xu2015,Katoch2021}).
Another common approach for solving MOO problems involves minimizing a single loss function formed by a weighted sum of the objective functions \cite{Miettinen1999book} or using specific scalarization techniques (e.g., see \cite{Pascoletti1984, Khorram2014}). These methods have stronger theoretical foundations but require multiple runs with varying parameter values to explore the Pareto set/front in depth.

A last approach for MOO problems is represented by Multiple-Gradient Descent (MGD) methods \cite{Fliege2000, Schaffler2002, Desideri2009, Desideri2012_MGDA2, Desideri2012, Peitz2018}. MGD methods aim to build a sequence that converges to the Pareto set by solving sub-problems that determine descent directions shared by all the objective functions. These methods are characterized by robust theoretical properties; nonetheless, similarly to the scalarization techniques, they must be executed multiple times to fully explore the Pareto set/front, starting from different points sampled from the domain.
The sub-problems used in MGD methods for computing the shared descent directions are often nonlinear; for example, in \cite{Schaffler2002, Desideri2009, Desideri2012_MGDA2, Desideri2012, Peitz2018} the sub-problem consists in finding the minimum-norm vector in the anti-gradients’ convex hull. Nonetheless, there are some approaches based on Linear Programming (LP) problems, as illustrated in \cite{Fliege2000}. Even if the methods for solving LP problems are fast and very efficient, often the nonlinear sub-problems are preferred, probably due to the different properties of the directions they return. For example, the anti-gradients’ convex hull used in \cite{Schaffler2002, Desideri2009, Desideri2012_MGDA2, Desideri2012} restricts the search of the descent directions to a sub-region of all the shared descent directions that, unless of particular cases, is sufficiently distant to the boundary defined by the set of perpendicular directions to one of the gradients; on the other hand, the LP problem described in \cite{Fliege2000} look for the direction in the whole region of shared descent directions, therefore the solution can be almost perpendicular to some of the objective gradients. More details and theoretical properties of this LP problem will be given in this work since it is used as a baseline for the new MGD method proposed.

In this work, the author proposes a new LP sub-problem that is a trade-off between the cheap LP sub-problem defined in \cite{Fliege2000} and the nonlinear sub-problems with solutions characterized by ``steepness properties''. Specifically, we introduce a new LP problem for computing directions in MGD methods such that its solution is a direction that tries to maximize the distance from the boundaries of the descent directions' region and tries to follow as much as possible the direction identified by the sum of the anti-gradients. For doing so, we start from the LP problem described in \cite{Fliege2000} and we modify its objective function and its constraints to achieve these properties. A theoretical analysis for characterizing the solutions of the new LP problem is performed. In particular, we prove that the new LP formulation admits the null direction as a solution for Pareto critical points only in the following two situations: $i$) all the non-ascent directions shared by all the objective functions are perpendicular to all the gradients; $ii$) the only non-ascent direction shared by all the objectives is the null vector. In a third situation, where there is at least one non-ascent direction that is a descent direction (i.e., not null) for at least one objective, the LP problem returns one of these vectors and not the null vector. The latter characteristic (not present in the baseline LP problem) is an extra property that the author deliberately incorporates in the new LP problem to enhance its interaction with a new backtracking strategy for MGD methods, which is also proposed in this work.

To the best of the author's knowledge, the backtracking strategies used for MGD methods are always focused on building a sequence of vectors $\{\v{x}^{(k)}\}_{k\in\N}$ that is strictly decreasing for all the objective functions (e.g., see all the MGD methods cited above). These backtracking strategies can lead the sequence to very good (approximated) Pareto optimal solutions or can stop it ``prematurely'', i.e., as soon as they do not find a sufficiently small step to decrease all the objectives along the chosen direction. In order to avoid this latter case and improve the possibility of building a sequence that converges toward a Pareto optimal solution, we develop a new backtracking strategy. This new strategy is designed to accept the new point $\v{x}^{(k+1)}$ if it is non-dominated by the previous point $\v{x}^{(k)}$ when the Armijo condition is not satisfied for all objectives for all the backtracking steps.
Additionally, we endow the backtracking strategy with a ``storing property''; i.e., in a sequence, we keep apart all the vectors $\v{x}^{(k)}$ that do not dominate $\v{x}^{(k+1)}$ and that are not dominated by it, since this vectors are candidate Pareto optimals, and not only the last one. A theoretical analysis of MGD methods based on this new backtracking strategy is performed and convergence properties are established.

Through numerical experiments, we analyze the performance of our new methods, both the new LP problem and the new backtracking strategy, and we compare them with the baseline given by the literature in \cite{Fliege2000}. The experiment results indicate that our backtracking strategy consistently offers only advantages with respect to the ``strictly decreasing'' ones. While the new LP problem performs well, its benefits are most evident when combined with the new backtracking strategy, outperforming the baseline when applied to MOO problems characterized by large regions of Pareto critical points.

The work is organized as follows. In the next section (\Cref{sec:MOOintro}), the main notations and definitions used in MOO are listed. In \Cref{sec:mymethod}, the new LP sub-problem for shared descent directions is presented, whereas in Subsection \ref{sec:theory} the theoretical characterization of the sub-problem solutions is described. In \Cref{sec:mgd} the new backtracking strategy is introduced and both the theoretical convergence analyses and its implementation pseudo-codes are reported. \Cref{sec:experiments} illustrates numerical results where the proposed methods are compared with a baseline on some typical MOO test problems taken from the literature; in particular, the results assess the potential of the new LP sub-problem and the new backtracking strategy. We end the paper with some conclusions drawn in \Cref{sec:conclusion}.

\section{Multi-Objective Optimization Setting and Notations}\label{sec:MOOintro}

In this work, we consider the case of an unconstrained Multi-Objective Optimization (MOO) problem
\begin{equation}\label{eq:MOOprob}
    \min_{\v{x}\in\R^n}(f_1(\v{x}),\ldots ,f_m(\v{x}))\,,
\end{equation}
with $m$ differentiable objective functions $f_i:\R^n\rightarrow\R$, for each $i=1,\ldots ,m$.

In this section, for the reader's convenience, we recall the main definitions related to MOO problems like \eqref{eq:MOOprob} and Pareto optimality. For a more detailed introduction about MOO, for example, see \cite{Miettinen1999book,Ehrgott2005book,Ehrgott2005}.

\begin{notaz}[Order relations and vectors]
    Let $\v{v}, \v{w}\in\R^n$. Then, we denote by $\v{v}<\v{w}$ the relation between $\v{v}$ and $\v{w}$ such that $v_i< w_i$, for each $i=1,\ldots , n$ (analogous for $\leq$, $>$, and $\geq$). We denote by $\v{v}\not <\v{w}$ the opposite of the relation $\v{v}<\v{w}$ (analogous for $\not\leq$, $\not >$, and $\not \geq$); i.e., $\v{v}\not < \v{w}$ if there is $i\in\{1,\ldots ,n\}$ such that $v_i\geq w_i$.
\end{notaz}

\begin{defin}\label{def:Pareto_definitions}
    Let $\v{f}:\R^n\rightarrow\R^m$ be the function such that $\v{f}(\v{x})=(f_1(\v{x}),\ldots ,f_m(\v{x}))$, where $f_1,\ldots f_m$ are the $m$ objective functions of \eqref{eq:MOOprob}. Then, with respect to \eqref{eq:MOOprob}, we have that:
    \begin{enumerate}
        \item $\v{x}_1\in\R^n$ \emph{dominates} $\v{x}_0\in\R^n$ if $\v{f}(\v{x}_1)\leq \v{f}(\v{x}_0)$ and $\v{f}(\v{x}_1) \neq \v{f}(\v{x}_0)$. Alternatively, we can say that $\v{x}_0$ \emph{is dominated by} $\v{x}_1$. 
        
        \item $\v{x}_0\in\R^n$ is \emph{non-dominated} by $\v{x}_1\in\R^n$ if $\v{x}_1$ does not dominate $\v{x}_0$; i.e., $\v{f}(\v{x}_1)\not\leq\v{f}(\v{x}_0)$

        \item $\v{x}^*\in\R^n$ is a \emph{local Pareto optimal} for \eqref{eq:MOOprob} if there is $\varepsilon > 0$ such that $\v{x}^*$ is non-dominated by $\v{x}$, for each $\v{x}\in\R^n$, $\norm{\v{x}^*-\v{x}}_2\leq \varepsilon$.

        \item $\v{x}^*\in\R^n$ is a \emph{global Pareto optimal} for \eqref{eq:MOOprob} if $\v{x}^*$ is non-dominated by $\v{x}$, for each $\v{x}\in\R^n$.

        \item The set of all and only the global Pareto optimals is defined as \emph{Pareto set} of \eqref{eq:MOOprob}; its image through $\v{f}$ is defined as \emph{Pareto front} of \eqref{eq:MOOprob}.

        \item $\widehat{\v{x}}\in\R^n$ is defined as \emph{Pareto critical point} for \eqref{eq:MOOprob} if descent directions for all the objective functions $f_1,\ldots ,f_m$ evaluated in $\widehat{\v{x}}$ do not exist; i.e, if 
        \begin{equation}\label{eq:criticalPareto}
            J_{\v{f}}(\widehat{\v{x}}) \v{v} \not < \v{0}\,,
        \end{equation}
        for each $\v{v}\in\R^n$, where $J_{\v{f}}$ denotes the Jacobian of $\v{f}$.
        By consequence, $\widehat{\v{x}}\in\R^n$ is not a Pareto critical point for \eqref{eq:MOOprob} if a descent direction for all the objective functions $f_1,\ldots ,f_m$ evaluated in $\widehat{\v{x}}$ exists; i.e, if $\v{p}\in\R^n$ exists such that
        \begin{equation}\label{eq:not_criticalPareto}
            J_{\v{f}}(\widehat{\v{x}}) \v{p} < \v{0}\,.
        \end{equation}
    \end{enumerate}
\end{defin}

\begin{rem}[Necessary Condition for Local Pareto Optimals]\label{rem:criticalPareto_necessary_optimal}
    Being a Pareto critical is a necessary condition for being a local Pareto optimal. Specifically, if $\v{x}^*$ is a local Pareto optimal, then $\v{x}^*$ is a Pareto critical.
\end{rem}

After recalling the main entities related to MOO, we introduce the definition of \emph{non-ascent directions' regions} for the objective functions. This definition will be useful for studying and analyzing the properties of the proposed MGD method (see \Cref{sec:mymethod}).

\begin{defin}[Non-ascent Directions' Regions]\label{def:nonascent}
Let $f_1,\ldots f_m$ be the $m$ objective functions of \eqref{eq:MOOprob}. Let $\v{x}\in\R^n$ be fixed. Then, for each $i=1,\ldots ,m$, we define \emph{non-ascent directions' region} of $f_i$ evaluated at $\v{x}$ the set
\begin{equation}\label{eq:nonascent}
    P_i(\v{x}):=\{\v{p}\in\R^n\,|\, \nabla f_i(\v{x})^T\v{p}\leq 0\}\,;
\end{equation}
i.e., $P_i(\v{x})$ is the union of the descent directions of $f_i$ at $\v{x}$ and the directions perpendicular to $\nabla f_i(\v{x})$. Moreover, we define as \emph{region of shared non-ascent directions} the intersection of all the non-ascent directions' region, i.e., the set $P(\v{x}):=\bigcap_{i=1}^m P_i(\v{x})$ (it always contains the null vector).
\end{defin}

\subsection{Descent Methods for Multi-Objective Optimization}\label{sec:descentMOO}

Descent methods for MOO problems are iterative methods such that
\begin{equation}\label{eq:iterative_descent}
\begin{cases}
    \v{x}^{(0)}\in\R^n \ \text{given}\\
    \v{x}^{(k+1)} = \v{x}^{(k)} + \eta^{(k)} \v{p}^{(k)}\,, \ \forall \ k\geq 0
\end{cases}
\,,
\end{equation}

where $\v{p}^{(k)}\in\R^n$ is a descent direction for all the objective functions evaluated at $\v{x}^{(k)}$. 

In literature, one of the favorite approaches used for the computation of a descent direction $\v{p}^*$ for $m$ functions $f_1, \ldots ,f_m$ evaluated at a point $\v{x}$ is the one based on finding the minimum-norm vector in the anti-gradients' convex hull \cite{Schaffler2002, Desideri2009, Desideri2012_MGDA2, Desideri2012, Peitz2018}, i.e.:
\begin{equation}
    \v{p}^* = - \sum_{i=1}^{m} \alpha_i^* \nabla f_i(\v{x})\,,
\end{equation}
where
\begin{equation}\label{eq:antigradsConvHull}
    \v{\alpha}^* = \argmin_{\v{\alpha}}
    \left\lbrace \ \left\lVert \sum_{i=1}^{m} \alpha_i \nabla f_i(\v{x}) \right\rVert^2 \ \text{ s.t. } \ \alpha_1,\ldots ,\alpha_m \geq 0\,, \ \sum_{i=1}^m \alpha_i = 1 \right\rbrace \,,
\end{equation}

Then, the solution $\v{\alpha}^*$ of problem \eqref{eq:antigradsConvHull} can be found by solving a quadratic optimization problem. 

On the other hand, there are methods simpler than \eqref{eq:antigradsConvHull} for finding a descent direction $\v{p}^*$; for example, the method described in \cite{Fliege2000} and based on solving the Linear Programming (LP) problem
\begin{equation}\tag{$\mathrm{LP}_\mathrm{base}$}\label{eq:fliege2000}
    \begin{cases}
        \min_{\beta\in\R} \beta\\
        \nabla f_i(\v{x})^T \v{p} \leq \beta\,, \quad \forall i=1,\ldots ,m \\
        \norm{\v{p}}_\infty \leq 1
    \end{cases}
    \,,
\end{equation}
i.e., explicitly, the problem 
\begin{equation}\tag{$\mathrm{LP}'_\mathrm{base}$}\label{eq:fliege2000_explicit}
    \begin{cases}
        \min_
        {\v{\rho}\in\R^{n+1}} 
        [0,\ldots ,0, 1]
        %
        \ \v{\rho}
        \\
        \quad \\
        \left[
        \begin{array}{c|c}
            \nabla f_1(\v{x})^T & -1\\
            \vdots & \vdots\\
            \nabla f_m(\v{x})^T & -1
        \end{array} 
        \right]
        \v{\rho}
        \leq \v{0}
        \\
        \quad \\
        -1\leq \rho_i \leq 1\,, \ \forall i=1,\ldots ,n
    \end{cases}
    \,,
\end{equation}
with solution $\v{\rho}^*=(\v{p}^*,\beta^*)\in\R^{n+1}$ that is the concatenation of the shared descent direction found $\v{p^*}$ and the optimal value $\beta^*\leq 0$ found for $\beta$.

\section{A New Method Based on Linear Programming}\label{sec:mymethod}

In this paper, one of our focuses is the development of a new LP approach for finding a descent direction shared by multiple functions. In particular, we start from the LP method described in \cite{Fliege2000} (see \eqref{eq:fliege2000} and \eqref{eq:fliege2000_explicit}), and we modify it. More precisely, analogously to \cite{Fliege2000}, the LP we define finds descent directions if they exist; otherwise, non-ascent directions are returned (see \Cref{def:nonascent}). Nonetheless, these non-ascent directions are positively used by our MGD method, contrary to other methods in literature (see \Cref{sec:intro}); indeed, in this work, we also define a novel and different backtracking strategy for the optimization procedure \eqref{eq:iterative_descent} that is able to exploit them. In this section we focus on the new LP approach, postponing the definition of the new backtracking strategy in \Cref{sec:mgd}.

First of all, solving \eqref{eq:fliege2000}, we observe that we obtain a shared descent direction $\v{p}^*$ (assuming it exists) such that its scalar product with each gradient is less than the minimum value $\beta^*$ (always less than or equal to zero). Then, the gradient magnitudes have a relatively small influence on finding $\v{p}^*$, such as the gradient directions that only define the feasible region in the search space. In other words, the formulation of problem \eqref{eq:fliege2000} guarantees the finding of a shared descent direction $\v{p}^*$ but does not guarantee that moving along the direction of $\v{p}^*$ all the objective functions decrease coherently with their gradients.

In addition, problem \eqref{eq:fliege2000} has good probabilities of returning the null vector as the solution (stopping the optimization procedure) also in the special case of Pareto critical points characterized by the presence of at least one non-ascent direction that is a descent direction for at least one objective. Nonetheless, such a kind of non-null direction could generate a new step of \eqref{eq:iterative_descent} where $\v{x}^{(k+1)}$ is non-dominated by $\v{x}^{(k)}$; therefore, under these circumstances, it can be useful to add this option to a MGD optimization procedure, removing the null vector from the solution set of the LP problem.

Given the observations above, the purposes of the proposed new LP problem are the following:
\begin{enumerate}
    \item\label{item:pstar_gmagnitudes} Modify problem \eqref{eq:fliege2000} such that, moving along the direction of $\v{p}^*$, all the objective functions decrease as much coherently as possible with their gradient's characteristics;
    \item\label{item:allows_perpendicularity} Avoid returning the null vector as a solution for Pareto critical points, if there is at least one non-ascent direction that is a descent direction for at least one objective. We recall that we are interested in these directions too because our new backtracking strategy will exploit them (see \Cref{sec:mgd}).
\end{enumerate}

Before starting with the formulation of the new LP problem, we introduce some notations.
We denote by $\v{g}_i(\v{x})$ the gradient $\nabla f_i(\v{x})$, for each $i=1,\ldots, m$, and by $\v{g}(\v{x})$ their sum, i.e.:
\begin{equation}\label{eq:g}
    \v{g}(\v{x}) := \sum_{i=1}^m \v{g}_i(\v{x})\,.
\end{equation}
In addition, we denote by $\widebar{\v{g}}_i(\v{x})$ the gradient $\v{g}_i$ normalized with respect to the euclidean norm; i.e., 
\begin{equation}\label{eq:gbar_i}
    \widebar{\v{g}}_i(\v{x}) := \frac{\v{g}_i(\v{x})}{\norm{\v{g}_i(\v{x})}_2}\,,
\end{equation}
for each $i=1,\ldots ,m$. Then, we denote by $G(\v{x})$ and $\widebar{G}(\v{x})$ the matrices with rows defined by the gradients (i.e., the Jacobian of $\v{f}$, see \Cref{def:Pareto_definitions}) and the normalized gradients, respectively; specifically:
\begin{equation}\label{eq:G_Gbar}
    G(\v{x}):=
    \begin{bmatrix}
        \v{g}(\v{x})_1^T\\
        \vdots\\
        \v{g(\v{x})}_m^T\\
    \end{bmatrix}
    \in\R^{m\times n}
    \,,
    \quad
    \widebar{G}(\v{x}):=
    \begin{bmatrix}
        \widebar{\v{g}}(\v{x})_1^T\\
        \vdots\\
        \widebar{\v{g}}(\v{x})_m^T\\
    \end{bmatrix}
    \in\R^{m\times n}
    \,.
\end{equation}

Given \eqref{eq:G_Gbar}, the inequality constraints of \eqref{eq:fliege2000_explicit} defined using the gradients can be rewritten as
\begin{equation}\label{eq:fliege2000_grad_ineq}
    \left[
        \begin{array}{c|c}
            G(\v{x}) & -\v{e}
        \end{array} 
        \right]
        \v{\rho}
        \leq \v{0}
        \,,
\end{equation}
where $\v{e}:=(1,\ldots ,1)\in\R^m$.

\subsection{Linear Programming Problem Formalization}\label{sec:myLP_formalization}

For improving the influence of the gradient magnitudes according to the purpose illustrated in item \ref{item:pstar_gmagnitudes}, we need to involve the gradients both in the objective function and in the boundary values of the variables corresponding to the descent direction. Specifically, we change the objective function of the problem \eqref{eq:fliege2000} into 
\begin{equation}\label{eq:obj_LP_simple}
\v{g}(\v{x})^T\v{p} + c_\beta(\v{x}) \beta
\end{equation}
and the boundary values into
\begin{equation}\label{eq:bounds_LP_simple}
    \norm{\v{p}}_\infty \leq \gamma(\v{x}) \,, \quad \beta\leq 0\,,
\end{equation}
where $c_\beta(\v{x}) > 0$ (in practice, $c_\beta(\v{x}) > \norm{\v{g}(\v{x})}_2$, see \Cref{sec:theory} later) and
\begin{equation}\label{eq:gamma_LP}
    \gamma(\v{x}) := \max\{\norm{\v{g}_1(\v{x})}_{\infty},\ldots ,\norm{\v{g}_m(\v{x})}_{\infty}, \norm{\v{g}(\v{x})}_{\infty}\}\,.
\end{equation}
Clearly, for each $\v{p}$ satisfying \eqref{eq:fliege2000_grad_ineq} and $\beta\leq 0$, it holds that $\v{g}(\v{x})^T\v{p} + c_\beta(\v{x}) \beta\leq m\beta + c_{\beta}(\v{x})\beta\leq 0$.

Using the same notation of problem \eqref{eq:fliege2000_explicit}, the objective function is now
\begin{equation}\label{eq:obj_LP}
    [ \ \v{g}(\v{x})^T \ | \ c_\beta(\v{x}) \ ]\, \v{\rho}\,,
\end{equation}
while the boundary values are
\begin{equation}\label{eq:boundaries_LP}
    -\gamma(\v{\v{x}}) \leq \rho_i \leq \gamma(\v{x})\,, \ \forall i=1,\ldots ,n \quad \text{ and } \quad \rho_{n+1}\leq 0\,.
\end{equation}

In \Cref{fig:comparison_objf_boundaries} we illustrate the proposed objective function \eqref{eq:obj_LP} with respect to the boundaries \eqref{eq:boundaries_LP}, in comparison with the corresponding objective function and boundaries of problem \eqref{eq:fliege2000}. 

\begin{figure}[htb!]
    \centering
    \subcaptionbox{Obj. function $\beta$ s.t. $\norm{\v{p}}_\infty\leq 1$.}{\includegraphics[trim={2cm 1cm 2cm 2cm},clip,width=0.49\textwidth]{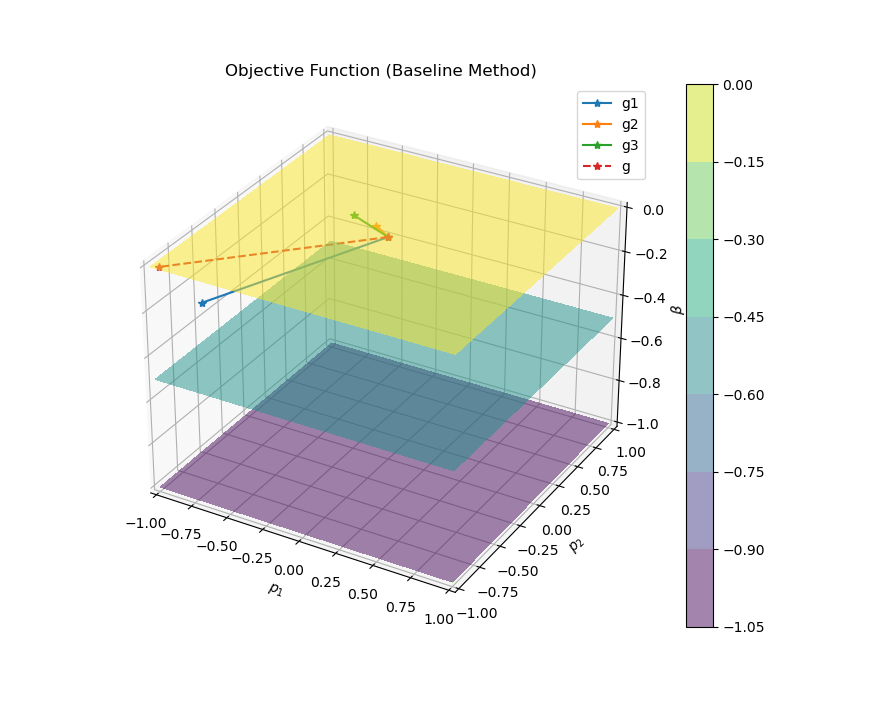}}
    \subcaptionbox{Obj. function $\v{g}(\v{x})^T\v{p}+c_\beta(\v{x}) \beta$ s.t. $\norm{\v{p}}_\infty\leq \gamma(\v{x})$.}{\includegraphics[trim={2cm 1cm 2cm 2cm},clip,width=0.49\textwidth]{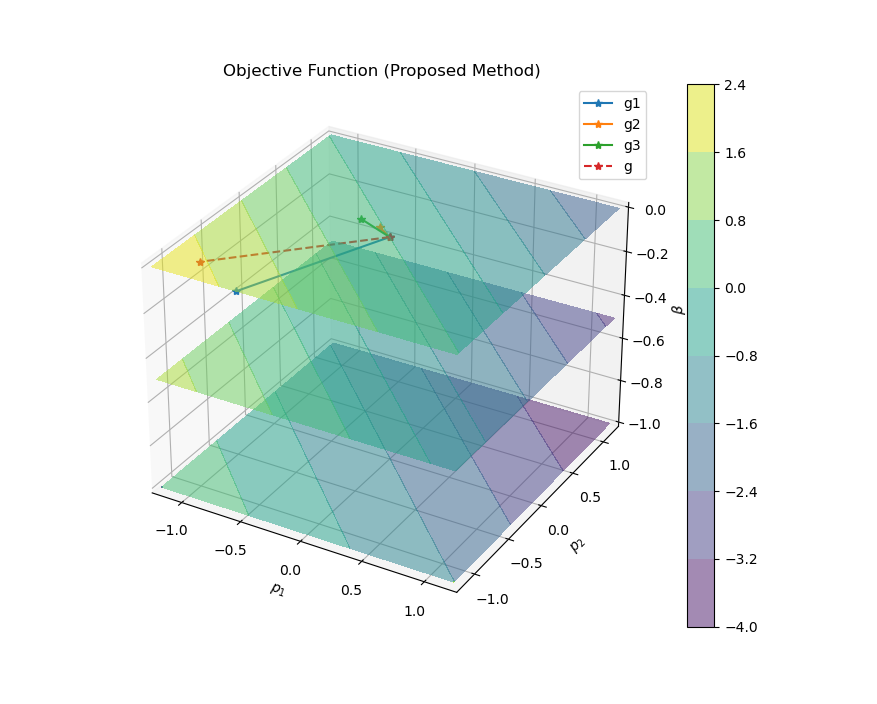}}
    \caption{Visual comparison of the objective function of problem \eqref{eq:fliege2000} (sub-figure a) and the proposed objective function \eqref{eq:obj_LP} (sub-figure b). The figures illustrate three function evaluations in $\R^3$ for three fixed values of $\beta=-1, -0.5, 0$, three randomly generated gradients $\v{g}_1(\v{x}),\v{g}_2(\v{x}),\v{g}_3(\v{x})\in\R^2$, and (for sub-figure b) $c_\beta(\v{x})=\norm{\v{g}(\v{x})}_2+0.25$.}
    \label{fig:comparison_objf_boundaries}
\end{figure}

The change of objective function is proposed because it forces the solution to find a descent direction that minimizes $\beta$ but follows as much as possible the direction identified by the sum of the anti-gradients. On the other hand, the boundary values of the descent directions have been changed in order to maintain at least a weak connection with the norms of the function gradients for a better embedding with an iterative method like \eqref{eq:iterative_descent}; this connection is not preserved if we use the condition $\norm{\v{p}}_\infty\leq 1$ like in \eqref{eq:fliege2000}. 

After the modification of the objective function and the boundary values, we observe that also the inequality constraints \eqref{eq:fliege2000_grad_ineq} can be modified to improve the quality of the descent direction $\v{p}^*$. 

Ideally, the more parallel a descent direction is to the anti-gradient, the better; then, in the region of the descent directions, the more distant it is from the hyperplane perpendicular to the gradient, the better. Actually, the inequality constraints \eqref{eq:fliege2000_grad_ineq} are implicitly trying to apply the latter reasoning because the distances of $\v{p}$ from each one of the hyperplanes perpendicular to the $m$ gradients are lower bounded by a quantity dependent on $\beta$. Specifically, denoted by $H_i(\v{x})$ the hyperplane perpendicular to $\v{g}_i(\v{x})$, if $\v{\rho}=(\v{p},\beta)$ satisfies \eqref{eq:fliege2000_grad_ineq}, then the distance of $\v{p}$ from $H_i(\v{x})$ is lower bounded by $|\beta|/\norm{\v{g}_i(\v{x})}_2$, i.e.:
\begin{equation}\label{eq:fliege2000_distbounds}
    {\rm dist}\left(H_i(\v{x}), \v{p}\right) = \frac{\left| \v{g}_i(\v{x})^T\v{p} \right|}{\norm{\v{g}_i(\v{x})}_2} \geq \frac{|\beta|}{\norm{\v{g}_i(\v{x})}_2}\,, \quad \forall \ i=1,\ldots ,m\,.
\end{equation}

Therefore, looking at \eqref{eq:fliege2000_distbounds}, we observe that a solution $\v{p}^*$ can result to be nearly-perpendicular to those gradients characterized by a large norm, because the distance's lower bound tends toward zero. This phenomenon can be a problem for an iterative method like \eqref{eq:iterative_descent}, because the step can be almost-perpendicular to the steepest descent direction of one of the steepest objective functions of the problem (see the lightly obscured region varying with $\beta$ and the vectors $\v{\rho}^*, \v{p}^*$ in \Cref{fig:comparison_solutions}a).

Then, we modify the inequality constraints \eqref{eq:fliege2000_grad_ineq} using the matrix $\widebar{G}(\v{x})$ instead of $G(\v{x})$; i.e., we use the inequality constraints
\begin{equation}\label{eq:my_grad_ineq}
    \left[
        \begin{array}{c|c}
            \widebar{G}(\v{x}) & -\v{e}
        \end{array} 
        \right]
        \v{\rho}
        \leq \v{0}
        \,.
\end{equation}

For a point $\v{x}$ that is not Pareto critical, if \eqref{eq:my_grad_ineq} is used as inequality constraint, we can prove (see \Cref{lem:my_LP}) that the minimum distance of a solution $\v{p}^*$ from all the hyperplanes $H_i(\v{x})$ is at least $|\beta^*|$ (see the lightly obscured region varying with $\beta$ and the vectors $\v{\rho}^*, \v{p}^*$ in \Cref{fig:comparison_solutions}b). 

In conclusion, applying to \eqref{eq:fliege2000_explicit} the changes listed above, we obtain the following new LP problem:
\begin{equation}\tag{LP$_\mathrm{new}$}\label{eq:my_LP_explicit}
    \begin{cases}
        \min_{\v{\rho}\in\R^{n+1}} [ \ \v{g}(\v{x})^T \ | \ c_\beta(\v{x}) \ ]\, \v{\rho}\\
        %
        %
        \left[
        \begin{array}{c|c}
            \widebar{G}(\v{x}) & -\v{e}
        \end{array} 
        \right]
        \v{\rho}
        \leq \v{0}\\
        %
        %
        -\gamma(\v{x}) \leq \rho_i \leq \gamma(\v{x})\,, \ \forall i=1,\ldots ,n\\
        \rho_{n+1}\leq 0
    \end{cases}
    \,.
\end{equation}

In \Cref{fig:comparison_solutions} we illustrate an example of different solutions obtained by applying \eqref{eq:fliege2000_explicit} and \eqref{eq:my_LP_explicit} to the same set of gradients.

\begin{figure}[htb!]
    \centering
    \subcaptionbox{Problem \eqref{eq:fliege2000}.}{\includegraphics[trim={2cm 1cm 2cm 2cm},clip,width=0.49\textwidth]{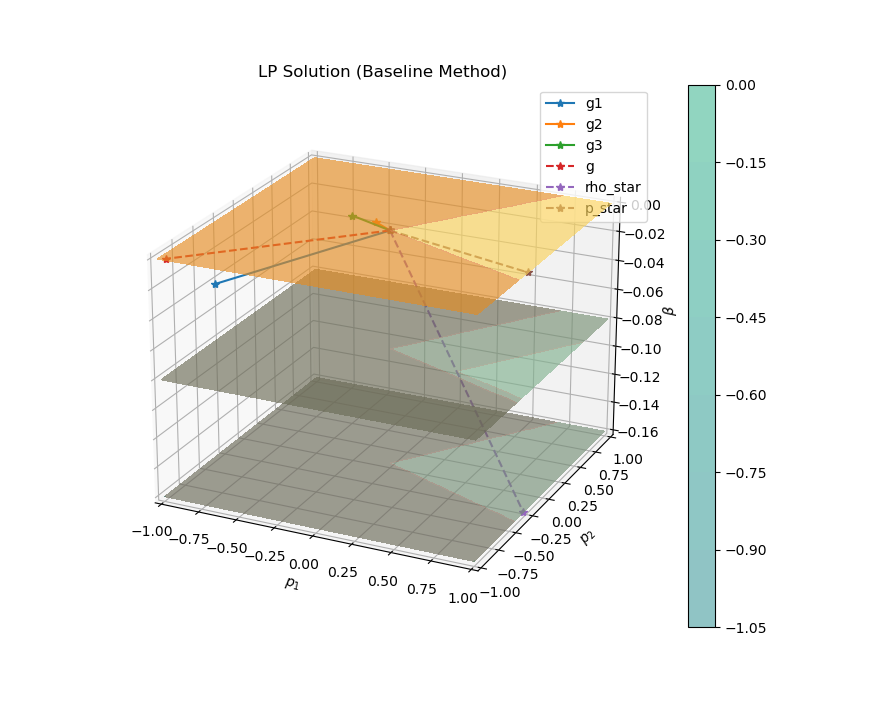}}
    \subcaptionbox{Problem \eqref{eq:my_LP_explicit}.}{\includegraphics[trim={2cm 1cm 2cm 2cm},clip,width=0.49\textwidth]{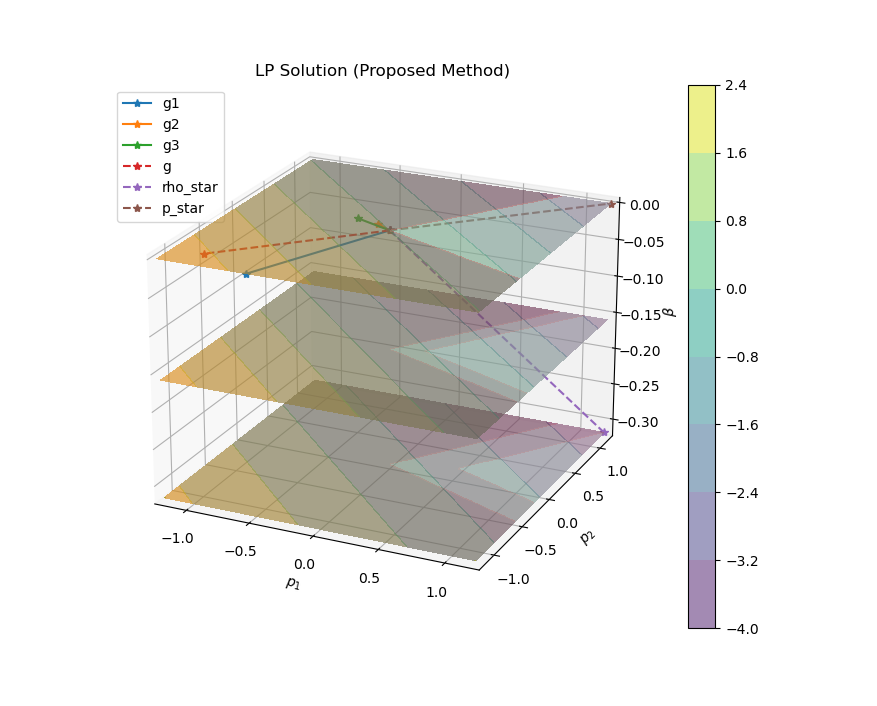}}
    \caption{Visual comparison of the solutions $\v{\rho}^*$ and $\v{p}^*$ obtained from problem \eqref{eq:fliege2000} (sub-figure a) and problem \eqref{eq:my_LP_explicit} (sub-figure b), respectively, given the same gradients $\v{g}_1(\v{x}),\v{g}_2(\v{x}),\v{g}_3(\v{x})\in\R^2$, and (for sub-figure b) $c_\beta(\v{x})=\norm{\v{g}(\v{x})}_2+0.25$. The figures illustrate three objective function evaluations in $\R^3$ for three fixed values of $\beta=\beta^*, \beta^*/2, 0$. The non-feasible region of each problem is obscured and changes with $\beta$ (heavily obscured for the region of non-descent directions).}
    \label{fig:comparison_solutions}
\end{figure}

\subsection{Theoretical Results and Characterization}\label{sec:theory}

Analogously to Lemma 3 in \cite{Fliege2000} for problem \eqref{eq:fliege2000}, we illustrate some theoretical results for the LP problem \eqref{eq:my_LP_explicit}.
We report the main results of the Lemma for problem \eqref{eq:fliege2000} in the following, adapted to the notation used in this work..

\begin{lem}[Lemma 3 - \cite{Fliege2000}]\label{lem:fliege2000}
    Let $V^*, \beta^*$ be the solution set and the optimum value of problem \eqref{eq:fliege2000}, respectively. Then:
    \begin{enumerate}
        \item if $\v{x}$ is Pareto critical for the functions $f_1,\ldots ,f_m$, then $\v{0}\in V^*$ and $\beta^*=0$;
        \item if $\v{x}$ is not Pareto critical for the functions $f_1,\ldots ,f_m$, then $\beta^* < 0$ and $\v{p}^*$ is a descent direction for all the functions $f_1,\ldots ,f_m$, for any $\v{p}^*\in V^*$;
   \end{enumerate}
\end{lem}

For the theoretical analysis of our problem, we start with a proposition that characterizes the influence of the parameter $c_\beta$ on the objective function evaluation. In the next results, we will show how this proposition is necessary to guarantee $\beta^*<0$ for non-null solutions of \eqref{eq:my_LP_explicit}. 

\begin{notaz}
    For ease of notation, from now on we will drop the dependency on $\v{x}$ in the LP problems.
\end{notaz}

\begin{prop}\label{prop:cbeta_lbound_0}
    Let $\v{\rho}_0=(\v{p}_0, \beta_0)\in\R^{n+1}$. Let $\v{c}$ denotes the vector characterizing the objective function \eqref{eq:obj_LP}; i.e., $\v{c}^T:=[\v{g}^T|c_\beta]$. Then, if $c_\beta >\norm{\v{g}}_2$, we have that 
        \begin{equation}\label{eq:ineq_cbeta_0}
            \v{c}^T \v{\rho}_0 - \v{c}^T \v{\rho}_1 > 0
        \end{equation}
    for each $\v{\rho}_1=(\v{p}_1,\beta_1)$ such that $\beta_1<\beta_0$  and ${\rm dist}(\v{p}_0,\v{p}_1)\geq (\beta_0 - \beta_1)$.

    Specifically, if $c_\beta=\norm{\v{g}}_2+\varepsilon$, $\varepsilon >0$, we have that
    \begin{equation}\label{eq:ineq_cbeta_1}
            \v{c}^T \v{\rho}_0 - \v{c}^T \v{\rho}_1 \geq (\beta_0 - \beta_1)\varepsilon > 0\,.
    \end{equation}
\end{prop}
\begin{proof}
    Let $\alpha$ be the angle between $\v{g}$ and $(\v{p}_0 -\v{p}_1)$. Let $\Delta\beta$ be the difference between $\beta_0$ and $\beta_1$; i.e., $\Delta\beta:=\beta_0-\beta_1>0$. Then,
    \begin{equation}
    \begin{aligned}
        \v{c}^T \v{\rho}_0 - \v{c}^T \v{\rho}_1 &= \v{g}^T(\v{p}_0-\v{p}_1) + c_\beta \Delta\beta =\\
        &= \norm{\v{g}}_2\norm{\v{p}_0-\v{p}_1}_2 \cos\alpha + c_\beta \Delta\beta \geq\\
        &\geq \Delta\beta (\norm{\v{g}}_2\cos\alpha + c_\beta)\geq\\
        &\geq \Delta\beta (-\norm{\v{g}}_2 + c_\beta)\,.
    \end{aligned}
    \end{equation}
    Then, \eqref{eq:ineq_cbeta_0} holds if $c_\beta>\norm{\v{g}}_2$ and \eqref{eq:ineq_cbeta_1} holds if $c_\beta = \norm{\v{g}}_2 + \varepsilon$, with $\varepsilon>0$.
\end{proof}

In the following Lemma, we characterize the solutions obtained solving \eqref{eq:my_LP_explicit}, depending on the Pareto nature of $\v{x}\in\R^n$; i.e., depending on $\v{x}$ that is a Pareto critical point or not, following the lines of \Cref{lem:fliege2000}.

\begin{lem}\label{lem:my_LP}
    Let $\v{\rho}^*=(\v{p}^*, \beta^*)\in\R^{n+1}$ be a solution of problem \eqref{eq:my_LP_explicit}, with $c_\beta >\norm{\v{g}}_2$. Then:
    \begin{enumerate}        
        \item If $\v{x}$ is Pareto critical for the functions $f_1,\ldots ,f_m$, then $\beta^*=0$.

        \item Let $\v{x}$ be Pareto critical for the functions $f_1,\ldots ,f_m$. Let $V^*$ be the solution set, let $P_i$ be the set of non-ascent directions of $f_i$ in $\v{x}$ and let $P$ be the intersection of all the sets $P_i$ (see \Cref{def:nonascent}). Let us denote by $\Pi$ the set of all the feasible directions in $P$, according to the constraints of the problem. Then, it holds:
        \begin{enumerate}
            \item $V^*=\Pi$ if and only if $\v{g}^T\v{p} = 0$ for each $\v{p}\in\Pi$
            \item $V^*=\{\v{0}\}$ if and only if $\Pi = \{\v{0}\}$;
            \item $\v{0}\notin V^*$ if and only if there is $\v{p}\in\Pi$ such that $\v{g}^T\v{p}\neq 0$;
        \end{enumerate}
        
        \item If $\v{x}$ is not Pareto critical for the functions $f_1,\ldots ,f_m$, then $\beta^* < 0$ and $\v{p}^*$ is a descent direction for all the functions $f_1,\ldots ,f_m$ in $\v{x}$.

        \item If $\v{x}$ is not Pareto critical for the functions $f_1,\ldots ,f_m$, then:
        \begin{enumerate}
            \item $\beta^*=\max_{i=1,\ldots ,m}\widebar{\v{g}}_i^T\v{p}^*$;
            \item For each $i = 1,\ldots ,m$, it holds
            \begin{equation}\label{eq:pstar_mindist_hyperplanes}
                {\rm dist}\left( H_i, \v{p}^*\right) \geq |\beta^*|\,;
            \end{equation} 
        \end{enumerate}      
    \end{enumerate}
\end{lem}

\begin{proof}
    In the following, we prove one by one the results listed in the Lemma.
    
    \begin{enumerate}
        \item The proof is almost straightforward. If $\v{x}$ is Pareto critical, no negative values of $\beta$ permit to satisfy the inequality constraints (see \eqref{eq:criticalPareto}); then, $\beta^*=0$. 

        \item If $\v{x}$ is Pareto critical, no descent directions are shared by the $m$ objective functions. Then, solutions are contained in the non-ascent directions' region; i.e., $V^*\subseteq \Pi$. 
        
        Let us start with item 2.1 and let us assume that $\v{g}^T \v{p} = 0$, for each $\v{p}\in \Pi$. Since $V^*\subseteq \Pi$, we have that $\v{g}^T \v{p}^*=0$ for each solution of the problem. But any vector in $\Pi$ satisfies the constraints and has the same minimum value for the objective function (i.e., value $0$). Then, $\Pi \subseteq V^*$ and we prove the thesis. 
        
        Concerning the proof of item 2.2, first of all, we observe that for item 2.1., $V^*=\{\v{0}\}$ if $\Pi=\{\v{0}\}$. On the other hand, we prove by contradiction that $\Pi=\{\v{0}\}$ if $V^*=\{\v{0}\}$. Let us assume that $\Pi\neq\{\v{0}\}$ and $V^*=\{\v{0}\}$; then, there is a feasible direction $\v{p}\in\Pi$, $\v{p}\neq \v{0}$, such that $\v{g}^T\v{p}\leq 0 = \v{g}^T\v{0}$. Therefore, there are two possible situations: $i$) $\v{g}^T\v{p}< 0$ and $\v{0}$ is not a solution; $ii$)  $\v{g}^T\v{p}= 0$ and $\v{0}$ is not the unique element of $V^*$. Both cases are a contradiction of the hypothesis $V^*=\{\v{0}\}$, proving the thesis.

        Moving to item 2.3, also in this case we use the proof by contradiction. Let the null vector be a solution and let $\v{p}\in \Pi$ be such that $\v{g}^T\v{p}\neq 0$. Then, $\v{g}_i^T\v{p}\leq 0$, for each $i=1,\ldots, m$, and there is $j\in\{1,\ldots ,m\}$ such that $\v{g}_j^T\v{p}< 0$. By consequence, $\v{g}^T\v{p}< 0 = \v{g}^T\v{0}$; therefore, $\v{0}$ is not a solution of problem \eqref{eq:my_LP_explicit}, which is a contradiction; therefore $\v{0}$ is not a solution if there is $\v{p}\in\Pi$ such that $\v{g}^T\v{p}\neq 0$. On the other hand, again by contradiction, let us assume that $\v{0}\notin V^*$ and that there is no $\v{p}\in\Pi$ such that $\v{g}^T\v{p}\neq 0$; therefore, it holds $\v{g}^T\v{p}=0$ for each $\v{p}\in\Pi$ and, for item 2.1, $\Pi=V^*\ni\v{0}$, that is a contradiction. Therefore, there is at least one $\v{p}\in\Pi$ such that $\v{g}^T\v{p}\neq 0$ if $\v{0}\notin V^*$.
        
        \item The proof of this item is done by contradiction. Let us assume that \eqref{eq:my_LP_explicit} has solution $\v{\rho}^*=(\v{p}^*, \beta^*)$ with $\beta^*=0$. Since $\v{x}$ is not Pareto critical, there is $\widehat{\v{p}}\in\R^n$ such that $\norm{\widehat{\v{p}}}_\infty \leq \gamma$, $\widehat{\v{p}}\neq\v{p}^*$, and $\widebar{G}\widehat{\v{p}} < 0$.
        Let $\delta$ denotes the euclidean distance between $\v{p}^*$ and $\widehat{\v{p}}$ and let $\widehat{\beta}$ be such that
        \begin{equation}\label{eq:betahat_lem1}
            \widehat{\beta}:= \max \left\lbrace -\delta, \max_{i=1,\ldots ,m}\widebar{\v{g}}_i^T\widehat{\v{p}} \right\rbrace < 0\,.
        \end{equation}
        Then, $\widehat{\v{\rho}}:=(\widehat{\v{p}}, \widehat{\beta})$ is a feasible point of the problem, $\widehat{\beta}<\beta^*=0$, ${\rm dist}(\v{p}^*,\widehat{\v{p}})=\delta\geq (\beta^* - \widehat{\beta})$, and $c_\beta > \norm{\v{g}}_2$. Therefore, for \Cref{prop:cbeta_lbound_0}, it holds
        \begin{equation}\label{eq:ineq_crho}
            \v{c}^T \v{\rho}^* > \v{c}^T \widehat{\v{\rho}}\,,
        \end{equation}
        which is a contradiction of the hypothesis that $\v{\rho}^*$ is solution of \eqref{eq:my_LP_explicit}.

        \item Let $\v{\rho}^*$ be a solution for the problem \eqref{eq:my_LP_explicit}. Then, $\beta^*<0$ and $\v{\rho^*}$ satisfies \eqref{eq:my_grad_ineq}. It is easy to prove by contradiction that $\beta^*=\max_{i=1,\ldots ,m}\widebar{\v{g}}_i^T\v{p}^*=:\delta_{\max}^*$ (item 4.1); Indeed, assuming $\beta^*>\delta_{\max}^*$, it is possible to find a feasible vector $(\v{p}^*, \delta_{\max}^*)$ with lower value of the objective function.

        Concerning item 4.2, as previously observed, the distance of $\v{p}^*$ from $H_i$ is lower bounded by $|\beta^*|$, for each $i=1,\ldots ,m$, because $\v{\rho}^*$ satisfies \eqref{eq:my_grad_ineq}; then, \eqref{eq:pstar_mindist_hyperplanes} holds. 
        
    \end{enumerate}

\end{proof}


Now we point the attention of the readers to item 2 of \Cref{lem:my_LP}. The three sub-items correspond to the three possible situations that occur for a Pareto critical point $\v{x}\in\R^n$; specifically, we have:
\begin{itemize}
    \item \Cref{lem:my_LP} - item 2.2 ($V^*=\Pi=\{\v{0}\}$, see \Cref{fig:comparison_lemmacases}a): the critical point is such that for each direction $\v{p}\in\R^n$, $\v{p}$ is an ascent direction for at least one objective function; i.e., there is $j\in\{1,\ldots ,m\}$ such that $\v{g}_j^T\v{p} > 0$. This is a particular case of item 2.1 of the Lemma (see below).
    
    \item \Cref{lem:my_LP} - item 2.1 ($\v{g}^T\v{p}=0$ for each $\v{p}\in\Pi$, $V^*=\Pi$, see \Cref{fig:comparison_lemmacases}b): the critical point is such that: $i$) all the non-ascent directions shared by all the objective functions are perpendicular to all the gradients; $ii$) the case 2.2 of the Lemma (see above).
    
    \item \Cref{lem:my_LP} - item 2.3 (there is $\v{p}\in\Pi$ such that $\v{g}^T\v{p}\neq 0$, $\v{0}\notin V^*$, see \Cref{fig:comparison_lemmacases}c): the critical point is such there is $j\in\{1,\ldots ,m\}$ and there is a non-ascent direction shared by all the objectives such that it is a descent direction for the objective function $f_j$.   
\end{itemize}

\begin{figure}[htb!]
    \centering
    \subcaptionbox{$V^*=\Pi=\{\v{0}\}$.}{\includegraphics[trim={1cm 1cm 1cm 1cm},clip,width=0.30\textwidth]{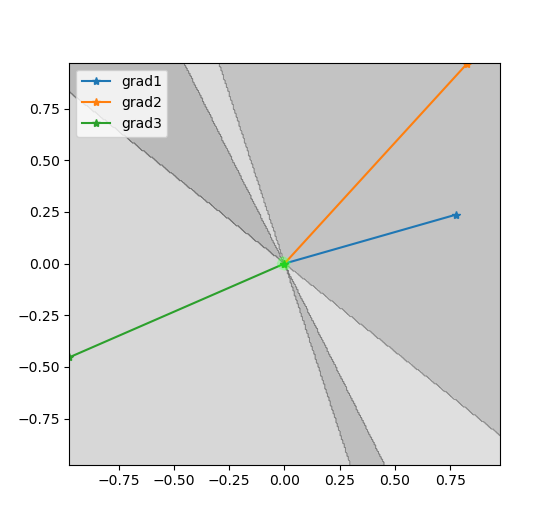}}
    \subcaptionbox{$V^*=\Pi$.}{\includegraphics[trim={1cm 1cm 1cm 1cm},clip,width=0.30\textwidth]{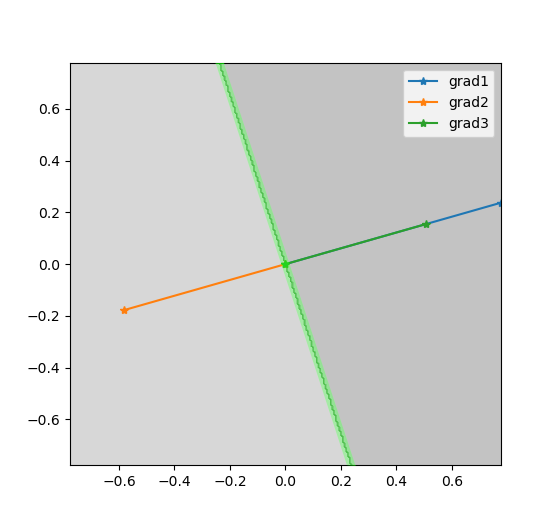}}
    \subcaptionbox{$\v{0}\notin V^*$.}{\includegraphics[trim={1cm 1cm 1cm 1cm},clip,width=0.30\textwidth]{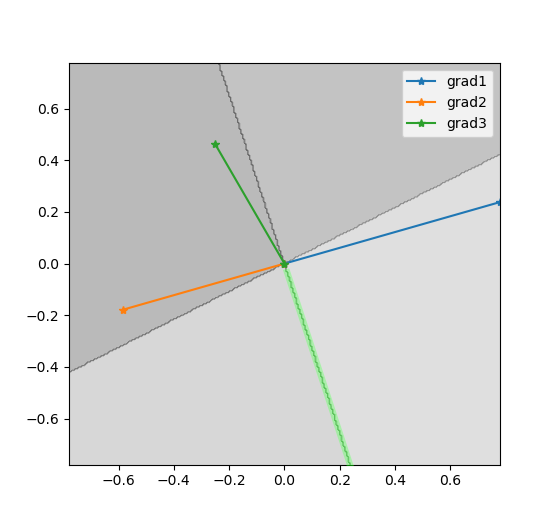}}
    \caption{Visual example in $\R^2$, $m=3$ objectives, of the three possible situations that occur for a Pareto critical point in \eqref{eq:my_LP_explicit}. Grey color denotes ascent regions for the objective functions (there are no shared descent directions because $\v{x}$ is Pareto critical); the darker the grey color, the more its objectives consider it as an ascent direction. Dark grey lines denote the regions of directions perpendicular to the gradients. We highlight in green color the solution set $V^*$ of the LP problem.}
    \label{fig:comparison_lemmacases}
\end{figure}

One of the main differences between \eqref{eq:fliege2000} and \eqref{eq:my_LP_explicit} is that for \eqref{eq:fliege2000} we have that $\v{0}\in V^*$ for all the three situations listed above and illustrated in \Cref{fig:comparison_lemmacases} (see \Cref{lem:fliege2000}). Therefore, \eqref{eq:fliege2000} can return anytime a null direction $\v{p}^*=\v{0}$ if $\v{x}$ is Pareto critical, stopping the optimization procedure as a consequence. On the contrary, the new LP problem \eqref{eq:my_LP_explicit} returns a null direction only if there are no directions in $\Pi$ that are descent directions for at least one objective; therefore, if there is a direction in $\Pi$ that is a descent direction for at least one objective, \eqref{eq:my_LP_explicit} returns a non-null solution $\v{p}^*$ and the optimization procedure \eqref{eq:iterative_descent} can continue looking for point $\v{x}^{(k+1)}$ along the direction $\v{p}^*$ such that $\v{x}^{(k+1)}$ is non-dominated by $\v{x}^{(k)}=\v{x}$. Moreover, it is possible to find $\v{x}^{(k+1)}$ with decreased values for the objective functions $f_{j_1},\ldots f_{j_M}$ for which it holds $\v{g}_j^T\v{p}^*<0$, $j=j_1,\ldots ,j_M$. In this situation, in order to find such a kind of new point for the sequence, we define a new backtracking strategy in the next section. Coupling \eqref{eq:my_LP_explicit} with this new backtracking strategy, we are able to increase the possibility of detecting and/or exploring the Pareto set of a MOO problem, as illustrated by the results of the numerical experiments of \Cref{sec:experiments}.

\begin{rem}[Case of two objective functions]\label{rem:case2objs}
    We observe that if $m=2$, i.e., the MOO problem is characterized by two objective functions only, the case $\v{0}\notin V^*$ is not possible for $\v{x}\in\R^n$ Pareto critical. Indeed, only the cases corresponding to items 2.1 and 2.2 of \Cref{lem:my_LP} are possible. Therefore, when $m=2$, the differences between \eqref{eq:fliege2000} and \eqref{eq:my_LP_explicit} are in the orientation and norm of the non-null solutions $\v{p}^*$.
\end{rem}

\section{New Backtracking Strategy and Convergence Analysis}\label{sec:mgd}

In this section, we describe how to better exploit a MGD method \eqref{eq:iterative_descent} that is based on descent directions evaluated with \eqref{eq:my_LP_explicit}. In particular, we use a new line-search method for the step lengths $\{\eta_k\}_{k\in\N}$, in order to guarantee that the sequence $\{\v{x}^{(k)}\}_{k\in\N}$ is such that $\v{x}^{(k+1)}$ is non-dominated by $\v{x}^{(k)}$, for each $k\in\N$. Therefore, we observe that the main difference between our line-search method and the ones typically adopted in literature (e.g., see \cite{Fliege2000,Desideri2009,Desideri2012,Desideri2012_MGDA2,Peitz2018}) is in the nature of the sequence $\{\v{x}^{(k)}\}_{k\in\N}$: all the other methods look for a sequence strictly decreasing for all the objective functions; on the contrary, we relax this requirement, accepting in general vectors $\v{x}^{(k+1)}$ that are non-dominated by $\v{x}^{(k)}$ (but giving priority to the ones that decrease the values of all the objectives).

One of the simplest but most effective implementations of line-search methods for one-objective optimization problems is the backtracking strategy with respect to the Armijo condition \cite{Armijo1966,Wolfe1969,Wolfe1971,Nocedal2012}. This condition can be easily extended to MOO for building the sequence $\{\v{x}^{(k)}\}_{k\in\N}$ strictly decreasing with respect to all the objective functions \cite{Fliege2000,Desideri2009,Desideri2012,Desideri2012_MGDA2,Peitz2018}. Then, in this work, we extend even further this approach for modifying the properties of the sequence $\{\v{x}^{(k)}\}_{k\in\N}$ as stated above.

Let $\alpha\in (0,1)$ be the factor of the backtracking strategy and let 
$c_1\in (0,1)$ be the parameter of the Armijo condition. Then, at each step $k>0$, we check if the Armijo condition 
\begin{equation}\label{eq:armijo_moo_0}
    f_i(\v{x}^{(k)} + \eta^{(k)}_0\v{p}^{*(k)}) \leq f_i(\v{x}^{(k)}) + c_1 \eta^{(k)}_0 \v{g}_i(\v{x}^{(k)})^T\v{p}^{*(k)}\,,
\end{equation}
is satisfied for the base-value of the step-size $\eta^{(k)}=\eta^{(k)}_0$ and for each $i=1,\ldots ,m$. If it is not satisfied, we decrease the step-size by the factor $\alpha$ until the condition 
\begin{equation}\tag{$\mathrm{A}_i^t$}\label{eq:armijo_moo_t}
    f_i(\v{x}^{(k)} + \eta^{(k)}_t\v{p}^{*(k)}) \leq f_i(\v{x}^{(k)}) + c_1 \eta^{(k)}_t \v{g}_i(\v{x}^{(k)})^T\v{p}^{*(k)}\,,
\end{equation}
is satisfied for the step-size $\eta^{(k)}_t := \eta^{(k)}_{t-1}\alpha = \eta_0^{(k)}\alpha^t$, $t>0$
, and for each $i=1,\ldots ,m$. 

However, it is possible that $\eta^{(k)}_t$ does not satisfy \eqref{eq:armijo_moo_t}, for each $t\geq 0$; in particular, it happens when $\v{x}^{(k)}$ is Pareto critical and, therefore, $\v{p}^{*\,(k)}$ is not a descent direction for all the objectives, but only a non-ascent direction or the null vector (see \Cref{lem:my_LP}). Under these circumstances, instead of stopping the sequence like the other MOO methods (e.g., see \cite{Fliege2000}), if $\v{p}^{*(k)}\neq\v{0}$ we make a last check:
\begin{equation}\label{eq:nodom_check_moo_bcktrck}
    \exists \ i\in\{1\, ,\ldots \, ,m\} \quad s.t. \quad f_i(\v{x}^{(k)}) > f_i(\v{x}^{(k)} + \widehat{\eta}\,\v{p}^{*(k)})
\end{equation}
where $\widehat{\eta}>0$ is a fixed arbitrary small parameter. Therefore, if \eqref{eq:nodom_check_moo_bcktrck} is satisfied, setting $\v{x}^{(k+1)}:=\v{x}^{(k)} + \widehat{\eta}\, \v{p}^{*(k)}$ returns a new point $\v{x}^{(k+1)}$ for the sequence that is non-dominated by $\v{x}^{(k)}$, and the sequence can continue.

Then, we can summarize the sequence generated by the line search method described above as
\begin{equation}\tag{$\mathrm{BT}_{\mathrm{new}}$}\label{eq:seq_bcktrck_theory}
    \begin{cases}
        \v{x}^{(0)}\in\R^n\,, \ \{\eta_0^{(k)}\}_{k\in\N}\subset\R^+\,, \ \widehat{\eta}>0 & \text{ given}\\
        \v{x}^{(k+1)}=\v{x}^{(k)}+\eta^{(k)}\v{p}^{*(k)}\,, \quad &k\geq 0
    \end{cases}
    \,,
\end{equation}
where 
\begin{equation}\label{eq:etak_bcktrck_theory}
    \eta^{(k)}=
    \begin{cases}
        \eta^{(k)}_{\tau} = \eta_0^{(k)}\alpha^\tau \quad &\text{if }\v{x}^{(k)}\text{ is not Par. crit. ($\tau$, $1^{st}$ value of $t\geq 0$ satisfying \eqref{eq:armijo_moo_t})}\\
        \widehat{\eta} \quad &\text{if }\v{x}^{(k)}\text{ is Par. crit. and does not dominate } \v{x}^{(k)}+\widehat{\eta} \ \v{p}^{*(k)}\\
        0 \quad & \text{otherwise (i.e., $\v{x}^{(k)}$ is Par. crit. and dominates $\v{x}^{(k)}+\widehat{\eta} \ \v{p}^{*(k)}$)}
    \end{cases}
    \,.
\end{equation}

The new backtracking strategy can be used in general with any MGD method, independently of the type of sub-problems used to compute the directions $\v{p}^{*\, (k)}$; therefore, it can be used also with a MGD method based on \eqref{eq:fliege2000}. Nonetheless, we recall that \eqref{eq:seq_bcktrck_theory} has been specifically designed to take advantage of the properties of \eqref{eq:my_LP_explicit}; in particular, we recall the observations related to item 2.3 of \Cref{lem:my_LP}.

\begin{rem}[Strictly decreasing backtracking as a special case of the new one]\label{rem:nondom_restart_and_speccase_old}
    Line search methods designed for building strictly decreasing sequences for all the objectives can be interpreted as a special case of the new type of backtracking strategies, where $\widehat{\eta}=0$.
\end{rem}

\subsection{Convergence Results}\label{sec:convergence}

If we look at all the possible sequences $\{\v{x}^{(k)}\}_{k\in\N}$ determined by the new backtracking strategy \eqref{eq:seq_bcktrck_theory}, we can identify three main types of sub-sequences of consecutive elements:
\begin{enumerate}
    \item $\{\v{x}^{(k+m)}\}_{m=0}^{M}$, $k, M\geq 0$, such that $\v{x}^{(k+M)}$ is Pareto critical, $\eta^{(k+M)}=\widehat{\eta}$, and $\v{x}^{(k+m)}$ is not Pareto critical for each $m < M$. 
    
    We define this type of sub-sequence as ${\widehat{\eta}}$\emph{-Pareto Critical end} ($\mathrm{PC}_{\widehat{\eta}}$);
    
    \item $\{\v{x}^{(k+m)}\}_{m\in\N}$, $k\geq 0$, such that there is $M\in\N$ for which $\v{x}^{(k+m)}$ is Pareto critical and $\eta^{(k+m)}=0$, for each $m\geq M$, and $\v{x}^{(k+m)}$ is not Pareto critical for each $m < M$.

    We define this type of sub-sequence as $0$\emph{-Pareto Critical end} ($\mathrm{PC}_0$);
    
    \item $\{\v{x}^{(k+m)}\}_{m\in\N}$, $k\geq 0$, such that $\v{x}^{(k+m)}$ is not Pareto critical for each $m \geq 0$.

    We define this type of sub-sequence as \emph{non-Pareto Critical} (nPC);
\end{enumerate}

\noindent More precisely, we have that any sequence $\{\v{x}^{(k)}\}_{k\in\N}$ has one of the following structures:
\begin{enumerate}
    \item $N\in\N$ consecutive PC$_{\widehat{\eta}}$ sub-sequences, followed by one PC$_0$ sub-sequence;
    
    \item $N\in\N$ consecutive PC$_{\widehat{\eta}}$ sub-sequences, followed by one nPC sub-sequence;
    
    \item Infinite consecutive PC$_{\widehat{\eta}}$ sub-sequences.
\end{enumerate}

On the other hand, a sequence generated for being strictly decreasing for all the objectives coincides with a sub-sequence of type nPC or, if it reaches a Pareto critical point, with a sub-sequence of type PC$_0$.
In \cite{Fliege2000}, the authors state that when such a kind of sequences are of the nPC type, assuming proper limitation hypotheses on the objective functions, they have at least one accumulation point, and each accumulation point of the sequence is a Pareto critical point.
Below, we report this theorem (\cite[Theorem 1]{Fliege2000}), adapted to the notation used in this work.

\begin{teo}[Theorem 1 - \cite{Fliege2000}]\label{teo:convergence_fliege}
    Let $\{\v{x}^{(k)}\}_{k\in\N}$ be a sequence defined by \eqref{eq:seq_bcktrck_theory}, but with $\widehat{\eta}=0$ (see \Cref{rem:nondom_restart_and_speccase_old}). Let $\v{x}^{(k)}$ be not Pareto critical, for each $k\in\N$. Then:
    \begin{enumerate}
        \item $\{\v{x}^{(k)}\}_{k\in\N}$ stays bounded and has at least one accumulation point, if the function $\v{f}$ (see \Cref{def:Pareto_definitions}) has bounded level sets in the sense that 
        $\{\v{x}\in\R^n \ | \ \v{f}(\v{x})\leq \v{f}(\v{x}^{(0)})\}$
        is bounded;
        \item Every accumulation point of the sequence $\{\v{x}^{(k)}\}_{k\in\N}$ is a Pareto critical point.
    \end{enumerate}
\end{teo}

Now, we state a theorem for characterizing the convergence properties of the sequences obtained using the new backtracking strategy proposed; i.e., sequences determined by \eqref{eq:seq_bcktrck_theory}. In this theorem, we exclude the case of a sequence ending with a PC$_0$ sub-sequence because, in that case, the results are trivial.

\begin{teo}\label{teo:my_conv}
    Let $\{\v{x}^{(k)}\}_{k\in\N}$ be a sequence defined by \eqref{eq:seq_bcktrck_theory} and let $\mathcal{C}$ denotes the set of all the Pareto critical points of $\v{f}:\R^m\rightarrow \R^n$ (see \Cref{def:Pareto_definitions}). Let us assume the sequence does not end with a PC$_0$ sub-sequence. Then:
    \begin{enumerate}
        \item $\liminf_{k\rightarrow +\infty} \mathrm{dist}(\v{x}^{(k)}, \mathcal{C}) = 0$;
        
        \item if $\{\v{x}^{(k)}\}_{k\in\N}$ ends with a sub-sequence $\{\v{x}^{(k_0+m)}\}_{m\in\N}$ of nPC type, $k_0\geq 0$ fixed, and $\{\v{x}\in\R^n \ | \ \v{f}(\v{x})\leq \v{f}(\v{x}^{(k_0)})\}$ is bounded, then $\{\v{x}^{(k_0+m)}\}_{m\in\N}$ is bounded, admits at least one accumulation point and all its accumulation points are Pareto critical;
    \end{enumerate}
\end{teo}

\begin{proof}
    The structure of the sequence can be of only two types:
    \begin{itemize}
        \item $N\in\N$ consecutive PC$_{\widehat{\eta}}$ sub-sequences, followed by one nPC sub-sequence;
    
        \item Infinite consecutive PC$_{\widehat{\eta}}$ sub-sequences.
    \end{itemize}

    In the first case, we have to prove both the theses of the theorem. Nonetheless, using \Cref{teo:convergence_fliege} [Theorem 1 - \cite{Fliege2000}] the proof of item 2 is trivial and, as a consequence, also the proof of item 1.

    Now, only item 1 for a sequence of infinite consecutive PC$_{\widehat{\eta}}$ sub-sequences remains to be proven. Let $\{\v{x}^{(k)}\}_{k\in\N}$ be a a sequence of this type; then, for each $k\geq 0$, there is $h\geq k$ such that $\v{x}^{(h)}$ is Pareto critical and $\eta^{(h)}=\widehat{\eta}$. Therefore, it holds

    $$\liminf_{k\rightarrow +\infty} \ \mathrm{dist}(\v{x}^{(k)},\mathcal{C}) = \sup_{k\geq 0} \lbrace \inf_{h\geq k} \mathrm{dist}(\v{x}^{(h)}, \mathcal{C})\rbrace = 0\,.$$

\end{proof}

\subsection{Implementation}

Here, we report the pseudocode of an algorithm (\Cref{alg:my_mgd}) that implements a MGD procedure with respect to the backtracking strategy defined in this section. We assume a maximum number of iterations $K\in\N$, a maximum number of backtracking steps $\Theta\in\N$, and, for simplicity, we set $\widehat{\eta}=\eta_0 \alpha^{\Theta}$, where $\eta_0\in\R_+$ is the base value of the step-size for each iteration. The direction $\v{p}^*$ described in the algorithm is the vector obtained solving the LP problem \eqref{eq:fliege2000} of \cite{Fliege2000}, the LP problem \eqref{eq:my_LP_explicit} proposed in this work, or any other problem defined with the same scope.

\begin{alg}\label{alg:my_mgd}
    Let us consider the unconstrained MOO problem \eqref{eq:MOOprob}. Then, we define the following MGD algorithm for the implementation of the descent method \eqref{eq:seq_bcktrck_theory}.
    
    \begin{description}
    
    \item[Data:] $\v{x}^{(0)}\in\R^n$ starting point for \eqref{eq:iterative_descent}; $c_1, \alpha\in (0, 1)$ parameters for the Armijo condition; $\eta_0$ starting value for the step length; $\Theta\in\N$ maximum number of backtracking steps; $K\in\N$ maximum number of iterations; $\mathcal{P}$ sub-problem for computing $\v{p}^{*\,(k)}$ at each iteration.
    
    \item[Procedure:] \quad
        \begin{algorithmic}[1]
        \For{$k=0,\ldots ,(K-1)$}
            \State $\v{p}^{*(k)}\gets$ solution of $\mathcal{P}$, defined by $\v{x}^{(k)}$
            \State $\eta^{(k)}\gets\eta_0$
            \For{$t=0,\ldots ,(\Theta-1)$}
                \If{\eqref{eq:armijo_moo_t} is true for each $i=1,\ldots ,m$}
                    \State break
                \Else
                    \State $\eta^{(k)}\gets \eta^{(k)}\alpha$
                \EndIf
            \EndFor
            \State $\widetilde{\v{x}}\gets \v{x}^{(k)} + \eta^{(k)} \v{p}^{*(k)}$
            \If{$\widetilde{\v{x}}$ is dominated by $\v{x}^{(k)}$} \Comment{Always false, if \eqref{eq:armijo_moo_t} is true $\forall \ i=1,\ldots ,m$}
                \State break
            \Else
                \State $\v{x}^{(k+1)}\gets \widetilde{\v{x}}$
            \EndIf
            
        \EndFor
        \State $\widehat{\v{x}}\gets \v{x}^{(k)}$
        \State \Return: $\widehat{\v{x}}$
        \end{algorithmic}
    \end{description}
\end{alg}

\subsubsection{Storage of Non-Dominated Intermediate Points}\label{sec:storage_pts}

Until now, we focused on the advantages of finding a new non-dominated point $\v{x}^{(k+1)}$ even if $\v{x}^{(k)}$ is Pareto critical. Nonetheless, since in MOO it is crucial to identify as many Pareto optimals as possible, it is better if we do not ``forget'' $\v{x}^{(k)}$ when it is not dominated by $\v{x}^{(k+1)}$; indeed, we recall that is possible to have both $\v{x}^{(k+1)}$ non-dominated by $\v{x}^{(k)}$, and vice-versa.

Therefore, we can modify \Cref{alg:my_mgd} adding a ``domination-check'' for $\v{x}^{(k)}$ with respect to $\v{x}^{(k+1)}$; specifically, if $\v{x}^{(k)}$ is not dominated by $\v{x}^{(k+1)}$, we store it into a list $C$. Then, at the end of the procedure, the algorithm returns the last vector computed for the sequence, denoted by $\widehat{\v{x}}$, and all the $\v{x}\in C$ that are non-dominated by all the other vectors in $C\cup \{\widehat{\v{x}}\}$; this set of vectors will represent the (approximation of) Pareto optimals discovered by the optimization method during its run (not necessarily global Pareto optimals, due to the local nature of the method).

The result of these modifications is \Cref{alg:my_mgd_storage}.

\begin{alg}\label{alg:my_mgd_storage}
    Let us consider the unconstrained MOO problem \eqref{eq:MOOprob}. Then, define the following MGD algorithm for the implementation of the descent method \eqref{eq:iterative_descent}.
    
    \begin{description}
    
    \item[Data:] $\v{x}^{(0)}\in\R^n$ starting point for \eqref{eq:iterative_descent}; $c_1, \alpha\in (0, 1)$ parameters for the Armijo condition; $\eta_0$ starting value for the step length; $\Theta\in\N$ maximum number of backtracking steps; $K\in\N$ maximum number of iterations; $\mathcal{P}$ sub-problem for computing $\v{p}^{*\,(k)}$ at each iteration.
    
    \item[Procedure:] \quad
        \begin{algorithmic}[1]
        \State $C\gets$ empty list \Comment{Initialization of the list of $\v{x}^{(k)}$ non-dominated by $\v{x}^{(k+1)}$}
        \For{$k=0,\ldots ,(K-1)$}
            \State $\v{p}^{*(k)}\gets$ solution of $\mathcal{P}$, defined by $\v{x}^{(k)}$
            \State $\eta^{(k)}\gets\eta_0$
            \For{$t=0,\ldots ,(\Theta-1)$}
                \If{\eqref{eq:armijo_moo_t} is true for each $i=1,\ldots ,m$}
                    \State break
                \Else
                    \State $\eta^{(k)}\gets \eta^{(k)}\alpha$
                \EndIf
            \EndFor
            \State $\widetilde{\v{x}}\gets \v{x}^{(k)} + \eta^{(k)} \v{p}^{*(k)}$
            \If{$\widetilde{\v{x}}$ is dominated by $\v{x}^{(k)}$}
                \State break
            \Else
                \State $\v{x}^{(k+1)}\gets \widetilde{\v{x}}$
            \EndIf
            \If{$\v{x}^{(k)}$ is not dominated by $\v{x}^{(k+1)}$}
                \State Add $\v{x}^{(k)}$ to C
            \EndIf
            
        \EndFor
        \State $\widehat{\v{x}}\gets\v{x}^{(k)}$
        \State $\widebar{C} \gets$ $\{\v{x}\in C \ | \ \v{x} \textit{ non-dom. by } \v{y}\,, \ \forall \v{y}\in C\cup\{\widehat{\v{x}}\}\,, \ \v{y}\neq \v{x} \}$
        \State \Return: $\widehat{\v{x}}$, $\widebar{C}$
        \end{algorithmic}
    \end{description}
\end{alg}

\subsubsection{Blockwise Implementation for Searching Global Pareto Optimals}\label{sec:blockwise}

We recall that MGD methods have a local nature. Therefore, all the Pareto critical points returned by the procedures are approximations of global and/or local Pareto optimals, mainly depending on the starting point $\v{x}^{(0)}$. Therefore, a basic multi-start approach for the optimization procedure, made with respect to a set of $N\in\N$ random starting points
\begin{equation*}\label{eq:multistart_X0}
    X^{(0)} := \{\v{x}_1^{(0)},\ldots ,\v{x}_N^{(0)}\}\subset\R^n\,,
\end{equation*}
can be a simple but efficient method for exploring and detecting as much as possible the Pareto set and the Pareto front of a MOO problem. 

Concerning a multi-start approach, it is interesting to observe that the solutions $\v{\rho}^*_1,\ldots ,\v{\rho}^*_N$ of $N$ LP problems like \eqref{eq:my_LP_explicit}, if concatenated into a vector $\v{\rho}^*=(\v{\rho}^*_1, \ldots ,\v{\rho}^*_N)=((\v{p}^*_1,\beta^*_1), \ldots ,(\v{p}^*_N, \beta^*_N))\in\R^{N(n+1)}$, are solution of the LP problem 

\begin{equation}\label{eq:my_LP_explicit_blockwise}
    \begin{cases}
        \min_{\v{\rho}\in\R^{N(n+1)}} [ \ \v{c}_1^T \ | \ \cdots \ | \ \v{c}_N^T  ]\, \v{\rho}\\
        \quad \\
        \begin{bmatrix}
            A_1 & \quad & \quad\\
            \quad & \ddots & \quad\\
            \quad & \quad & A_N
        \end{bmatrix}
        \v{\rho}
        \leq \v{0}\\
        \quad \\
        %
        -\gamma_j \v{e} \leq \v{p}_j \leq \gamma_j \v{e}\,, \ &\forall j=1,\ldots ,N\\
        \beta_j\leq 0\,, \ &\forall j=1,\ldots ,N\\
    \end{cases}
    \,,
\end{equation}
and vice-versa; where $\v{c}_j\in\R^{n+1}$, $A_j\in\R^{m\times (n+1)}$, and $\gamma_j\in\R_+$ are: the vector of the linear objective function, the matrix of inequality constraints, and the scalar for boundary values of $\v{p}_j$, respectively, of the $j$-th LP problem, for each $j=1,\ldots ,N$.

Then, if the user cannot directly parallelize the multi-start procedure, they can still exploit \eqref{eq:my_LP_explicit_blockwise} for a fast computation of the $N$ directions $\v{p}_1^*,\ldots ,\v{p}_N^*$. Of course, techniques for fastening the evaluation of the Jacobians $G_1,\ldots , G_N$ are suggested, but their discussion does not fall within the scope of this work.

A similar block-wise parallelization can be implemented also for problem \eqref{eq:fliege2000}.

\section{Numerical Experiments}\label{sec:experiments}

In this section, we study and analyze the behavior of the new method \eqref{eq:my_LP_explicit} for computing the shared descent direction of multiple objectives, using the new backtracking strategy \eqref{eq:seq_bcktrck_theory} for MGD methods. Specifically, we run \Cref{alg:my_mgd_storage}, using \eqref{eq:my_LP_explicit} for computing the direction $\v{p}^{*\,(k)}$, with parameter $c_\beta=\norm{\v{g}}_2 + 1$. Then, the analysis is made through a comparison with \cite{Fliege2000}, representing the baseline; i.e., the baseline is represented by $\eqref{eq:fliege2000}$, using a backtracking strategy designed for building a sequence strictly decreasing for all the objectives; we report the pseudocode of a MGD implementation using this backtracking strategy in \ref{sec:strdec_bcktrck} (see \Cref{alg:mgd_classic}). Nonetheless, we study also how the behavior of \eqref{eq:fliege2000} changes if \eqref{eq:seq_bcktrck_theory} is used (i.e., we run \Cref{alg:my_mgd_storage} with respect to \eqref{eq:fliege2000}) and how the behavior of \eqref{eq:my_LP_explicit} changes if the backtracking strategy for strictly decreasing objectives is used (i.e., we run \Cref{alg:mgd_classic} with respect to \eqref{eq:my_LP_explicit}).

Summarizing, we analyze the behavior of:
\begin{itemize}
    \item \Cref{alg:mgd_classic} and \Cref{alg:my_mgd_storage} run with respect to \eqref{eq:fliege2000}. We denote them as BT$_{\rm base}$-\ref{eq:fliege2000} and \ref{eq:seq_bcktrck_theory}-\ref{eq:fliege2000}, respectively;
    \item \Cref{alg:mgd_classic} and \Cref{alg:my_mgd_storage} run with respect to \eqref{eq:my_LP_explicit}. We denote them as BT$_{\rm base}$-\ref{eq:my_LP_explicit} and \ref{eq:seq_bcktrck_theory}-\ref{eq:my_LP_explicit}, respectively;.
\end{itemize}
The baseline is represented by BT$_{\rm base}$-\ref{eq:fliege2000} (see \cite{Fliege2000}), while the new method we propose is represented by \ref{eq:seq_bcktrck_theory}-\ref{eq:my_LP_explicit}. All the other ``intermediate'' cases are added to the analysis to better understand the effects of using different backtracking strategies and different shared descent directions in a MGD method.

All the cases are compared analyzing the results obtained by applying them to three well-known MOO test problems: the Fonseca-Fleming problem \cite{Fonseca1995}, the Kursawe problem \cite{Kursawe1990}, and the Viennet problem \cite{Viennet1996}. For all these problems, all the algorithms are executed with respect to $N=500$ randomly sampled starting points (exploiting a blockwise implementation, see \Cref{sec:blockwise}); the backtracking strategy is characterized by a maximum number of steps $\Theta=40$ and parameters $c_1=10^{-9}$, $\alpha=0.8$, and $\eta_0=1$ for the Armijo condition \eqref{eq:armijo_moo_t}. For more details about the test problems, the starting point samplings, etc., see \ref{sec:moo_test_probs}.

The comparison between the methods is performed by measuring the ratio of sequences, among all the $N$ ones, that have detected at least one Pareto critical that is non-dominated by all the other Pareto critical points detected by itself and the other sequences. Indeed, such a kind of Pareto critical points can be considered as a good approximation of a global Pareto optimal and, therefore, it means that the sequence reached at least one point near to the Pareto set of the problem. 

Specifically, let $\widebar{C}_{j}$ be the set containing the outputs of a MGD algorithm, starting from $\v{x}^{(0)}_j$, for each $j=1,\ldots ,N$; i.e., $\widebar{C}_j = \{\widehat{\v{x}}\}$ for \Cref{alg:mgd_classic}, while $\widebar{C}_j$ contains $\widehat{\v{x}}$ and all the non-dominated critical points of the sequence for \Cref{alg:my_mgd_storage} (see the algorithm's pseudocode). Let $\widebar{C}^N$ be the union of all the sets $\widebar{C}_j$, i.e., $\widebar{C}^N:=\bigcup_{j=1}^N \widebar{C}_j$, and let $\widetilde{C}_j$ be the set points in $\widebar{C}_j$ non-dominated by all the other points in $\widebar{C}^N$, i.e.
\begin{equation}\label{eq:nondom_j}
    \widetilde{C}_j := \lbrace \v{x}\in\widebar{C}_j \ | \ \v{x} \textit{ non-dom. by } \v{y}\,, \ \forall \v{y}\in\widebar{C}^N, \v{y}\neq\v{x} \rbrace\,.
\end{equation}
Then, for a MGD algorithm, we define the index
\begin{equation}\label{eq:nodom_ratio_index}
    \mathrm{P}^N := \frac{\#\text{ of non-empty } \widetilde{C}_j}{N}=\frac{\left|\{j \ | \ \widetilde{C}_j\neq\emptyset\,, \ j=1,\ldots ,N\}\right|}{N}\,,
\end{equation}
where the superscript $N$ denotes the number of sequences used for computing the index. We define this index as the \emph{global Pareto ratio} of the MGD algorithm, with respect to $N$ runs.

\begin{rem}[About the global Pareto ratio index]
    We introduce index \eqref{eq:nodom_ratio_index} in order to analyze the performance of different MGD algorithms independently on the availability of a known Pareto set and Pareto front. Alternative indices can be used, such as the Hypervolume (HV) indicator, which is able to measure the ``quality'' of a Pareto set/front through the hyper-volume of the dominated region in the objectves' space (see \cite{Zitzler2000,Audet2021}). Nonetheless, the HV (and the other MOO indicators) focuses on the spread properties of the detected global Pareto optimals; therefore, we prefer instead to use index \eqref{eq:nodom_ratio_index} for our analyses, because it is better for quantifying the ability of a MGD algorithm to reach the Pareto set of the problem, varying the starting point.
\end{rem}

\subsection{Numerical Results and Analyses}\label{sec:num_res}

In this subsection, we analyze and study the results obtained by the four given MGD algorithms applied to the three test problems considered (FonsecaFleming, Kursawe, and Viennet, see \ref{sec:moo_test_probs}). 
We start analyzing the results for the Fonseca-Fleming problem, focusing on the different paths generated by using directions that are solutions of \eqref{eq:fliege2000} or \eqref{eq:my_LP_explicit}; in this case, the type of backtracking strategy is not crucial, due to the characteristics of the objective functions. On the contrary, when we continue to the results for the Kursawe problem, we observe how the usage of \ref{eq:seq_bcktrck_theory} (i.e., \Cref{alg:my_mgd_storage}) improves concretely the global Pareto ratio $\mathrm{P}^N$. In the end, we analyze the results for the Viennet problem; for this latter case, we observe that only the MGD method \ref{eq:seq_bcktrck_theory}-\ref{eq:my_LP_explicit} is able to return a good $\mathrm{P}^N$ value, outperforming all the other methods. A summary of the global Pareto ratio values for all the methods, with respect to all the problems, is reported in \Cref{tab:global_pareto_percentage}.

\begin{table}[htb]
    \centering
    \begin{tabular}{|c|c||c|c|}
    \hline
    Problem & LP Problem & BT$_{\rm base}$ (\Cref{alg:mgd_classic}) & \ref{eq:seq_bcktrck_theory} (\Cref{alg:my_mgd_storage})\\
    \hline
    \hline
    \multirow{2}{*}{Fonseca-Fleming} & \ref{eq:fliege2000} & $100.00\%$ & $100.00\%$\\
    \hhline{~---}
    & \ref{eq:my_LP_explicit} & $100.00\%$ & $100.00\%$ \\
    \hline
    \hline
    \multirow{2}{*}{Kursawe} & \ref{eq:fliege2000} & $31.20\%$ & $66.40\%$\\
    \hhline{~---}
    & \ref{eq:my_LP_explicit} & $24.00\%$ & $63.60\%$\\
    \hline
    \hline
    \multirow{2}{*}{Viennet} & \ref{eq:fliege2000} & $37.00\%$ & $42.00\%$\\
    \hhline{~---}
    & \ref{eq:my_LP_explicit} & $34.80\%$ & $92.80\%$\\
    \hline
    \end{tabular}
    \caption{Global Pareto convergence ratios $\mathrm{P}^N$ (represented as percentages) computed for all the different cases, with respect to $N=500$ sequences $\{\v{x}_j^{(k)}\}_{k\in\N}$, $j=1,\ldots ,N$.}
    \label{tab:global_pareto_percentage}
\end{table}

\subsubsection{Fonseca-Fleming}

With respect to the Fonseca-Fleming test problem, all the MGD methods return $100\%$ of $\mathrm{P}^N$ (see \Cref{tab:global_pareto_percentage}). Looking only at these values, we deduce that the sequences built with \eqref{eq:my_LP_explicit} are a valuable alternative to the ones built with \eqref{eq:fliege2000}. We do not see differences in the performances while changing the backtracking strategy because the characteristics of the problem are such that \eqref{eq:armijo_moo_t} is very easy to be satisfied in regions far from the Pareto set (e.g., both the objective functions are convex); therefore, all the sequences are decreasing with respect to all the objectives except when an approximation of a global Pareto optimal has already been reached.

However, looking at the paths of the sequences' images in the objectives' space in \Cref{fig:FF_obj}, we observe different behaviors for \eqref{eq:fliege2000} and \eqref{eq:my_LP_explicit}. For both the methods, the paths end (approximately) on the Pareto front of the problem, but the paths generated with \eqref{eq:fliege2000} are attracted more by the ``center'' of the front (Figures \ref{fig:FF_obj}a and \ref{fig:FF_obj}b), while the paths generated with \eqref{eq:my_LP_explicit} are distributed more uniformly on the front (Figures \ref{fig:FF_obj}c and \ref{fig:FF_obj}d). 

Summarizing, with these results we show that \eqref{eq:my_LP_explicit} is able to generate good sequences for a MGD algorithm, but the sequences we obtain are characterized by a different behavior from the ones generated with \eqref{eq:fliege2000} (as expected, by construction of \eqref{eq:my_LP_explicit}).

\begin{figure}[hbt]
    \centering
    \subcaptionbox{BT$_{\rm base}$-\ref{eq:fliege2000}}{\includegraphics[trim={0.5cm 1cm 5.5cm 3.5cm},clip,width=0.49\textwidth,height=0.165\textheight]{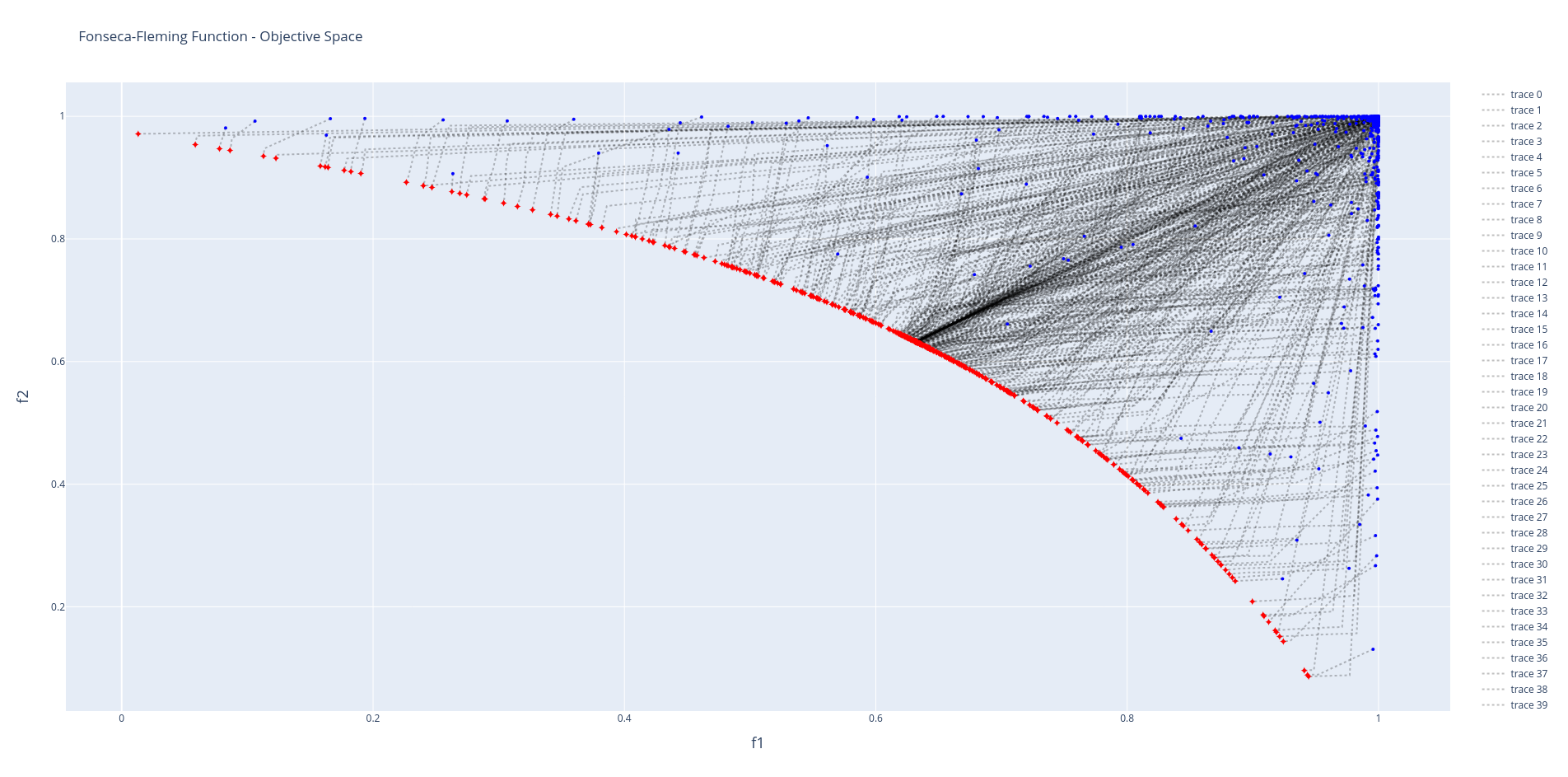}}
    \subcaptionbox{\ref{eq:seq_bcktrck_theory}-\ref{eq:fliege2000}}{\includegraphics[trim={0.5cm 1cm 5.5cm 3.5cm},clip,width=0.49\textwidth,height=0.165\textheight]{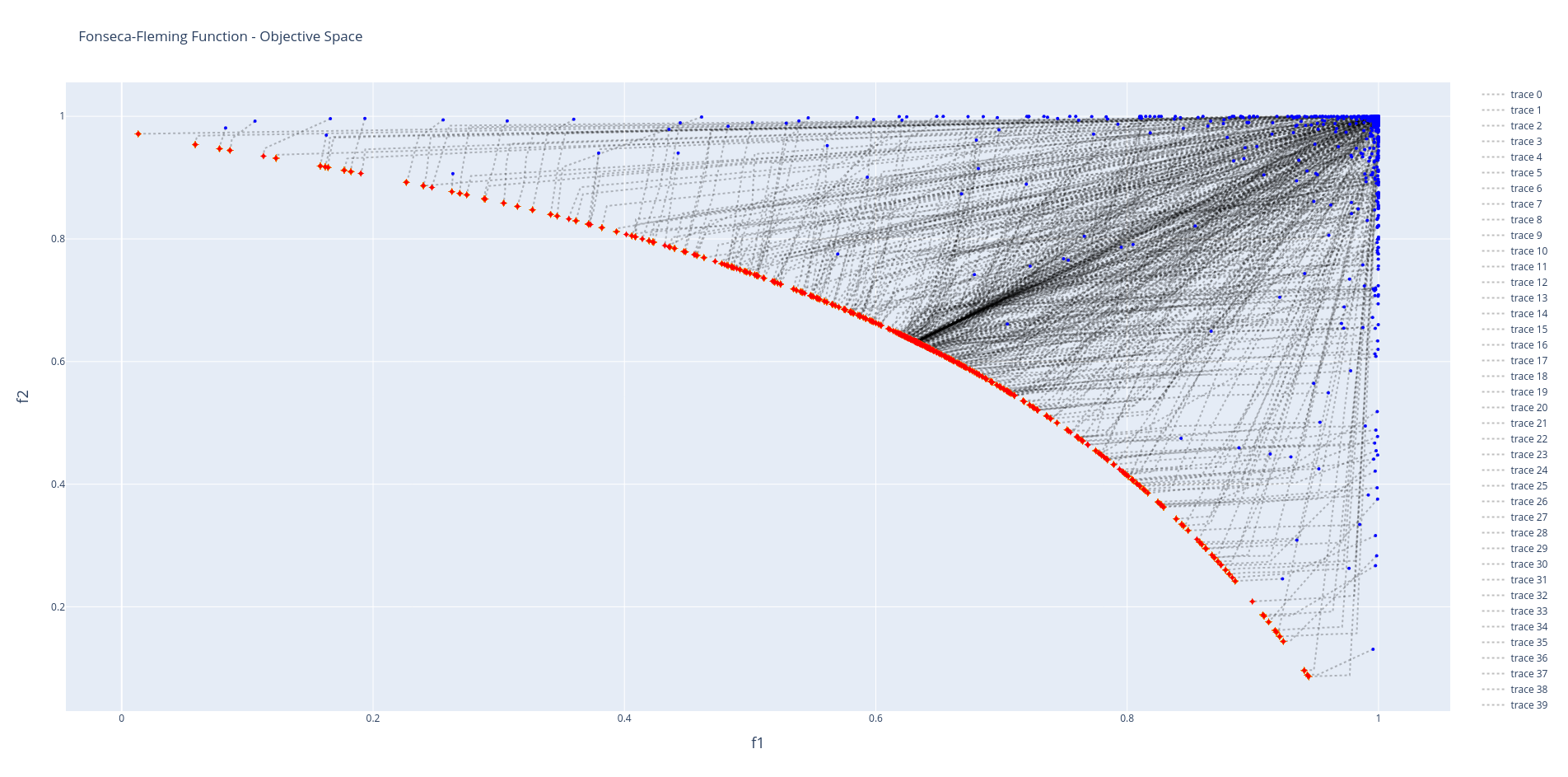}}
    \subcaptionbox{BT$_{\rm base}$-\ref{eq:my_LP_explicit}}{\includegraphics[trim={0.5cm 1cm 5.5cm 3.5cm},clip,width=0.49\textwidth,height=0.165\textheight]{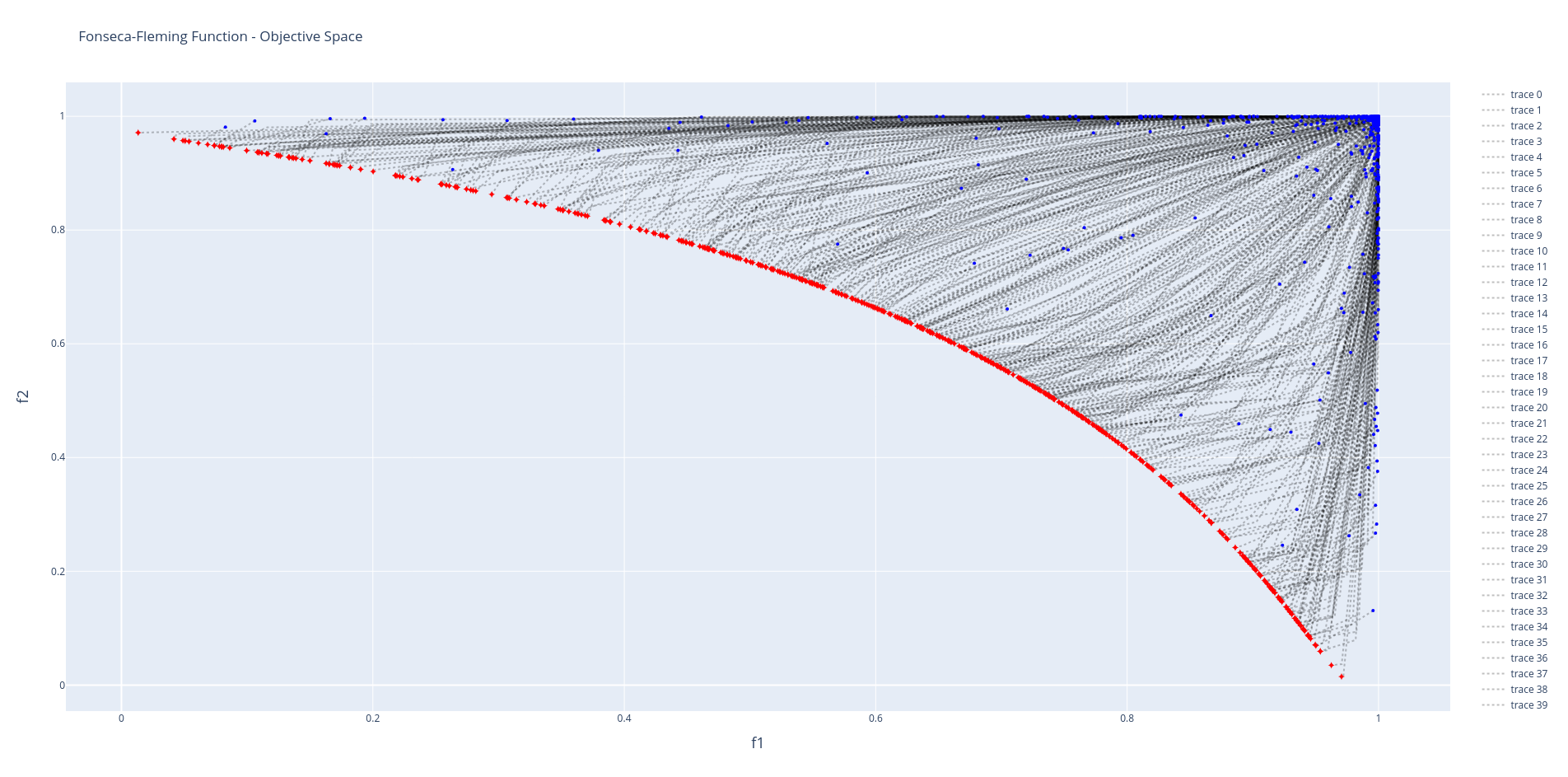}}
    \subcaptionbox{\ref{eq:seq_bcktrck_theory}-\ref{eq:my_LP_explicit}}{\includegraphics[trim={0.5cm 1cm 5.5cm 3.5cm},clip,width=0.49\textwidth,height=0.165\textheight]{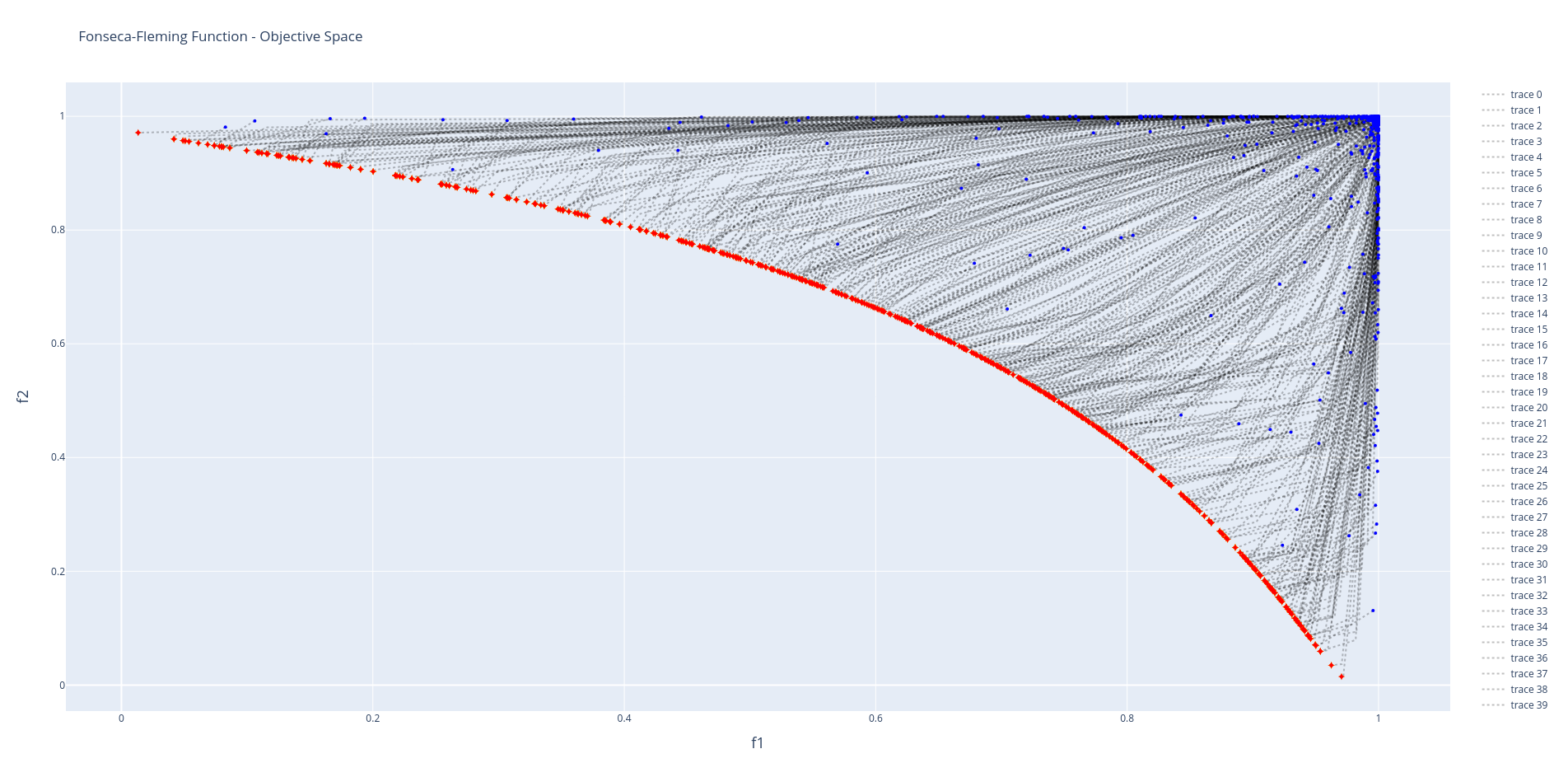}}
    \caption{Fonseca-Flemig test case. Movement of the $N=500$ sequences in the objectives' space. The blue dots are the starting points' images $\v{f}(\v{x}^{(0)}_j)\in\R^2$, $j=1,\ldots ,N$, while the red dots are the last points of the sequence. The black, dotted, and piece-wise linear curves describe the movement of each sequence from its starting point to its last element.  
    Critical points stored during \Cref{alg:my_mgd_storage} are not represented because they are (almost) coincident with red dots.}
    \label{fig:FF_obj}
\end{figure}

\subsubsection{Kursawe}

The Kursawe problem is more complicated than the Fonseca-Fleming one. In this case, analogously to the Fonseca-Fleming test, we observe similar performances but different path behaviors for \eqref{eq:fliege2000} and \eqref{eq:my_LP_explicit} (given the same backtracking strategy); actually, the performances of methods based on \eqref{eq:my_LP_explicit} are a bit smaller than the performances of methods based on \eqref{eq:fliege2000} (see \Cref{tab:global_pareto_percentage}). The different path behaviors of the two methods are evident if \Cref{fig:K_obj}; in particular, we have that \eqref{eq:my_LP_explicit} generates sequences attracted more by the ``center'' of the front, while \eqref{eq:fliege2000} generates sequences mainly focused in reducing the value of $f_1$ (i.e., ``leftward movement'' in the objectives' space).

With respect to the global Pareto ratio, the real difference is made by \ref{eq:seq_bcktrck_theory}. Indeed, independently on the LP problem used for computing the directions, the adoption of the new backtracking strategy guarantees a $\mathrm{P}^N\simeq 65\%$; more than twice the value obtained using \Cref{alg:mgd_classic} (see \Cref{tab:global_pareto_percentage} and the approximated Pareto front/set in \Cref{fig:K_front} and \Cref{fig:K_set}, respectively). We can explain this phenomenon looking at the orange dots in Figures \ref{fig:K_obj}b and \ref{fig:K_obj}d; the orange dots are points $\v{x}_j^{(k)}$ with respect to which the backtracking strategy cannot find a point along the direction $\v{p}^{*\,(k)}$ that is able to decrease all the objectives, but \eqref{eq:seq_bcktrck_theory} finds $\v{x}_j^{(k+1)}$ such that it is not dominated by $\v{x}_j^{(k)}$ and vice-versa; such a kind of points are Pareto critical points or points very near to them (see \Cref{fig:K_dom_critregions}). In Figures \ref{fig:K_obj}b and \ref{fig:K_obj}d, we clearly see that there are paths of (almost) Pareto critical orange dots that (often) end with an approximated global Pareto optimal; i.e., \ref{eq:seq_bcktrck_theory} permits the movement along trajectories made of Pareto critical points, increasing the probability of reaching the Pareto set/front of the problem. On the contrary, BT$_{\rm base}$ stops the iterations as soon as a Pareto critical point is reached, reducing the probability of finding a global Pareto optimal from a ``non-suitable'' starting point; an example is represented by the dominated red dots in Figures \ref{fig:K_obj}a and \ref{fig:K_obj}c.

\begin{figure}[htb!]
    \centering
    \includegraphics[trim={4.5cm 0cm 4.5cm 0cm},clip,width=0.45\textwidth]{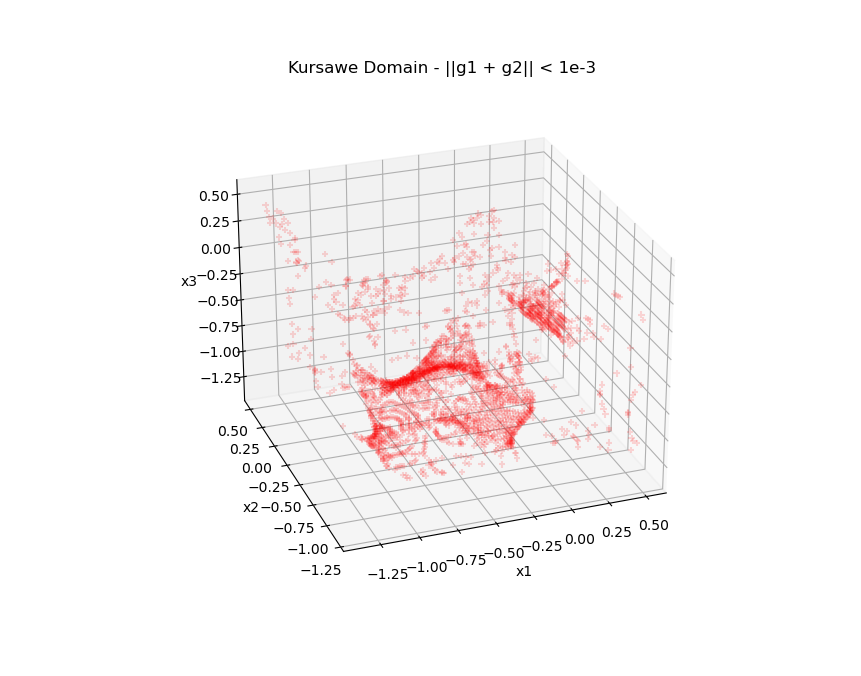}
    \caption{Kursawe test case. The red dots are points in the cube $[-1.5,0.5]^3$ (see \ref{sec:moo_test_probs}) where the normalized gradients are (almost) opposite (i.e., $\norm{\bar{\v{g}}_1 + \bar{\v{g}}_2}_2 < 10^{-3}$).}
    \label{fig:K_dom_critregions}
\end{figure}

We conclude the analysis of the Kursawe problem focusing again on \Cref{fig:K_dom_critregions}. We observe that the Kursawe problem is characterized by $m=2$ objective functions; therefore, for the observations in \Cref{rem:case2objs} we have that the only difference between \eqref{eq:fliege2000} and \eqref{eq:my_LP_explicit} is in the properties of the directions evaluated for the points that are not Pareto critical. For Pareto critical points, the only difference is in the norm of $\v{p}^*$, if it is not equal to the null vector.

\begin{figure}[htb!]
    \centering
    \subcaptionbox{BT$_{\rm base}$-\ref{eq:fliege2000}}{\includegraphics[trim={0.5cm 1cm 5.5cm 3.5cm},clip,width=0.49\textwidth,height=0.165\textheight]{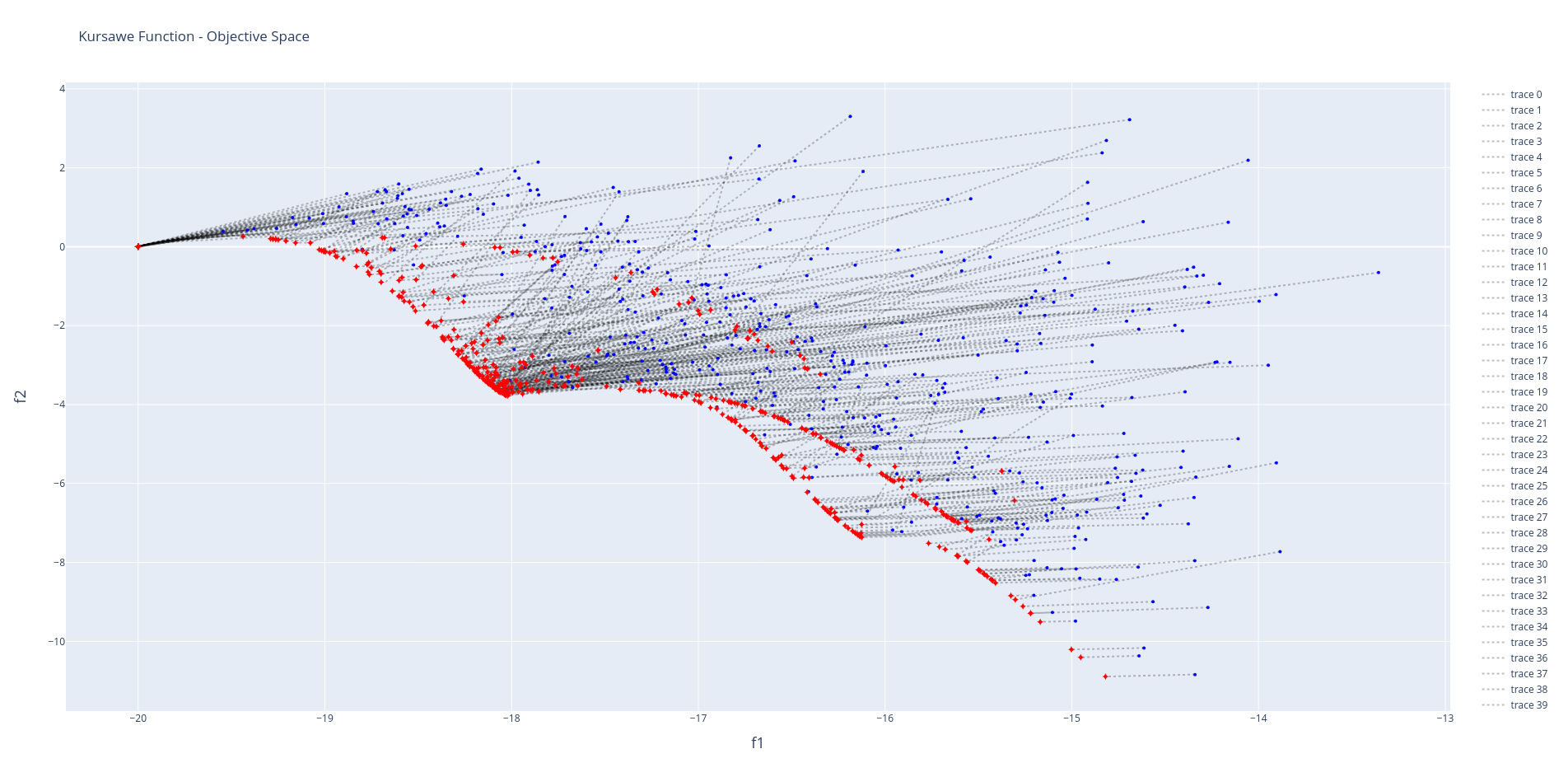}}
    \subcaptionbox{\ref{eq:seq_bcktrck_theory}-\ref{eq:fliege2000}}{\includegraphics[trim={0.5cm 1cm 5.5cm 3.5cm},clip,width=0.49\textwidth,height=0.165\textheight]{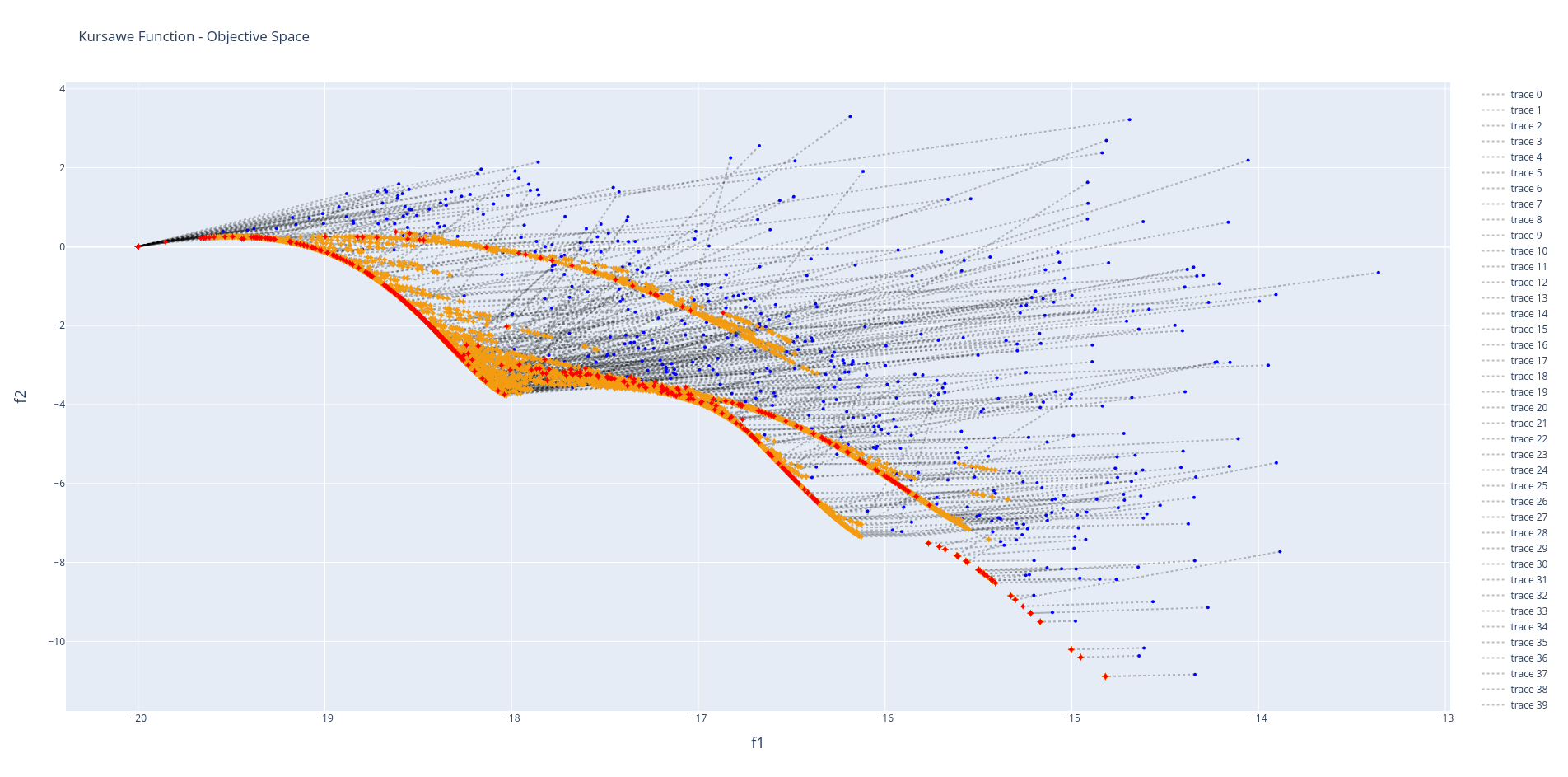}}
    \subcaptionbox{BT$_{\rm base}$-\ref{eq:my_LP_explicit}}{\includegraphics[trim={0.5cm 1cm 5.5cm 3.5cm},clip,width=0.49\textwidth,height=0.165\textheight]{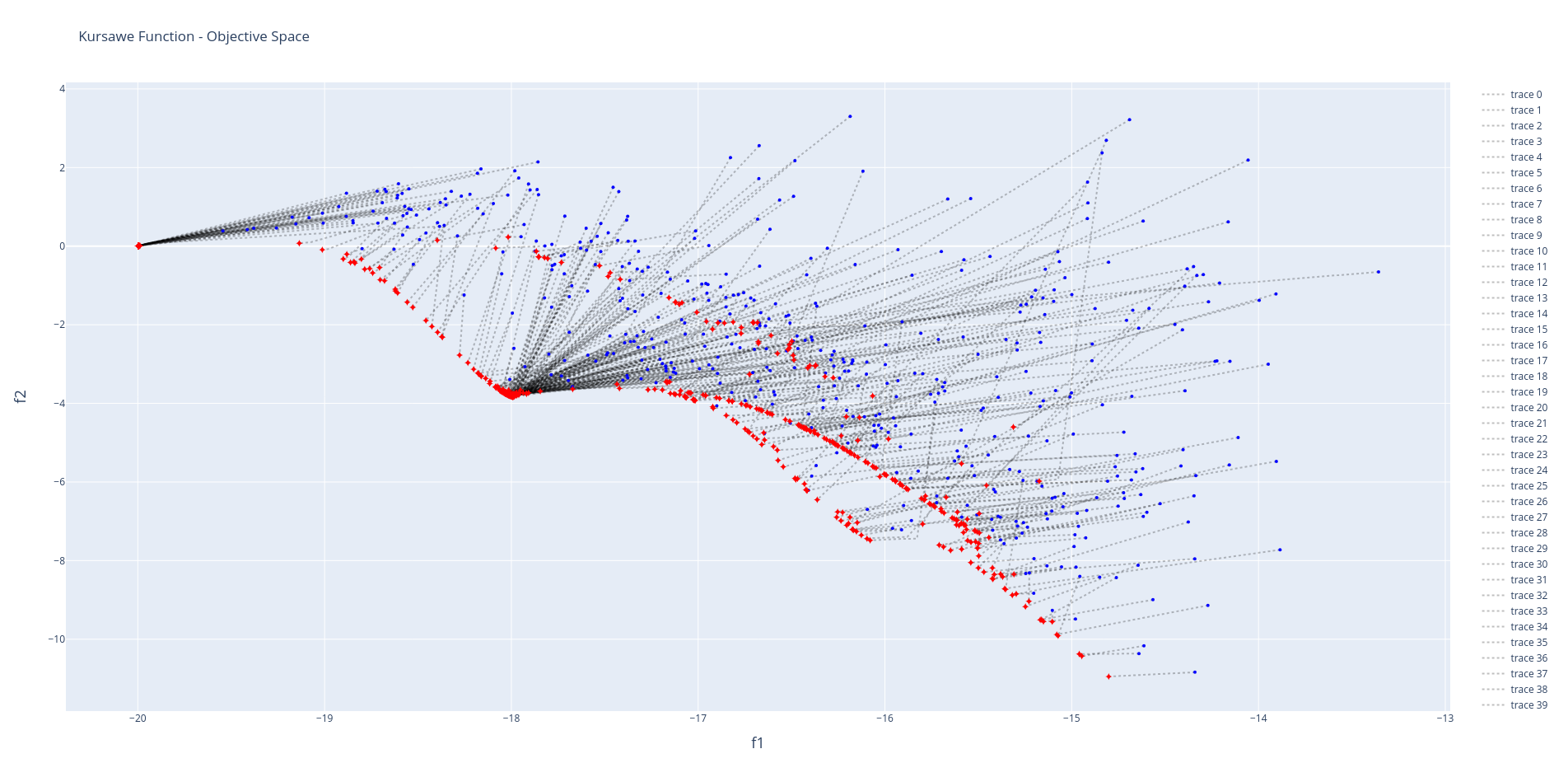}}
    \subcaptionbox{\ref{eq:seq_bcktrck_theory}-\ref{eq:my_LP_explicit}}{\includegraphics[trim={0.5cm 1cm 5.5cm 3.5cm},clip,width=0.49\textwidth,height=0.165\textheight]{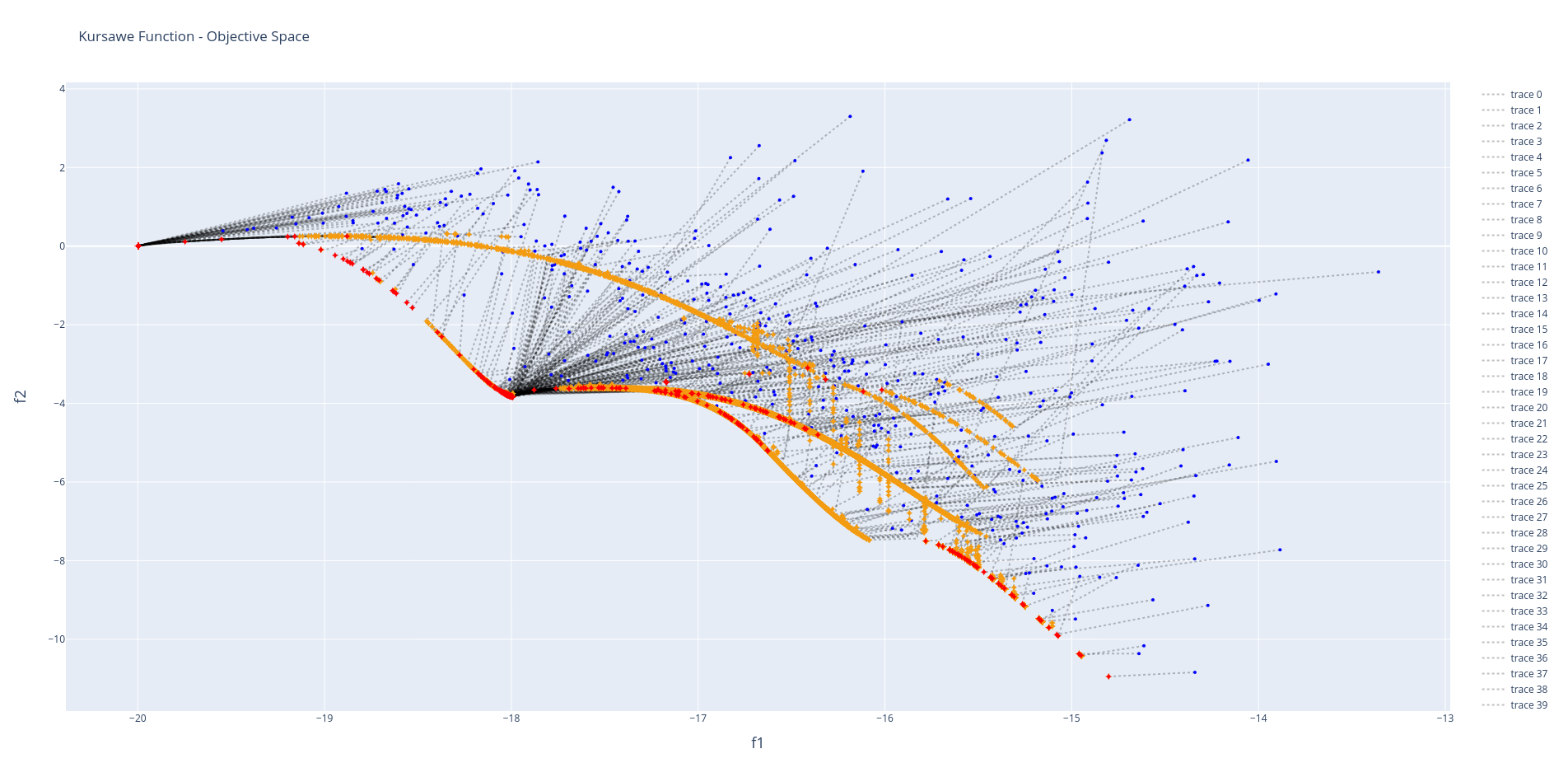}}
    \caption{Kursawe test case. Movement of the $N=500$ sequences in the objectives' space. The blue dots are the starting points' images $\v{f}(\v{x}^{(0)}_j)\in\R^2$, $j=1,\ldots ,N$, while the red dots are the last points of the sequence. The orange dots are the critical points stored during \Cref{alg:my_mgd_storage}, before the final pruning (see the pseudocode).
    The black, dotted, and piece-wise linear curves describe the movement of each sequence from its starting point to its last element.}
    \label{fig:K_obj}
\end{figure}

\begin{figure}[htb!]
    \centering
    \subcaptionbox{BT$_{\rm base}$-\ref{eq:fliege2000}}{\includegraphics[trim={0.5cm 1cm 5.5cm 3.5cm},clip,width=0.49\textwidth,height=0.165\textheight]{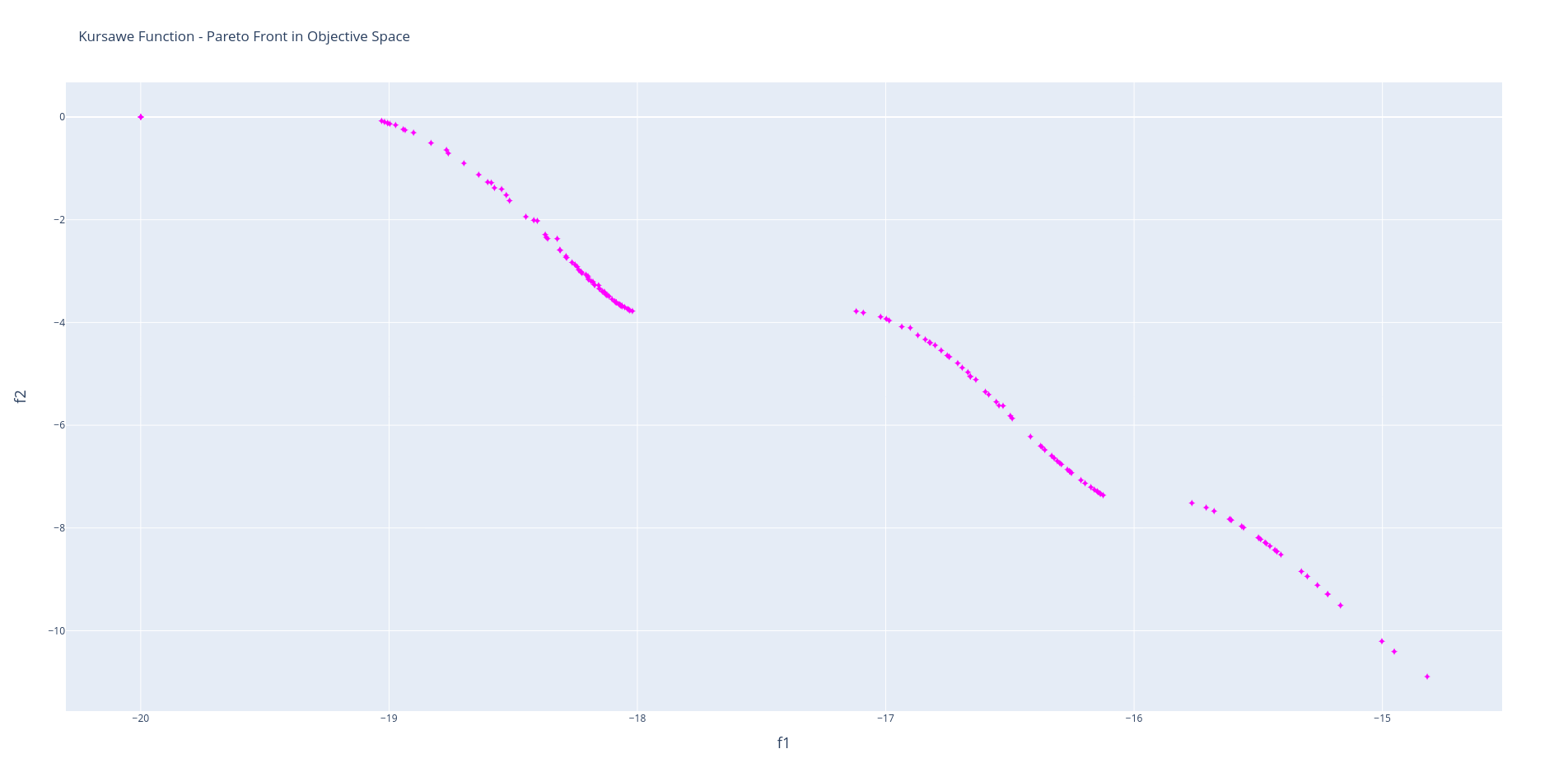}}
    \subcaptionbox{\ref{eq:seq_bcktrck_theory}-\ref{eq:fliege2000}}{\includegraphics[trim={0.5cm 1cm 5.5cm 3.5cm},clip,width=0.49\textwidth,height=0.165\textheight]{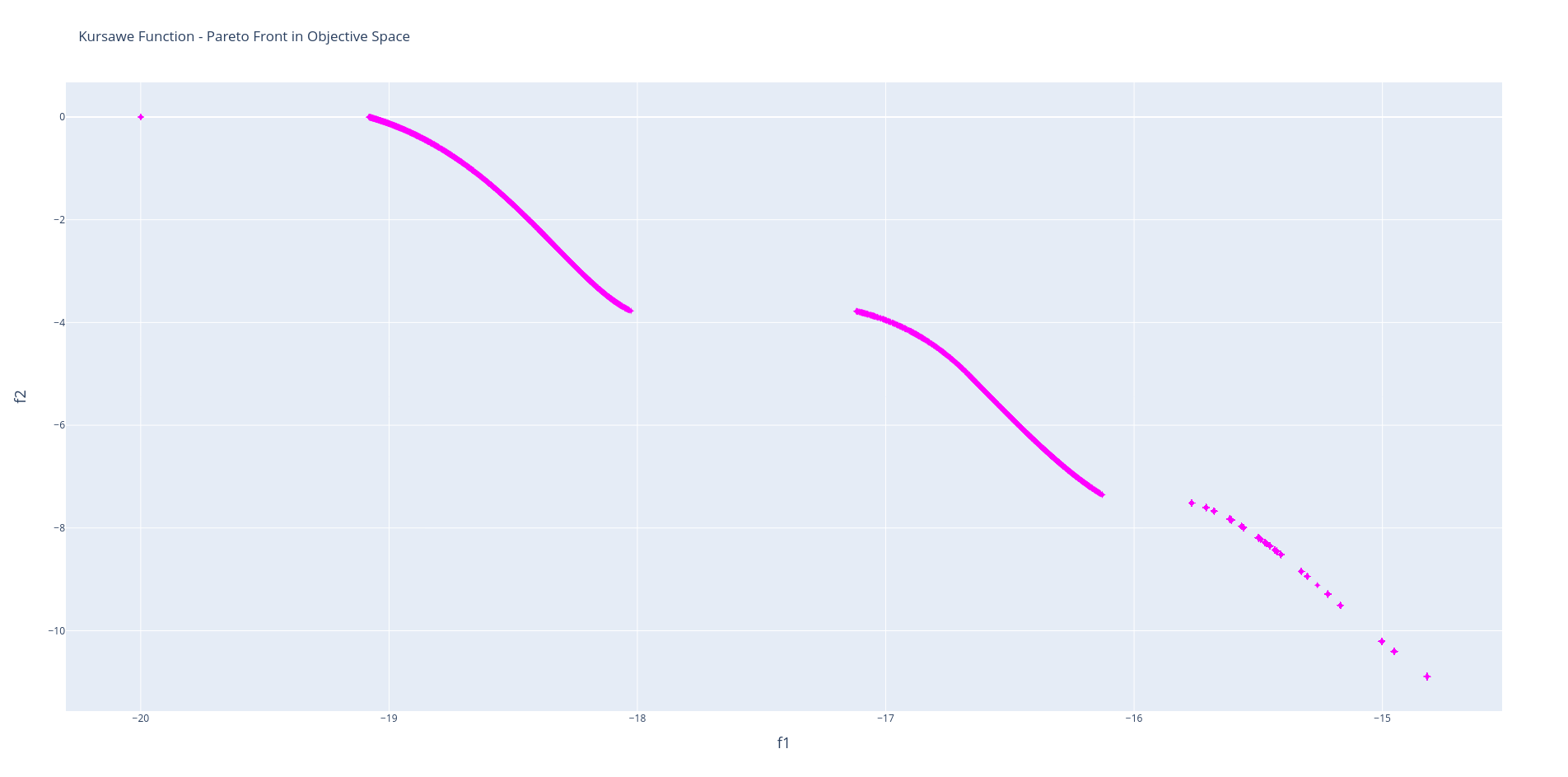}}
    \subcaptionbox{BT$_{\rm base}$-\ref{eq:my_LP_explicit}}{\includegraphics[trim={0.5cm 1cm 5.5cm 3.5cm},clip,width=0.49\textwidth,height=0.165\textheight]{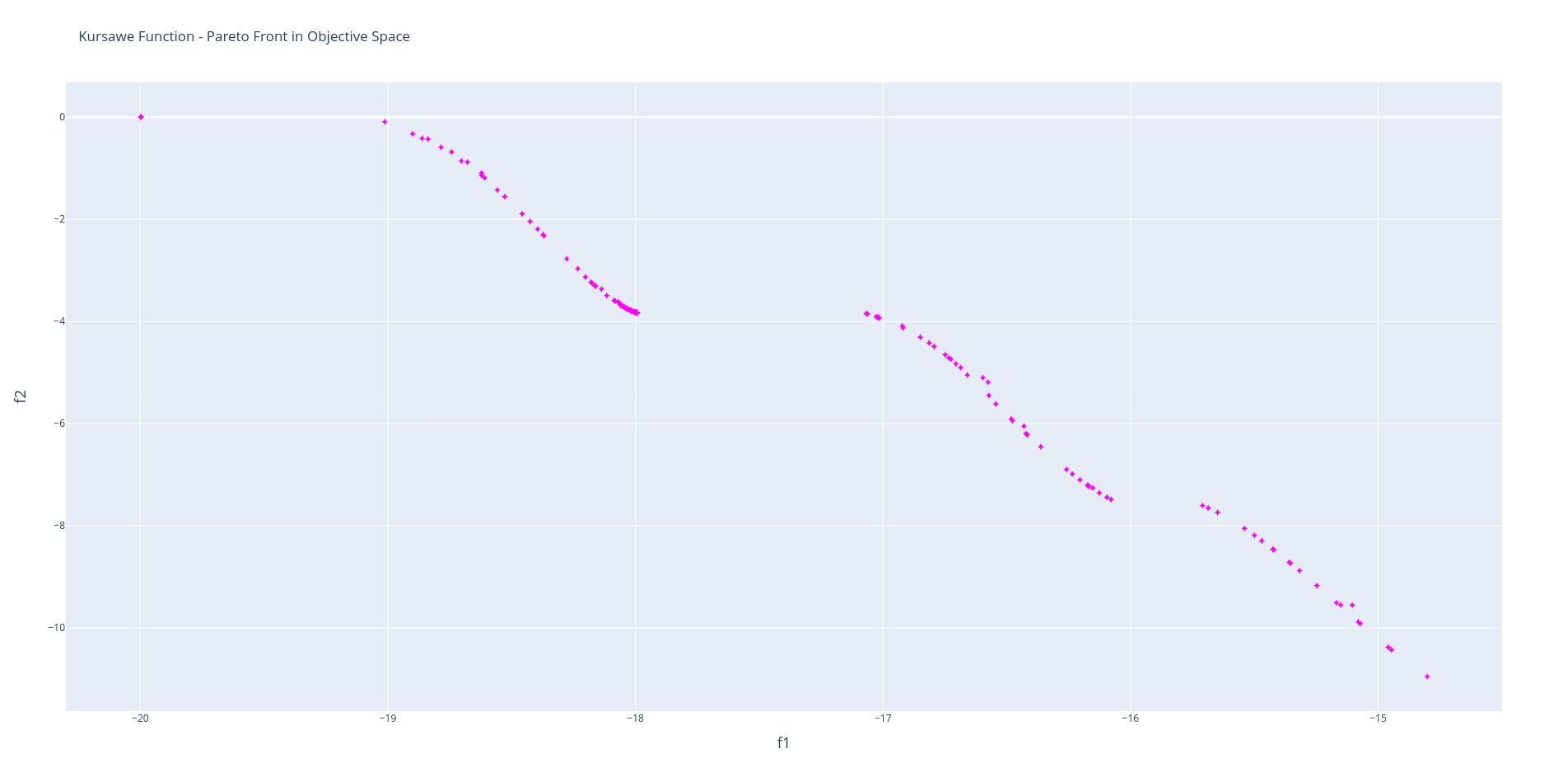}}
    \subcaptionbox{\ref{eq:seq_bcktrck_theory}-\ref{eq:my_LP_explicit}}{\includegraphics[trim={0.5cm 1cm 5.5cm 3.5cm},clip,width=0.49\textwidth,height=0.165\textheight]{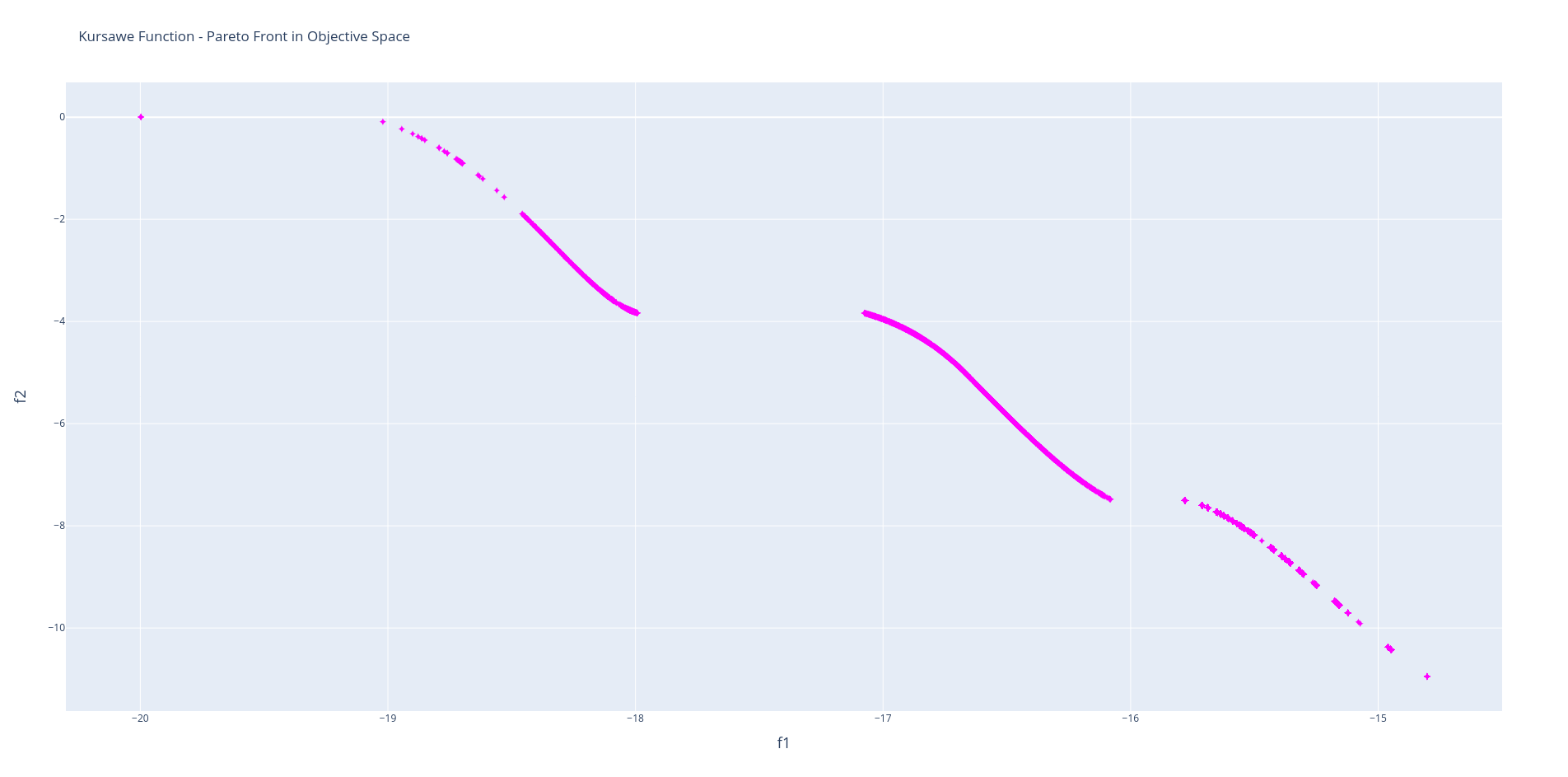}}
    \caption{Kursawe test case. Images in the objectves' space of the non-dominated points among all the outputs returned by the MGD algorithms; i.e., images of the points in $\bigcup_{j=1}^N \widetilde{C}_j$ (see \eqref{eq:nondom_j}).}
    \label{fig:K_front}
\end{figure}

\begin{figure}[htb!]
    \centering
    \subcaptionbox{BT$_{\rm base}$-\ref{eq:fliege2000}}{\includegraphics[trim={20cm 2cm 20cm 6cm},clip,width=0.4\textwidth]{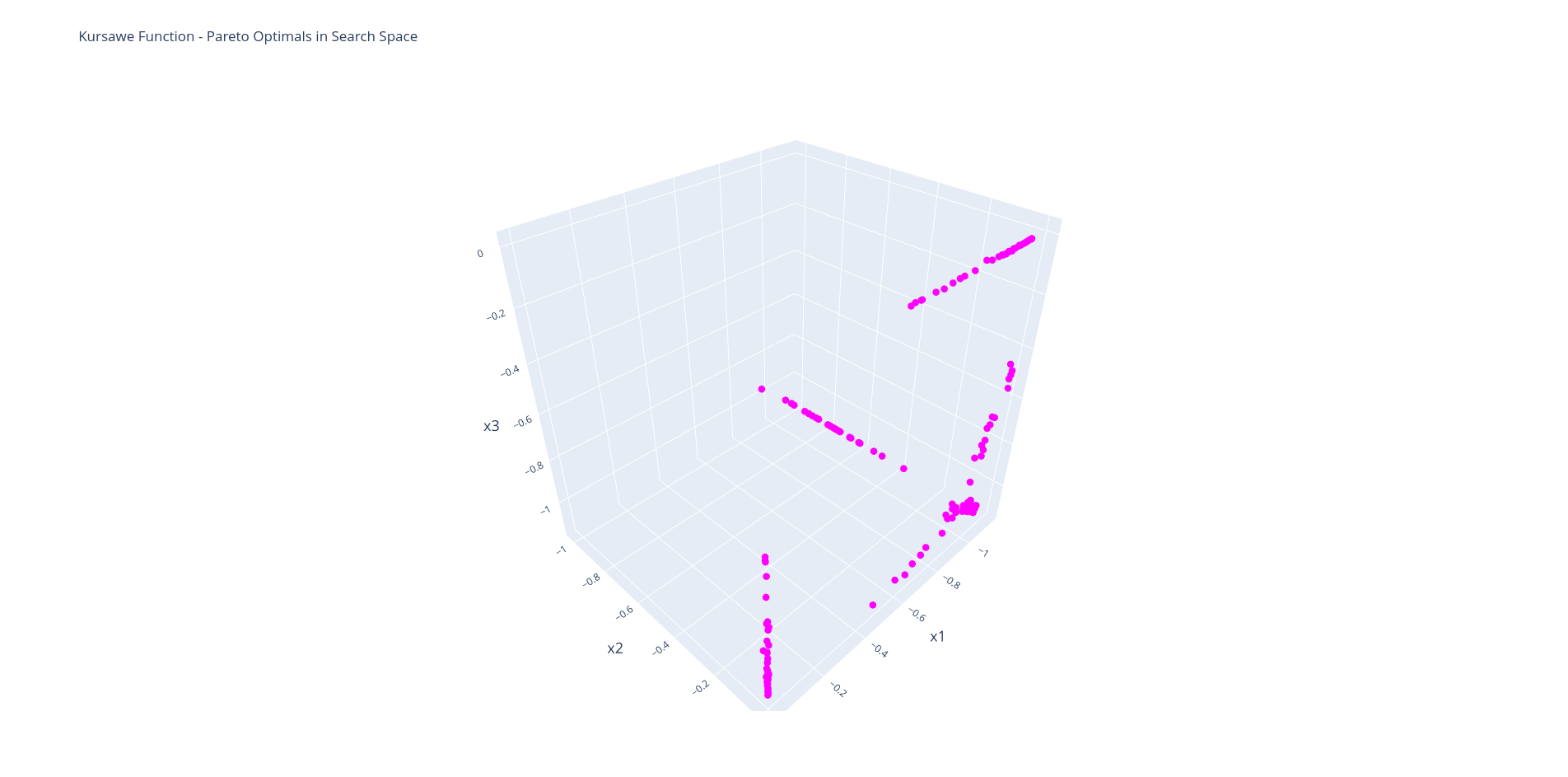}}
    \subcaptionbox{\ref{eq:seq_bcktrck_theory}-\ref{eq:fliege2000}}{\includegraphics[trim={20cm 2cm 20cm 6cm},clip,width=0.4\textwidth]{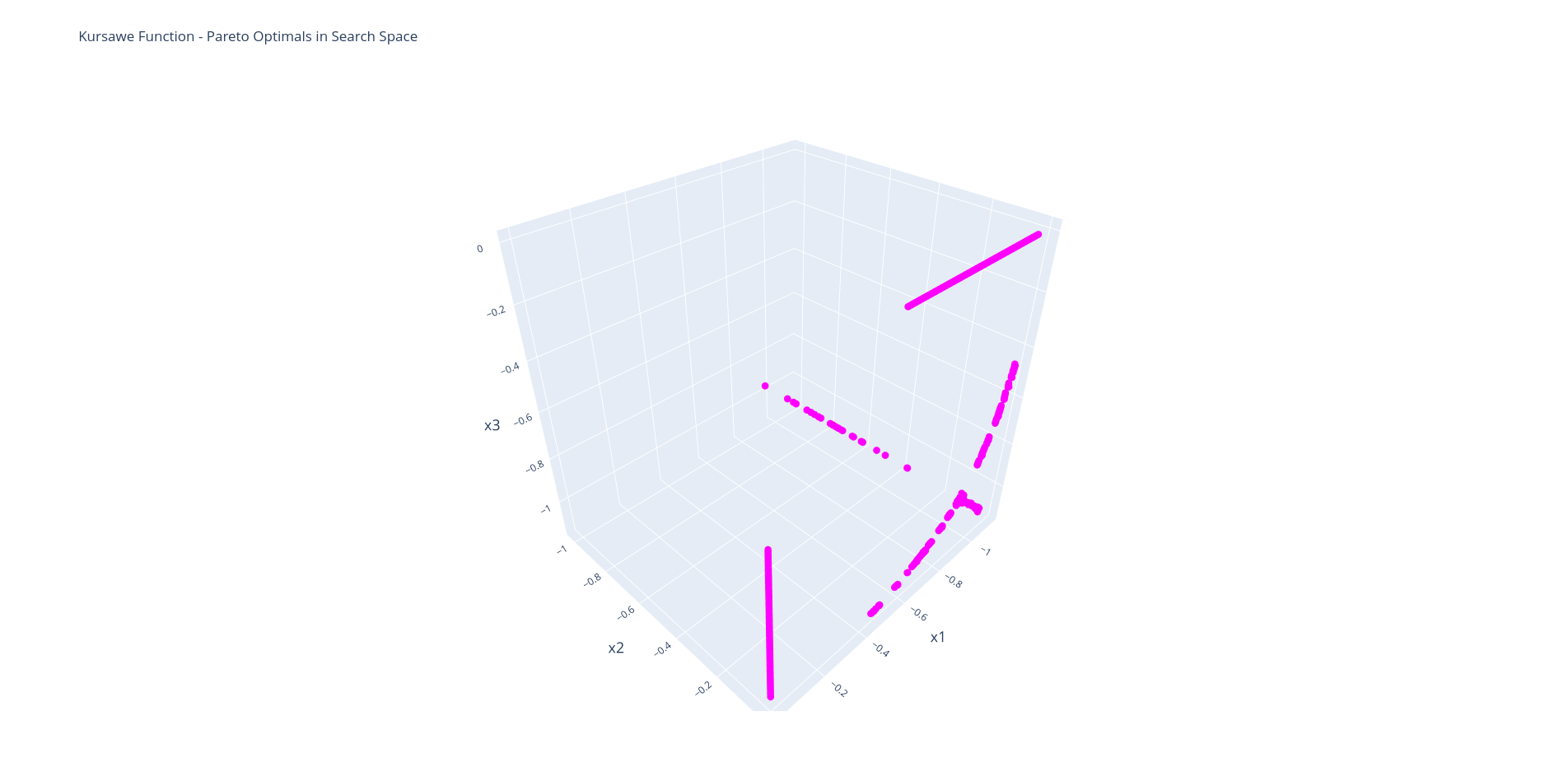}}
    \subcaptionbox{BT$_{\rm base}$-\ref{eq:my_LP_explicit}}{\includegraphics[trim={20cm 2cm 20cm 6cm},clip,width=0.4\textwidth]{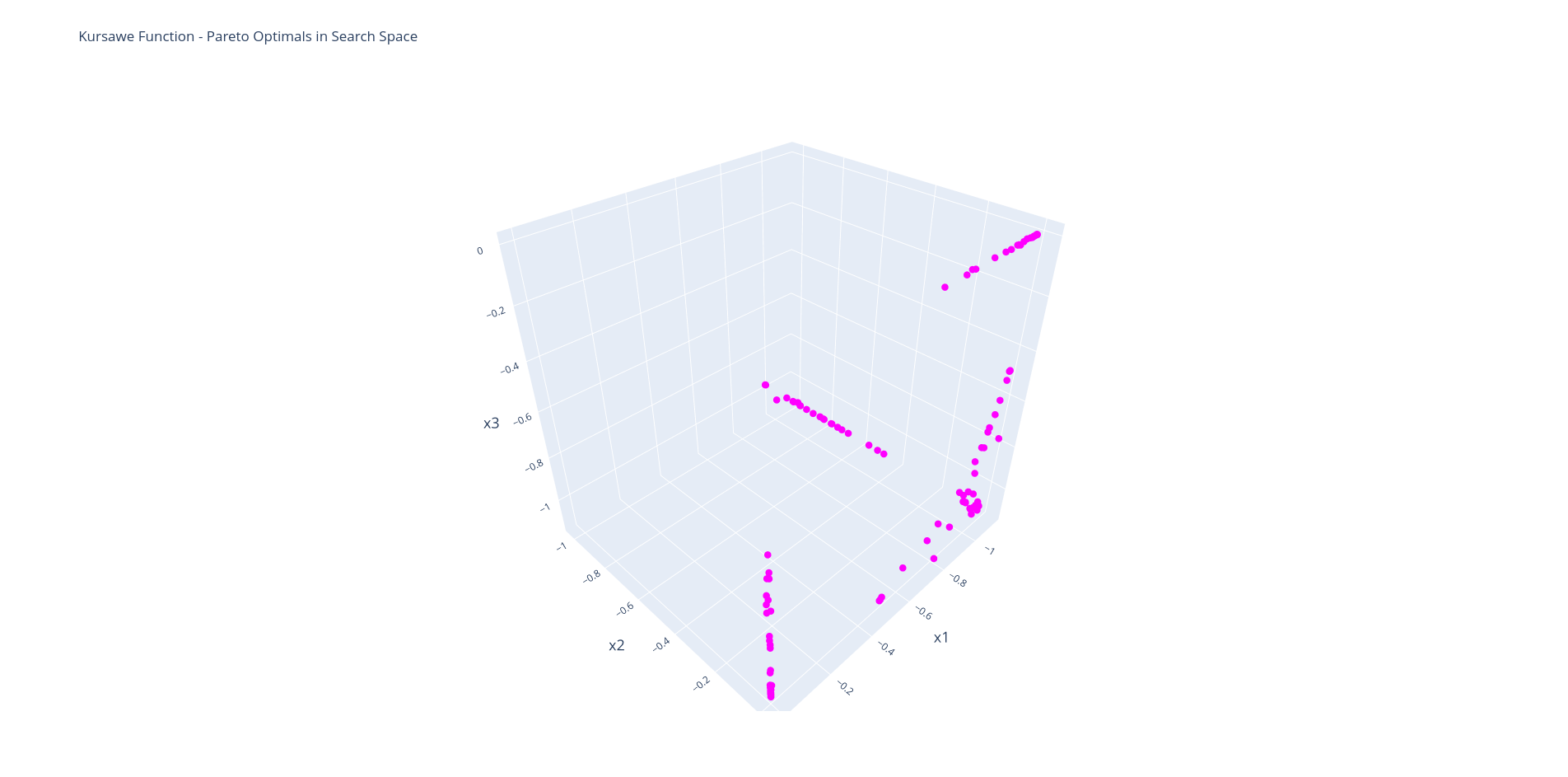}}
    \subcaptionbox{\ref{eq:seq_bcktrck_theory}-\ref{eq:my_LP_explicit}}{\includegraphics[trim={20cm 2cm 20cm 6cm},clip,width=0.4\textwidth]{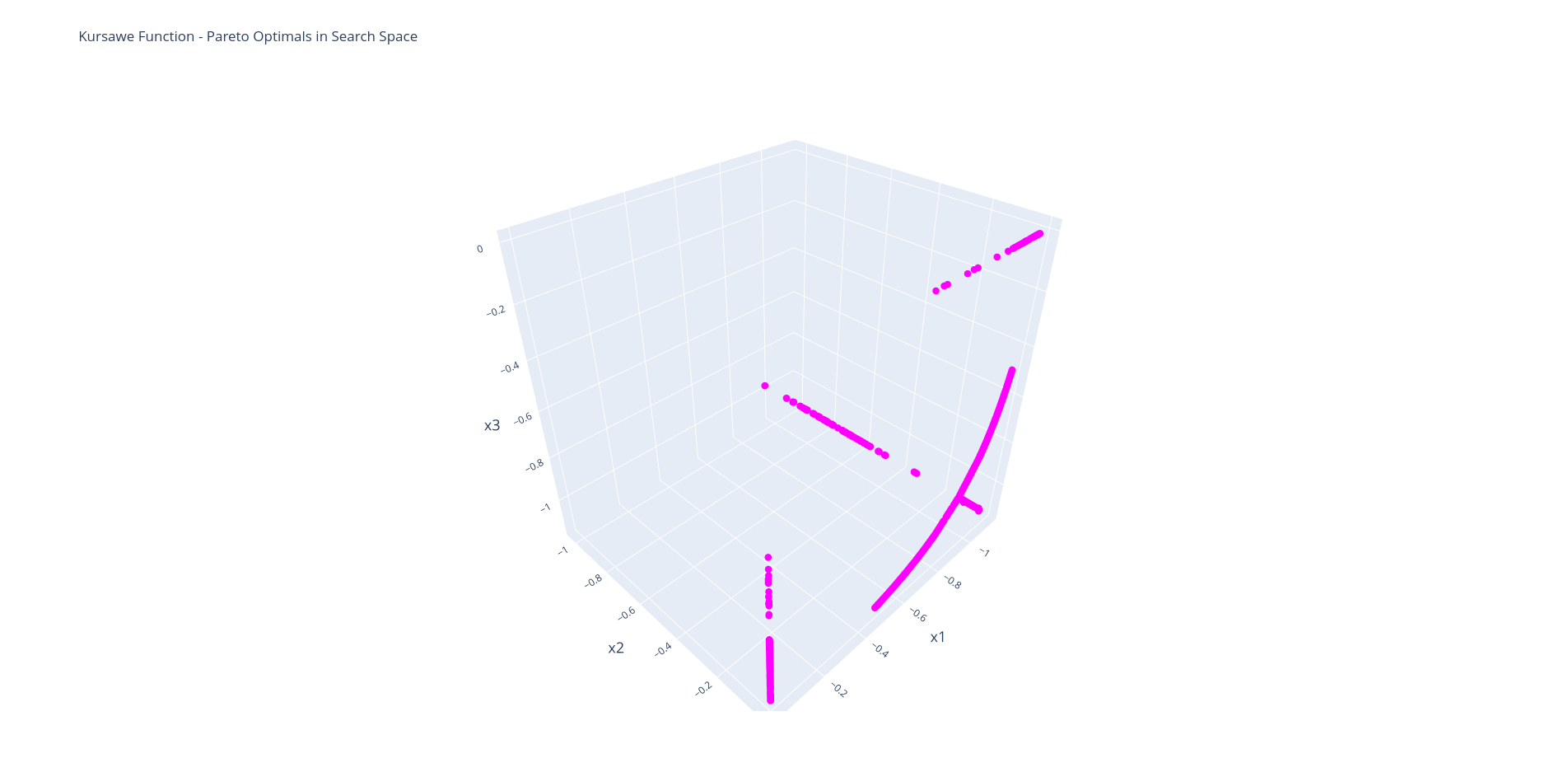}}
    \caption{Kursawe test case. Non-dominated points among all the outputs returned by the MGD algorithms; i.e., the points in $\bigcup_{j=1}^N \widetilde{C}_j$ (see \eqref{eq:nondom_j}).}
    \label{fig:K_set}
\end{figure}

\subsubsection{Viennet}

The last problem we consider is the Viennet problem. This problem is particularly difficult for MGD methods, because it is characterized by regions where the gradient $\v{g}_1$ of the first objective function is parallel and opposite to the gradient $\v{g}_3$ of the third objective function (i.e., $\bar{\v{g}}_1=-\bar{\v{g}}_3$, see \Cref{fig:V_dom_critregions}). Therefore, all these points are Pareto critical, and shared descent directions do not exist for the objectives in those points; at most, there are non-ascent directions perpendicular to $\v{g}_1$ and $\v{g_3}$ that are descent directions for the second objective function (e.g., see \Cref{fig:comparison_lemmacases}c).

\begin{figure}[htb!]
    \centering
    \includegraphics[width=0.5\textwidth]{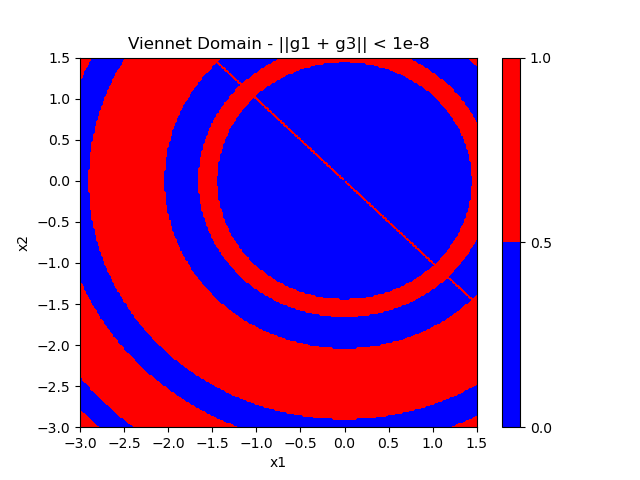}
    \caption{Viennet test case. The red color denotes the points in the square $[-3,1.5]^2$ (see \ref{sec:moo_test_probs}) where the normalized gradient of the first objective function is (almost) equal to minus the normalized gradient of the third objective function (i.e., $\norm{\bar{\v{g}}_1 + \bar{\v{g}}_3}_2 < 10^{-8}$); blue color, otherwise. In particular, red regions are regions of critical points of the same type illustrated in \Cref{fig:comparison_lemmacases}c.}
    \label{fig:V_dom_critregions}
\end{figure}

Under these circumstances, we clearly see how the pairing of \eqref{eq:seq_bcktrck_theory} with \eqref{eq:my_LP_explicit} is crucial for increasing the possibility of reaching the Pareto set/front of the problem, starting from a random point. Indeed, looking at the values of \Cref{tab:global_pareto_percentage}, we see that only the MGD method \ref{eq:seq_bcktrck_theory}-\ref{eq:my_LP_explicit} reaches a good $\mathrm{P}^N$ value of $92.80\%$; all the other MGD methods reaches a value of $\mathrm{P}^N$ that is approximately between $35\%$ and $45\%$. Nonetheless, some of the observations made for the Fonseca-Fleming and the Kursawe problems still hold. Specifically, we see that \eqref{eq:seq_bcktrck_theory} improves the performances in general, while the usage of \eqref{eq:fliege2000} or \eqref{eq:my_LP_explicit} generates directions with different characteristics that, using \Cref{alg:mgd_classic} as backtracking strategy, does not result in consistent differences in the performances. We can see the effects of these observations looking also at the approximated Pareto front/set in \Cref{fig:V_front} and \Cref{fig:V_set}, respectively.

The performance improvements obtained using \eqref{eq:seq_bcktrck_theory} depend on the fact that this backtracking strategy helps to push the sequences along and/or beyond the ``static regions'' of Pareto critical points (see \Cref{fig:V_dom_critregions}). This behavior is evident looking at the paths made of orange dots in Figures \ref{fig:V_obj}b and \ref{fig:V_obj}d (objectives' space) and Figures \ref{fig:V_dom}b and \ref{fig:V_dom}d (domain); in particular, see how the paths in Figures \ref{fig:V_dom}b and \ref{fig:V_dom}d corresponds to the red regions in \Cref{fig:V_dom_critregions}. On the other hand, since BT$_{\rm base}$ cannot help to overcome this kind of difficulty, instead of paths made of orange dots, for these regions we observe ``frozen'' red dots (i.e., sequences that stops at the beginning); see Figures \ref{fig:V_obj}a and \ref{fig:V_obj}c (objectives' space) and Figures \ref{fig:V_dom}a and \ref{fig:V_dom}c (domain). We observed a similar phenomenon in the Kursawe problem too, but less explicit.

Moreover, it is evident from \Cref{tab:global_pareto_percentage} that the best performances are obtained when we use \eqref{eq:seq_bcktrck_theory} and the sequences are built with respect to \eqref{eq:my_LP_explicit}. The reasons for these better performances depend on the different properties of the solutions returned by \eqref{eq:my_LP_explicit} and the solutions returned by \eqref{eq:fliege2000}, with respect to the Pareto critical points belonging to the red regions illustrated in \Cref{fig:V_dom_critregions}. Indeed, \eqref{eq:fliege2000} can return a null direction anytime the point of the sequence is Pareto critical (see \Cref{lem:fliege2000}), de-facto stopping the sequence itself; moreover, the infinite-norm of the computed direction is at most $1$, generating steps that are not related to the gradients' order of magnitude (in this case, too small steps). On the other hand, \eqref{eq:my_LP_explicit} returns a non-null direction for sure for all the points in the red regions illustrated in \Cref{fig:V_dom_critregions} where $\v{g}_2\neq\v{0}$, because of \Cref{lem:my_LP} - item 2.3; moreover, the upper bound of the infinite-norm of the computed direction depends on the infinite-norms of the gradients of the objective functions (see \eqref{eq:gamma_LP}), granting longer steps when objectives are steeper. From these observations, we deduce that \eqref{eq:my_LP_explicit}, together with \eqref{eq:seq_bcktrck_theory}, is better for a problem like Viennet because it guarantees the building of sequences that are more able to move along/beyond the ``static regions'', increasing the probability of reaching a good approximation of a global Pareto optimal. See how the paths generated by orange/critical Points in \Cref{fig:V_dom}d (but also \Cref{fig:V_dom}b) follow the circular shapes of the red regions of \Cref{fig:V_dom_critregions}, while in the objectives' space their movement actually results mostly in a reduction of the second objective function (see \Cref{fig:V_obj}d, but also \Cref{fig:V_obj}b).

\begin{figure}[htb!]
    \centering
    \subcaptionbox{BT$_{\rm base}$-\ref{eq:fliege2000}}{\includegraphics[trim={19cm 2cm 20cm 6cm},clip,width=0.4\textwidth]{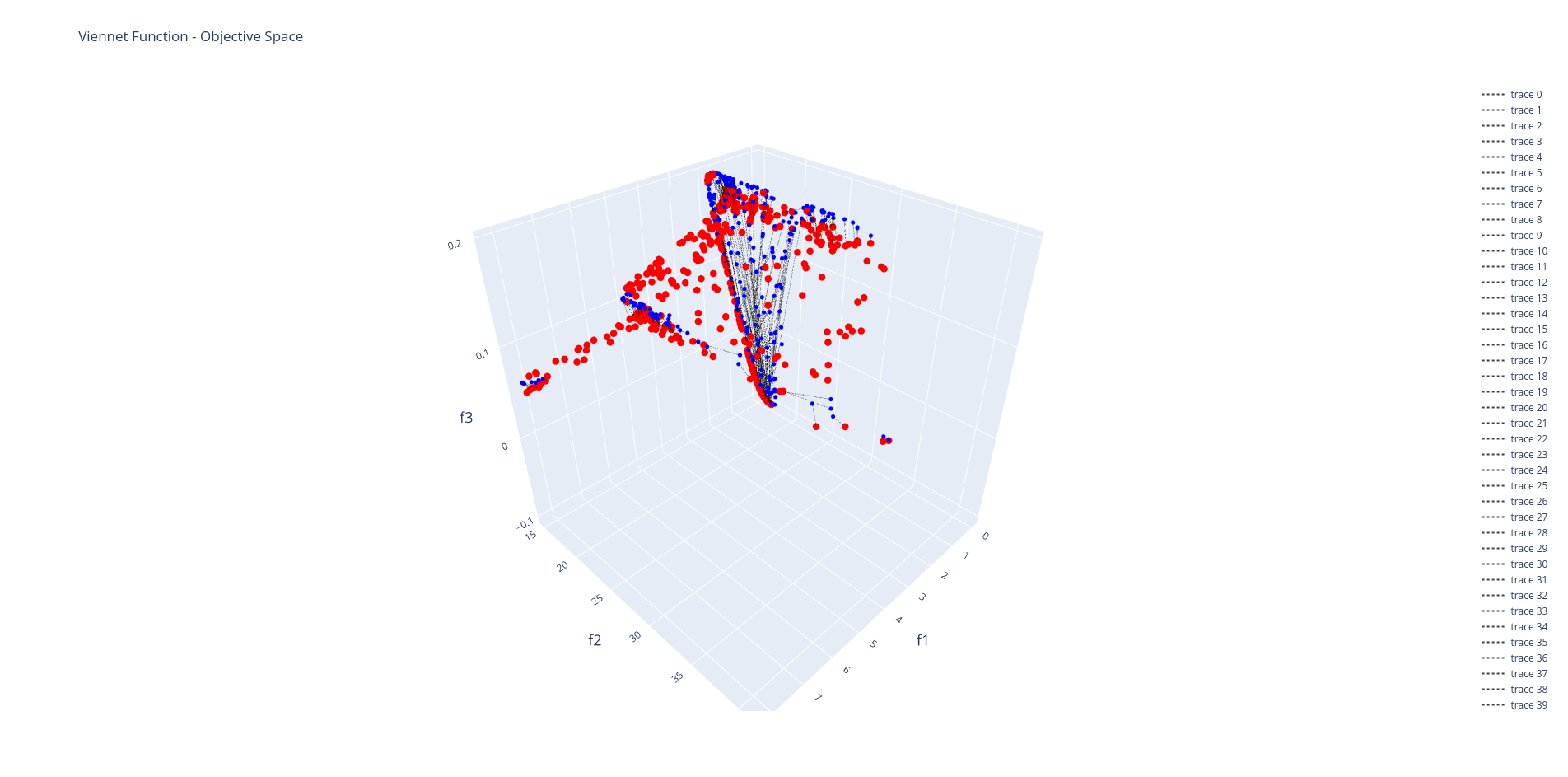}}
    \subcaptionbox{\ref{eq:seq_bcktrck_theory}-\ref{eq:fliege2000}}{\includegraphics[trim={19cm 2cm 20cm 6cm},clip,width=0.4\textwidth]{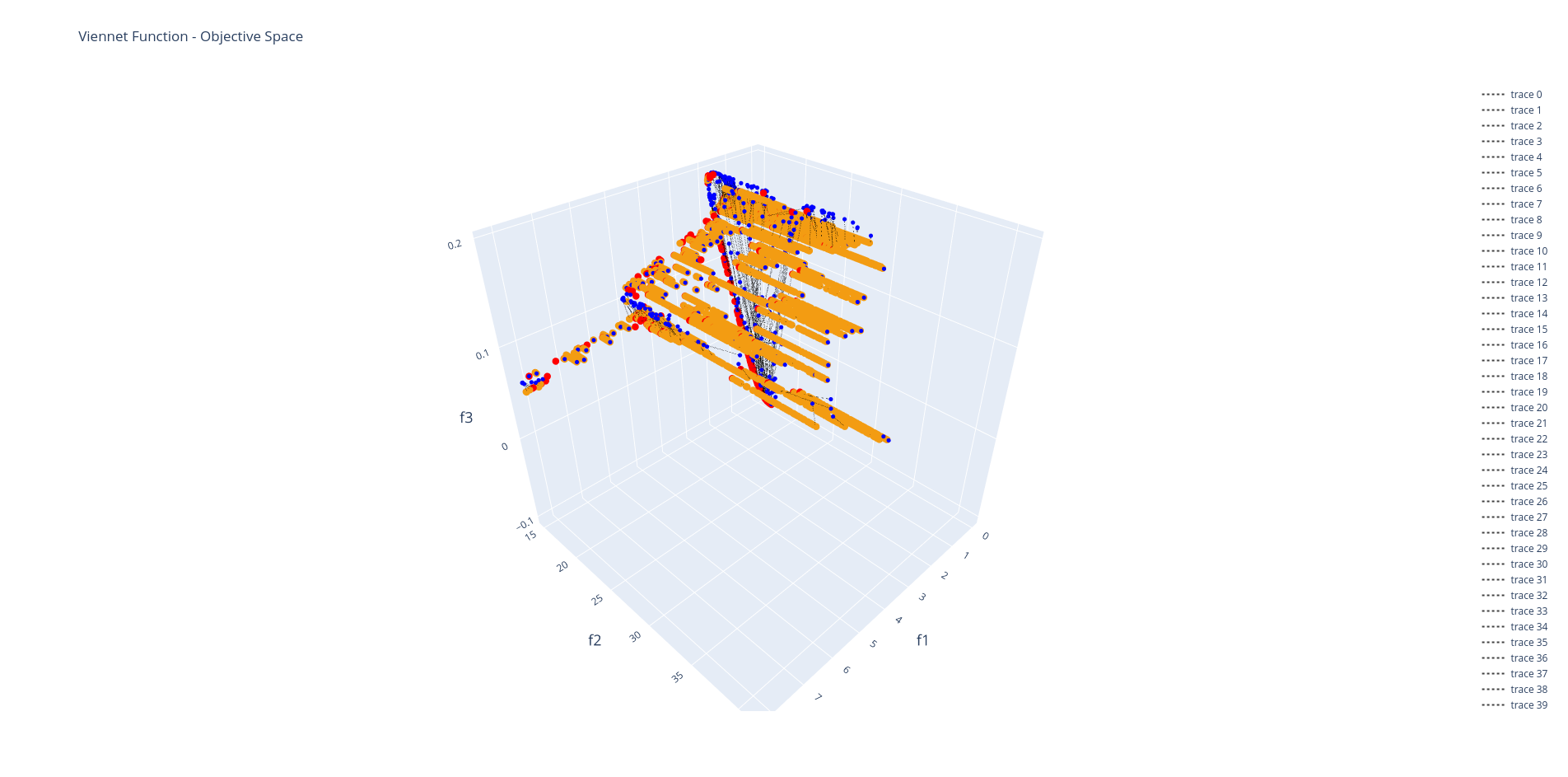}}
    \subcaptionbox{BT$_{\rm base}$-\ref{eq:my_LP_explicit}}{\includegraphics[trim={19cm 2cm 20cm 6cm},clip,width=0.4\textwidth]{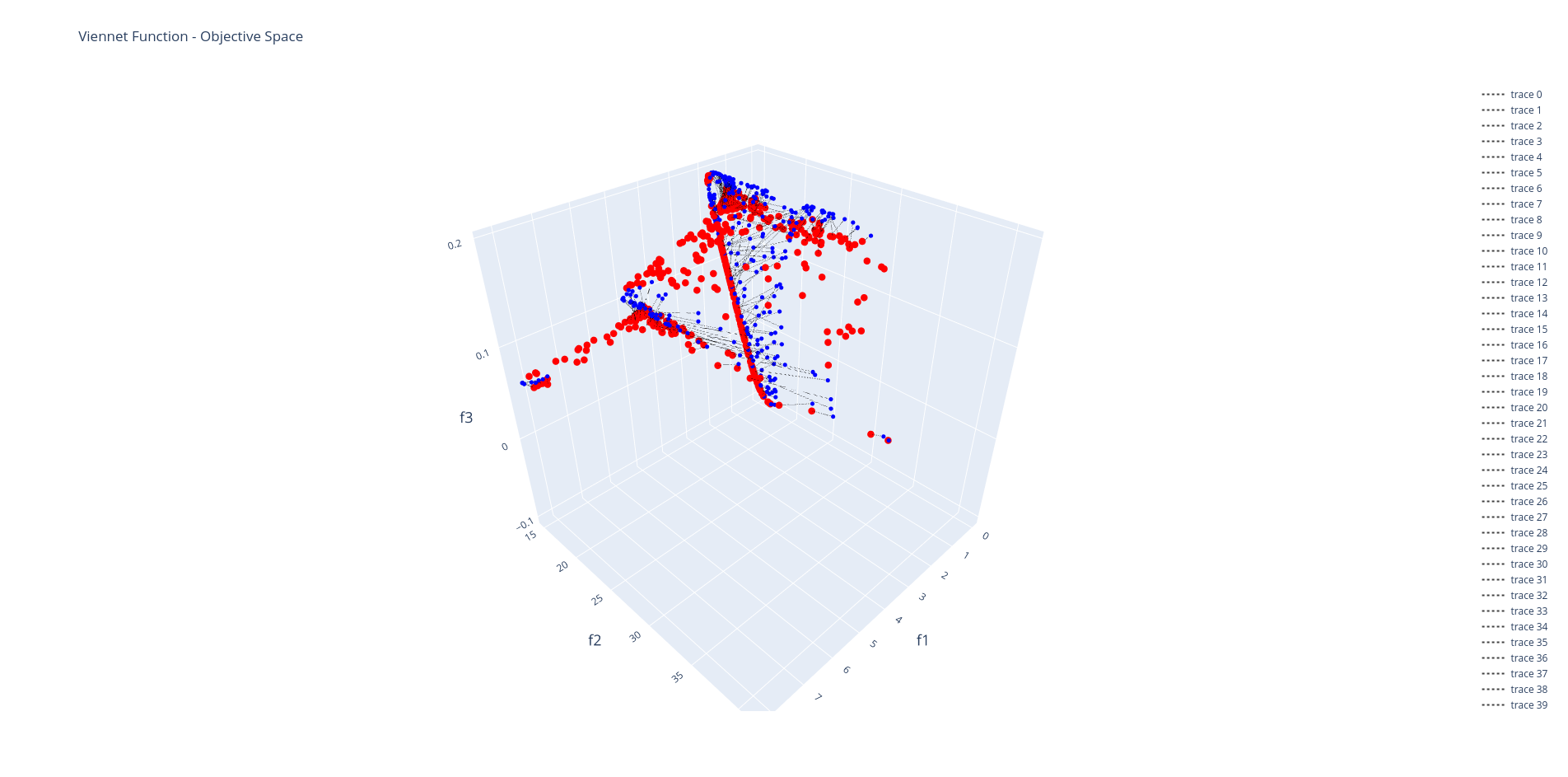}}
    \subcaptionbox{\ref{eq:seq_bcktrck_theory}-\ref{eq:my_LP_explicit}}{\includegraphics[trim={19cm 2cm 20cm 6cm},clip,width=0.4\textwidth]{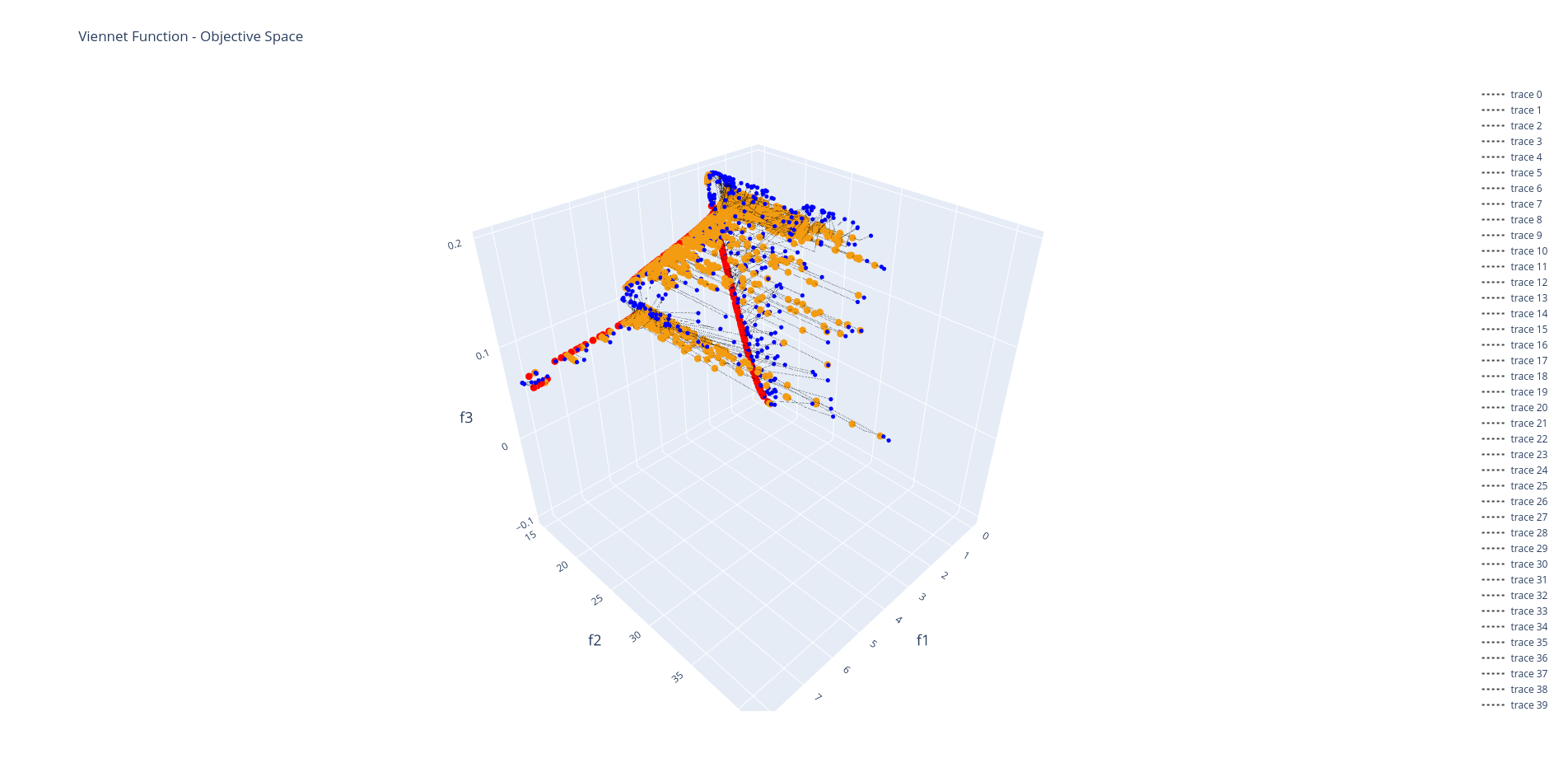}}
    \caption{Viennet test case. Movement of the $N=500$ sequences in the objectives' space. The blue dots are the starting points' images $\v{f}(\v{x}^{(0)}_j)\in\R^3$, $j=1,\ldots ,N$, while the red dots are the last points of the sequence. The orange dots are the critical points stored during \Cref{alg:my_mgd_storage}, before the final pruning (see the pseudocode).
    The black, dotted, and piece-wise linear curves describe the movement of each sequence from its starting point to its last element.}
    \label{fig:V_obj}
\end{figure}

\begin{figure}[htb!]
    \centering
    \subcaptionbox{BT$_{\rm base}$-\ref{eq:fliege2000}}{\includegraphics[trim={19cm 2cm 20cm 6cm},clip,width=0.4\textwidth]{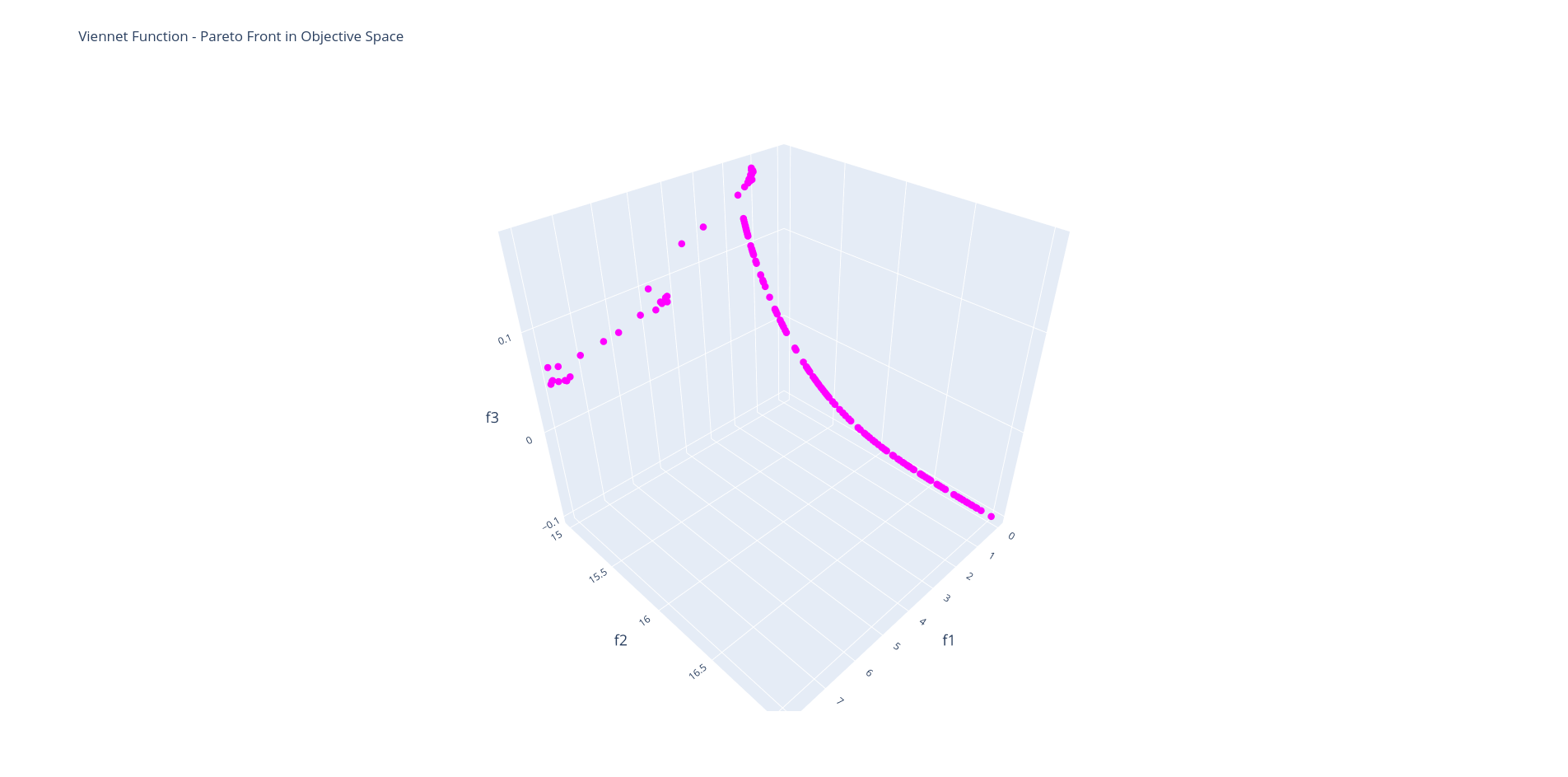}}
    \subcaptionbox{\ref{eq:seq_bcktrck_theory}-\ref{eq:fliege2000}}{\includegraphics[trim={19cm 2cm 20cm 6cm},clip,width=0.4\textwidth]{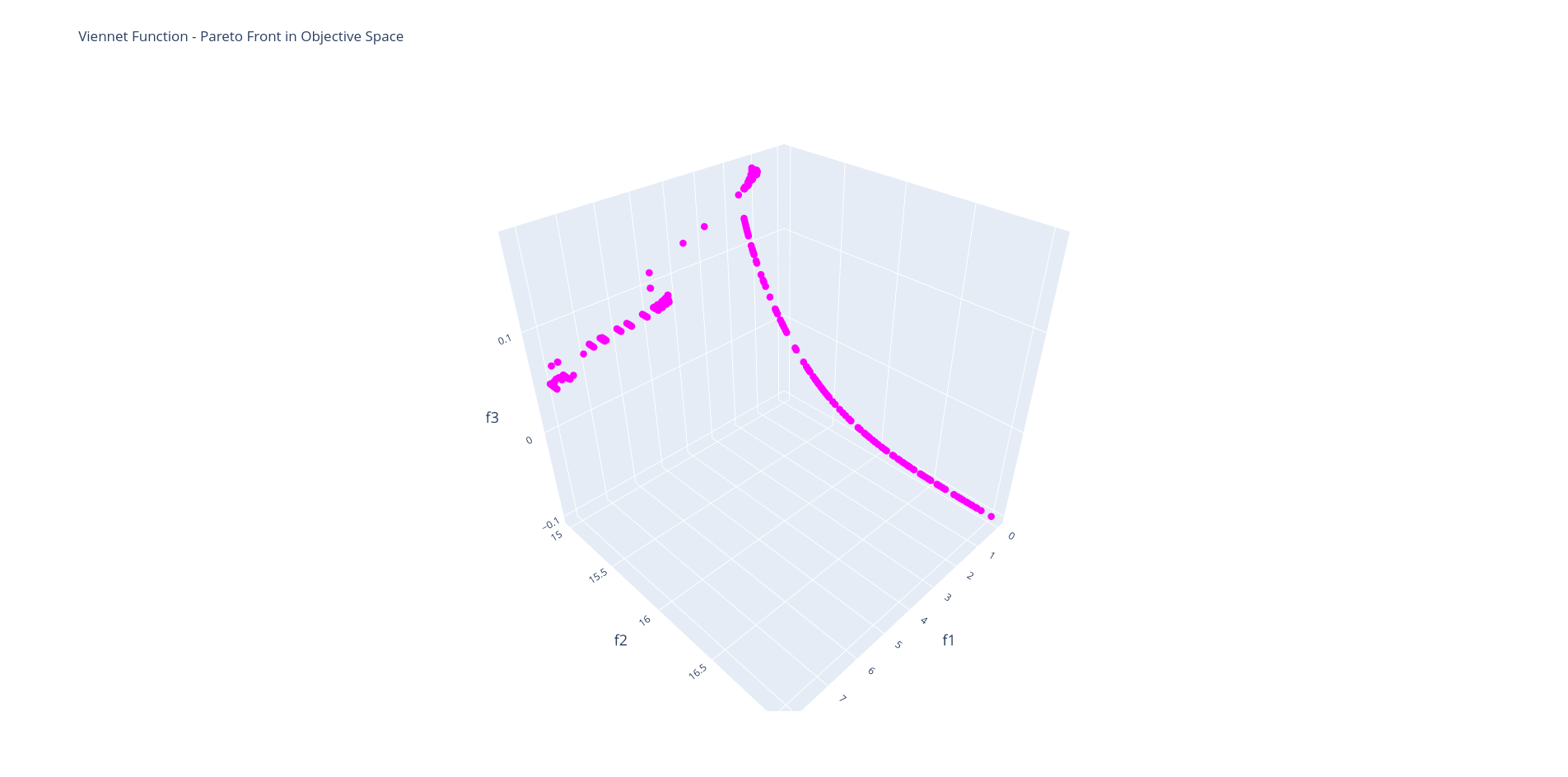}}
    \subcaptionbox{BT$_{\rm base}$-\ref{eq:my_LP_explicit}}{\includegraphics[trim={19cm 2cm 20cm 6cm},clip,width=0.4\textwidth]{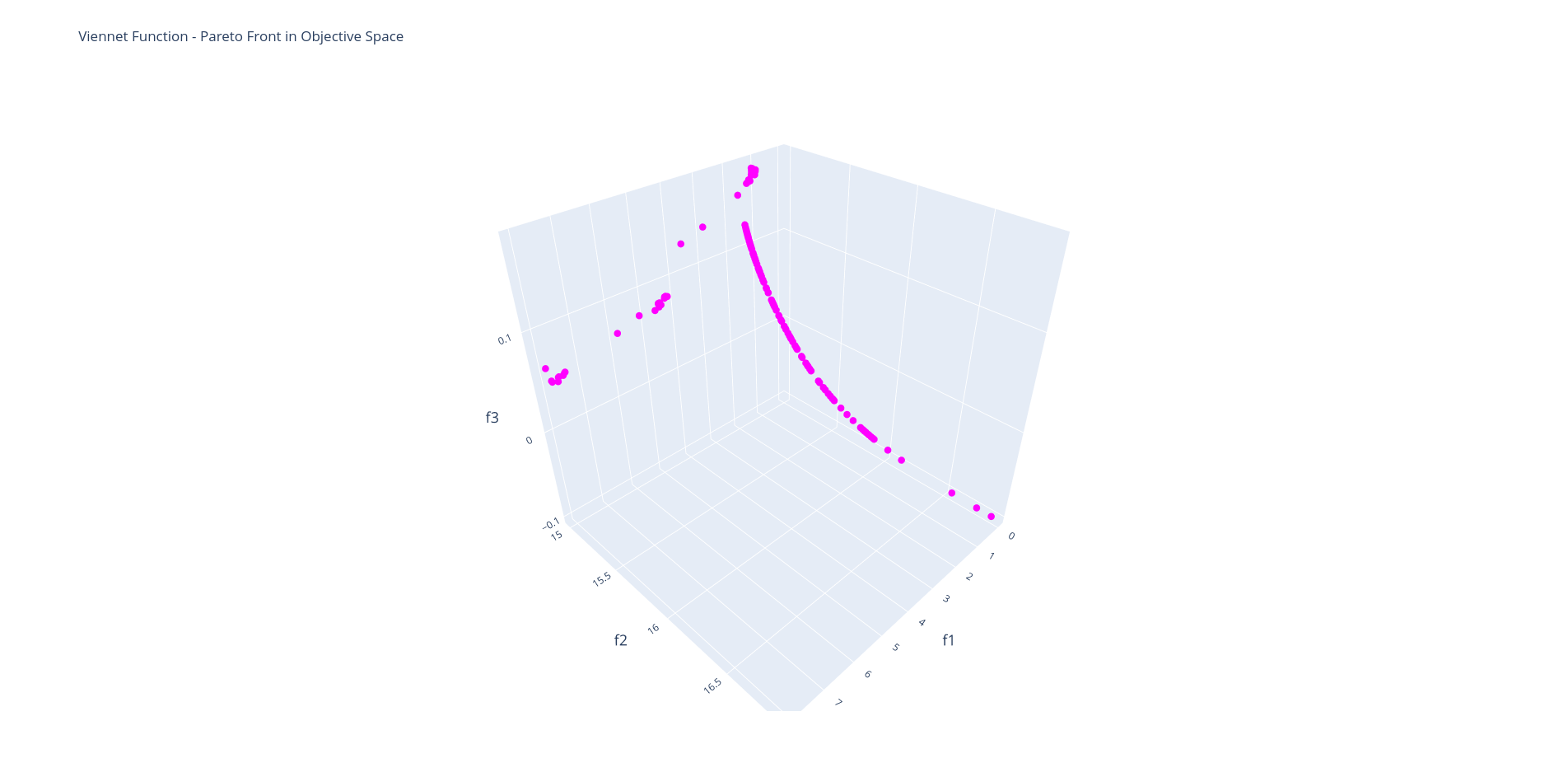}}
    \subcaptionbox{\ref{eq:seq_bcktrck_theory}-\ref{eq:my_LP_explicit}}{\includegraphics[trim={19cm 2cm 20cm 6cm},clip,width=0.4\textwidth]{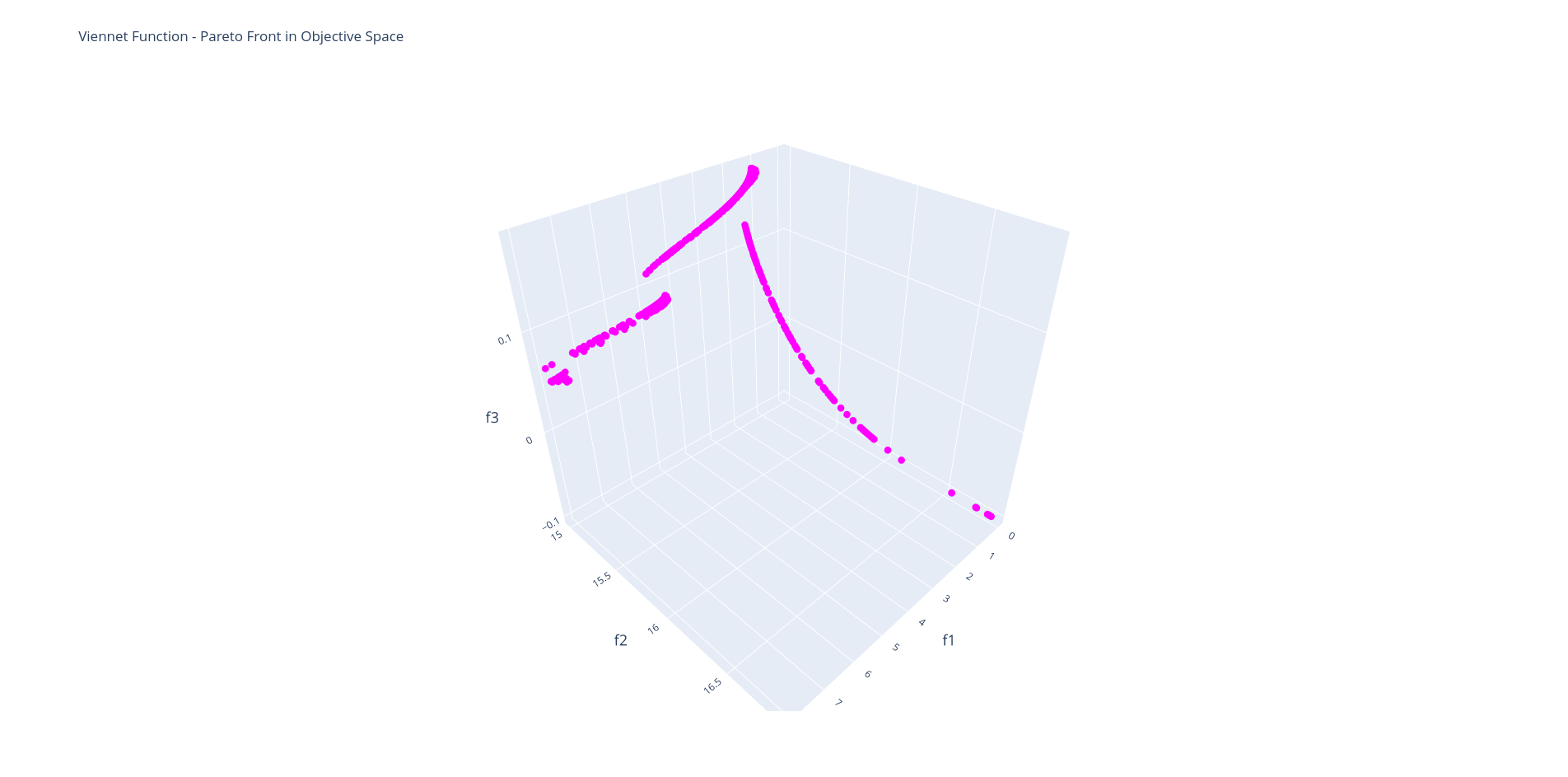}}
    \caption{Viennet test case. Images in the objectves' space of the non-dominated points among all the outputs returned by the MGD algorithms; i.e., images of the points in $\bigcup_{j=1}^N \widetilde{C}_j$ (see \eqref{eq:nondom_j}).}
    \label{fig:V_front}
\end{figure}

\begin{figure}[htb!]
    \centering
    \subcaptionbox{BT$_{\rm base}$-\ref{eq:fliege2000}}{\includegraphics[trim={0.5cm 1cm 5.5cm 3.5cm},clip,width=0.49\textwidth,height=0.165\textheight]{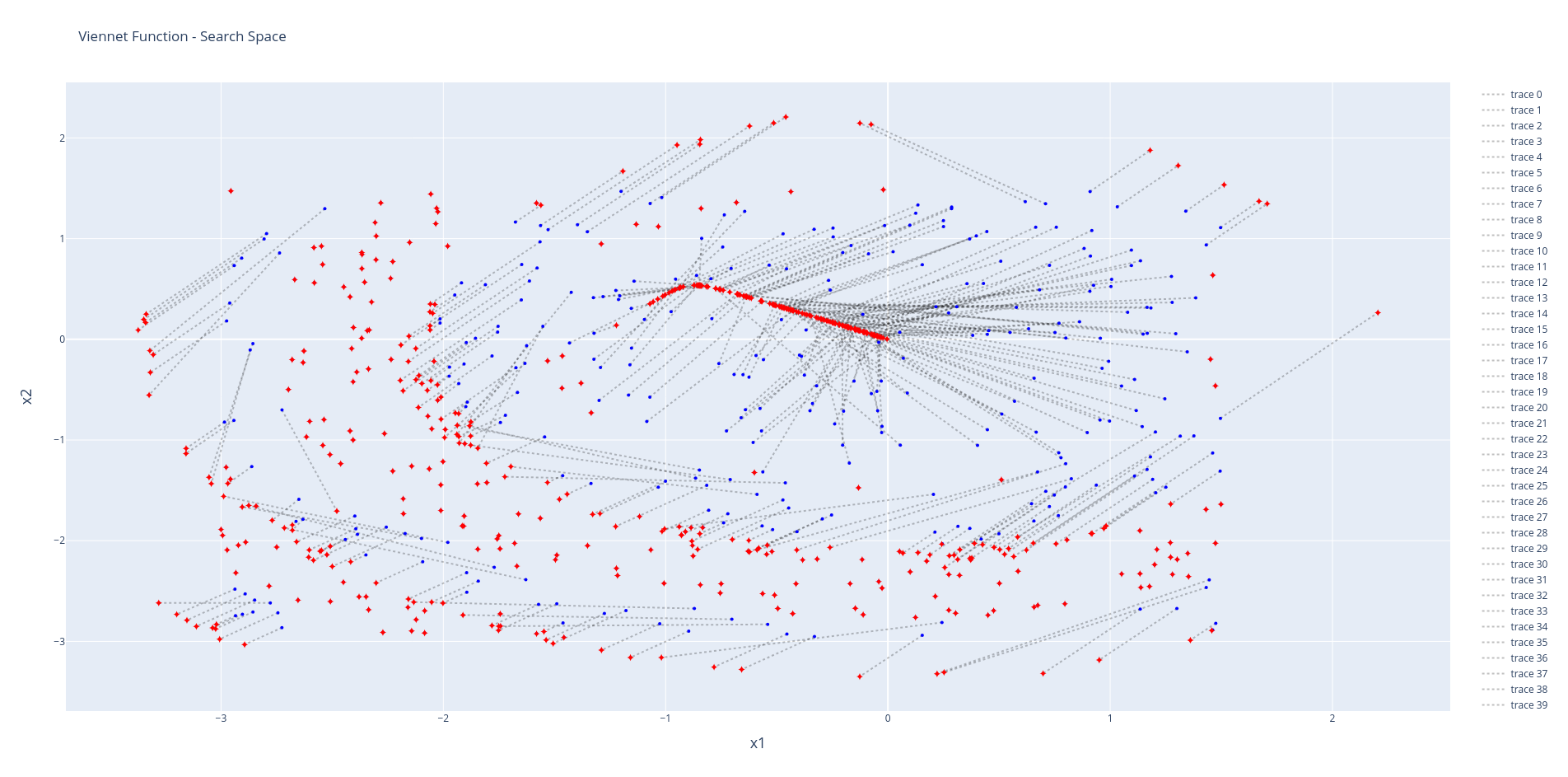}}
    \subcaptionbox{\ref{eq:seq_bcktrck_theory}-\ref{eq:fliege2000}}{\includegraphics[trim={0.5cm 1cm 5.5cm 3.5cm},clip,width=0.49\textwidth,height=0.165\textheight]{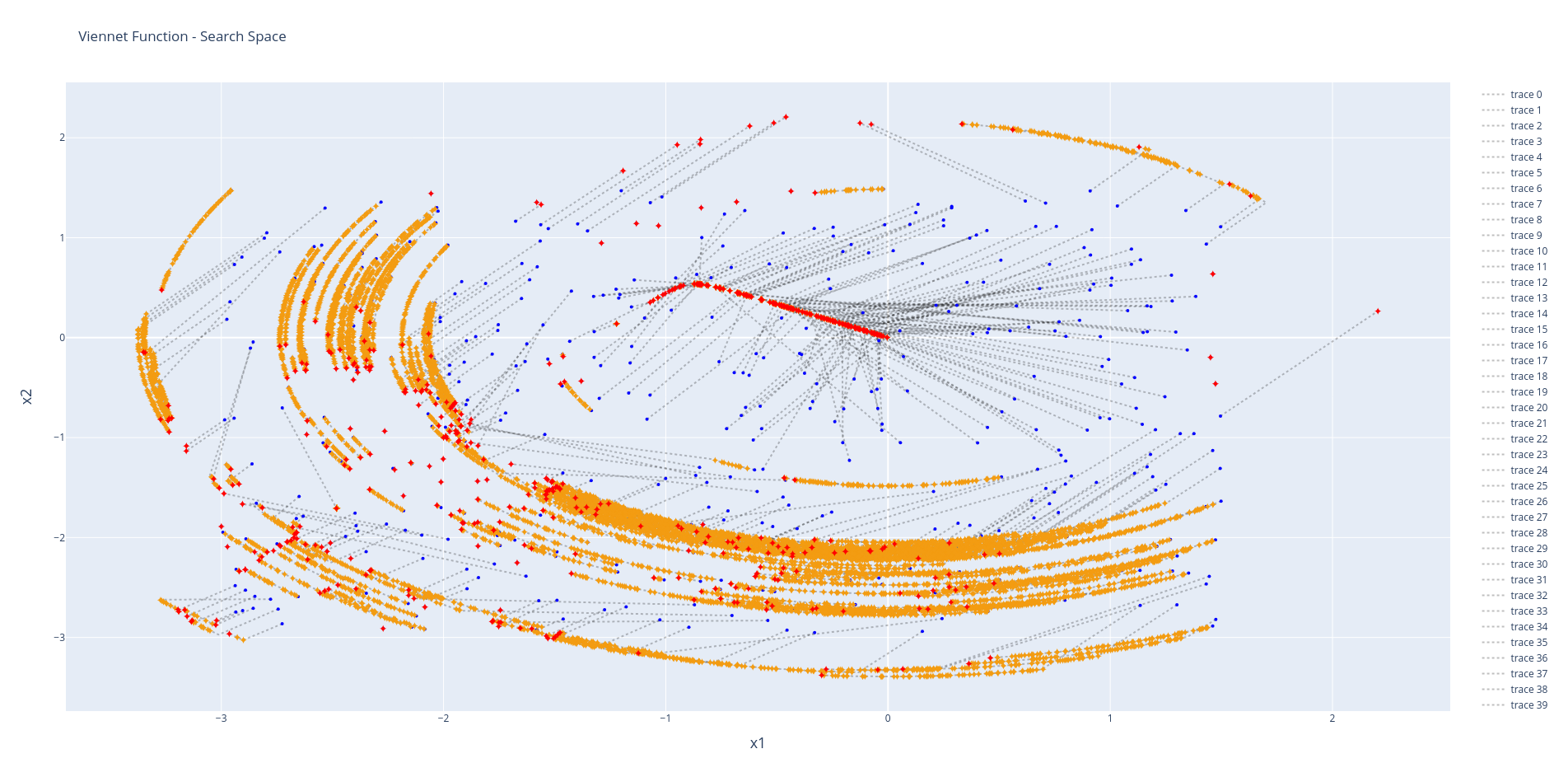}}
    \subcaptionbox{BT$_{\rm base}$-\ref{eq:my_LP_explicit}}{\includegraphics[trim={0.5cm 1cm 5.5cm 3.5cm},clip,width=0.49\textwidth,height=0.165\textheight]{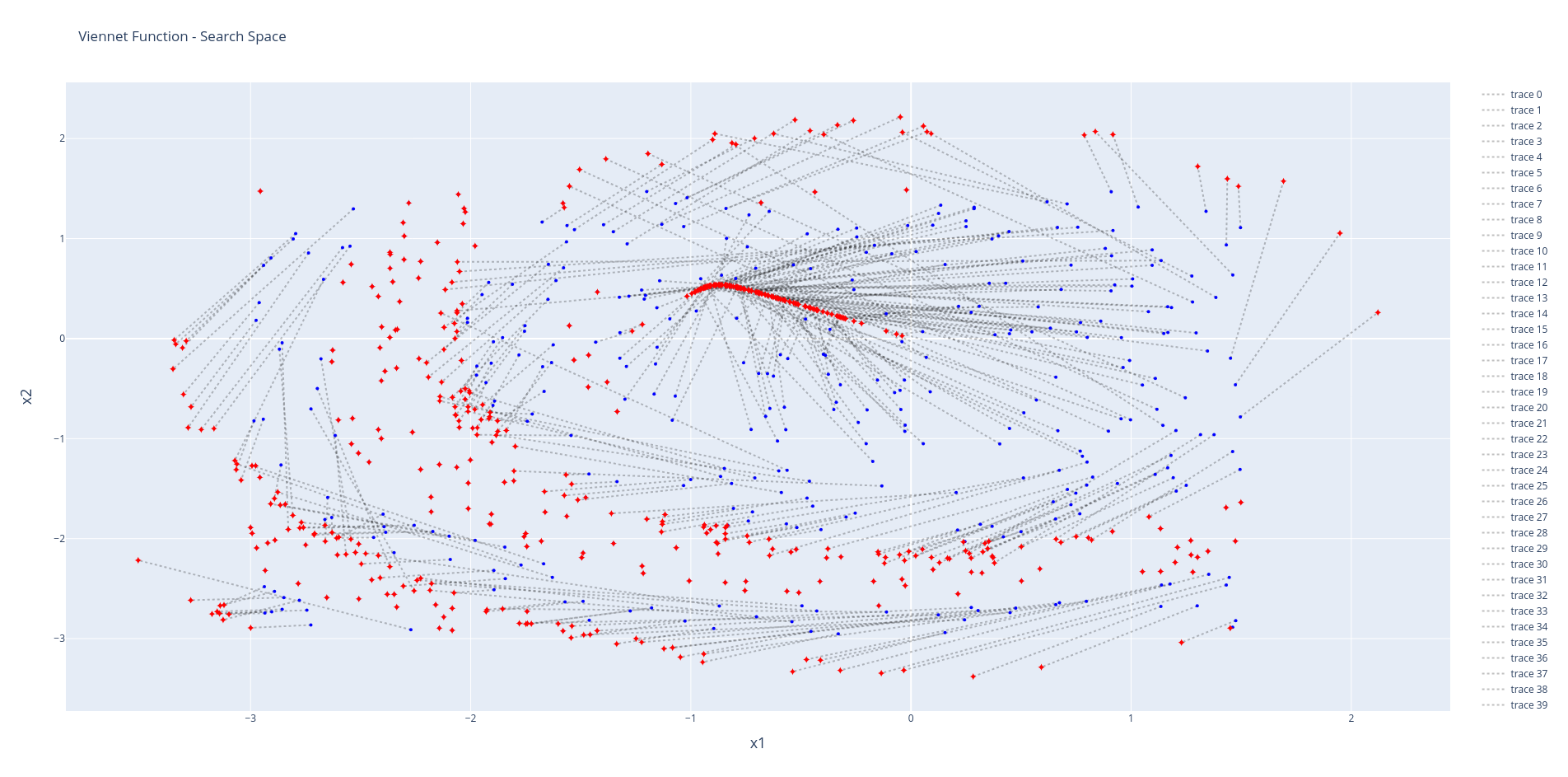}}
    \subcaptionbox{\ref{eq:seq_bcktrck_theory}-\ref{eq:my_LP_explicit}}{\includegraphics[trim={0.5cm 1cm 5.5cm 3.5cm},clip,width=0.49\textwidth,height=0.165\textheight]{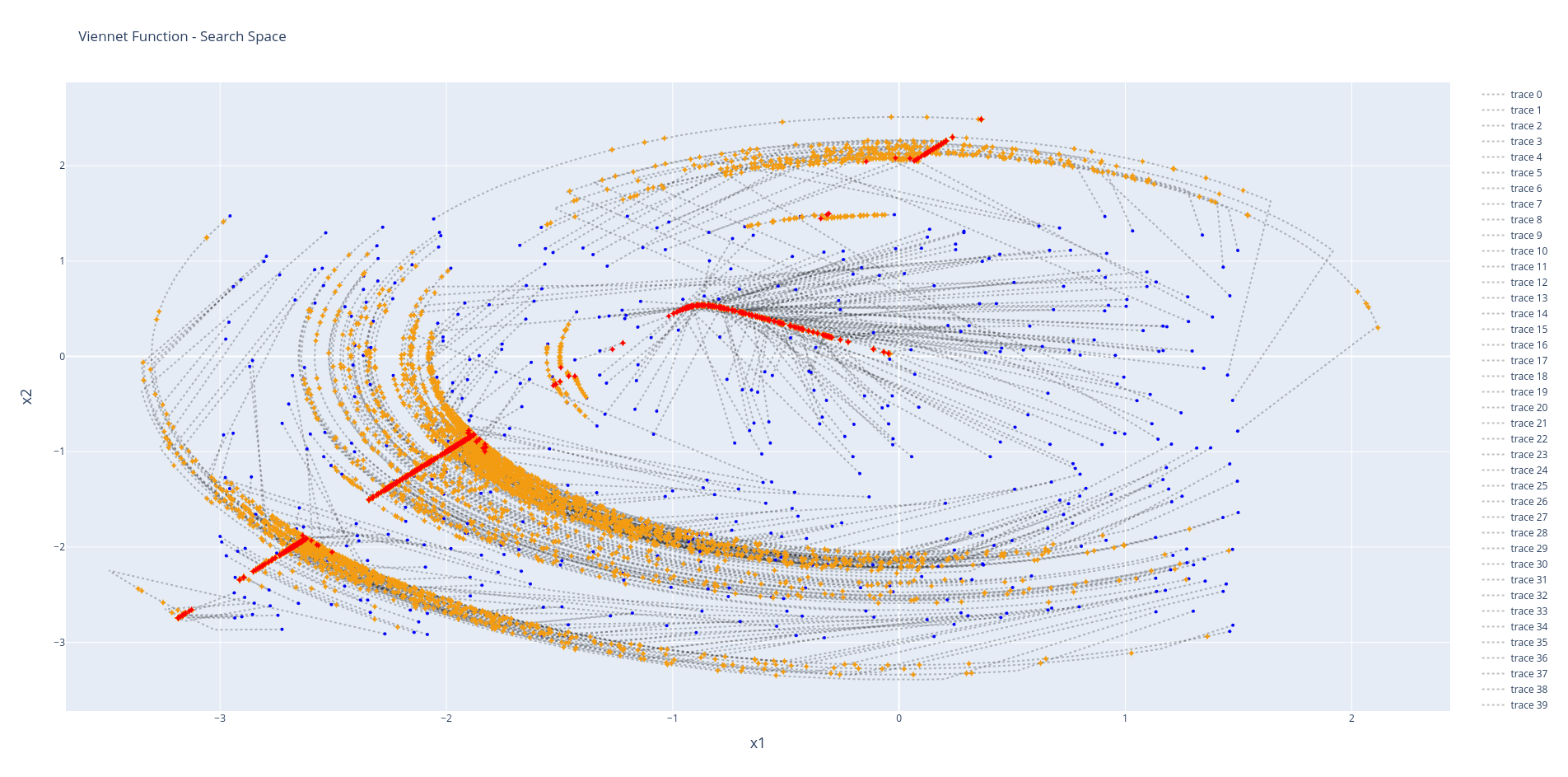}}
    \caption{Viennet test case. Movement of the $N=500$ sequences. The blue dots are the starting points $\v{x}^{(0)}_j\in\R^2$, $j=1,\ldots ,N$, while the red dots are the last points of the sequence. The orange dots are the critical points stored during \Cref{alg:my_mgd_storage}, before the final pruning (see the pseudocode).
    The black, dotted, and piece-wise linear curves describe the movement of each sequence from its starting point to its last element.}
    \label{fig:V_dom}
\end{figure}

\begin{figure}[htb!]
    \centering
    \subcaptionbox{BT$_{\rm base}$-\ref{eq:fliege2000}}{\includegraphics[trim={0.5cm 1cm 5.5cm 3.5cm},clip,width=0.49\textwidth,height=0.165\textheight]{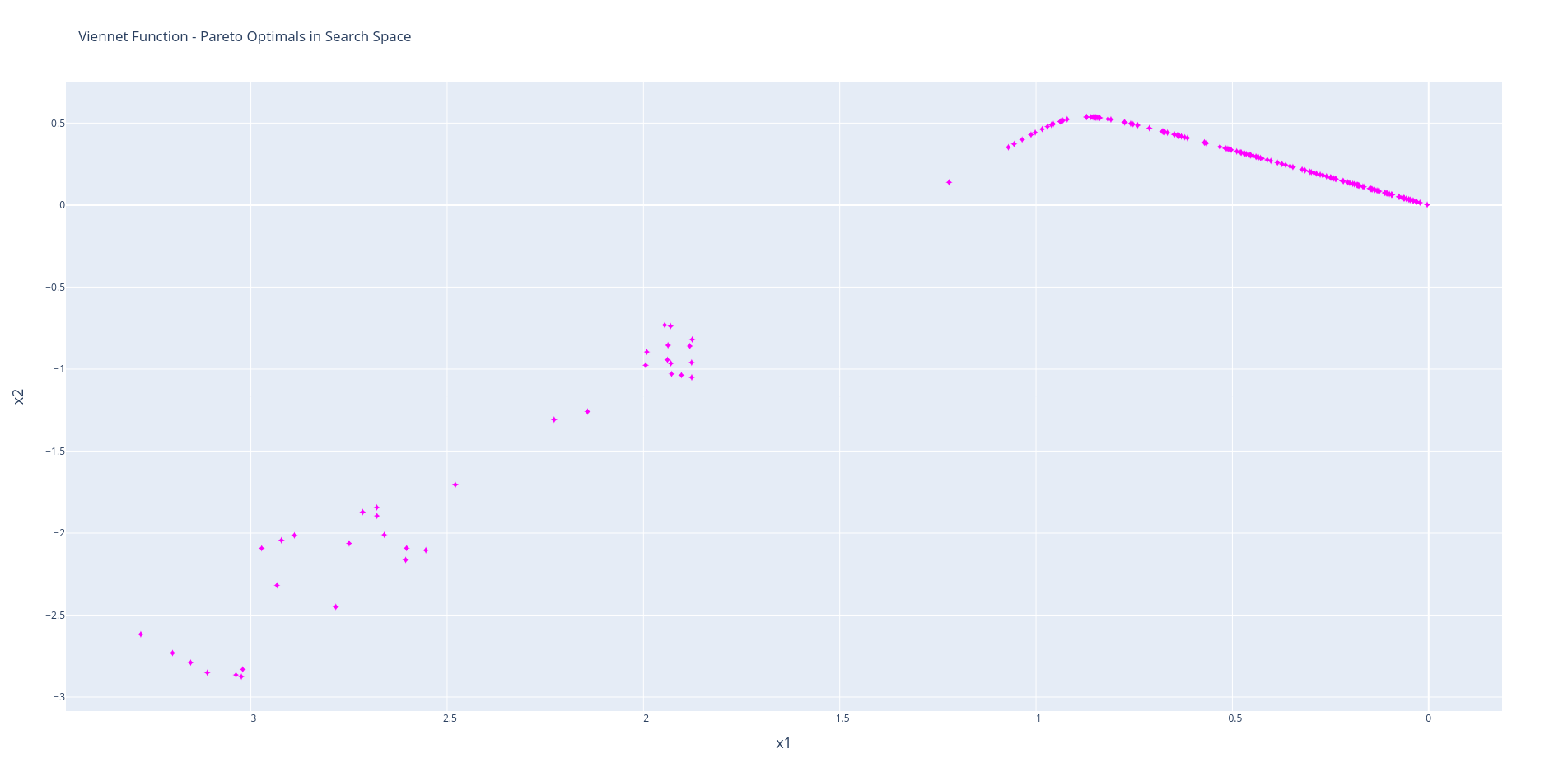}}
    \subcaptionbox{\ref{eq:seq_bcktrck_theory}-\ref{eq:fliege2000}}{\includegraphics[trim={0.5cm 1cm 5.5cm 3.5cm},clip,width=0.49\textwidth,height=0.165\textheight]{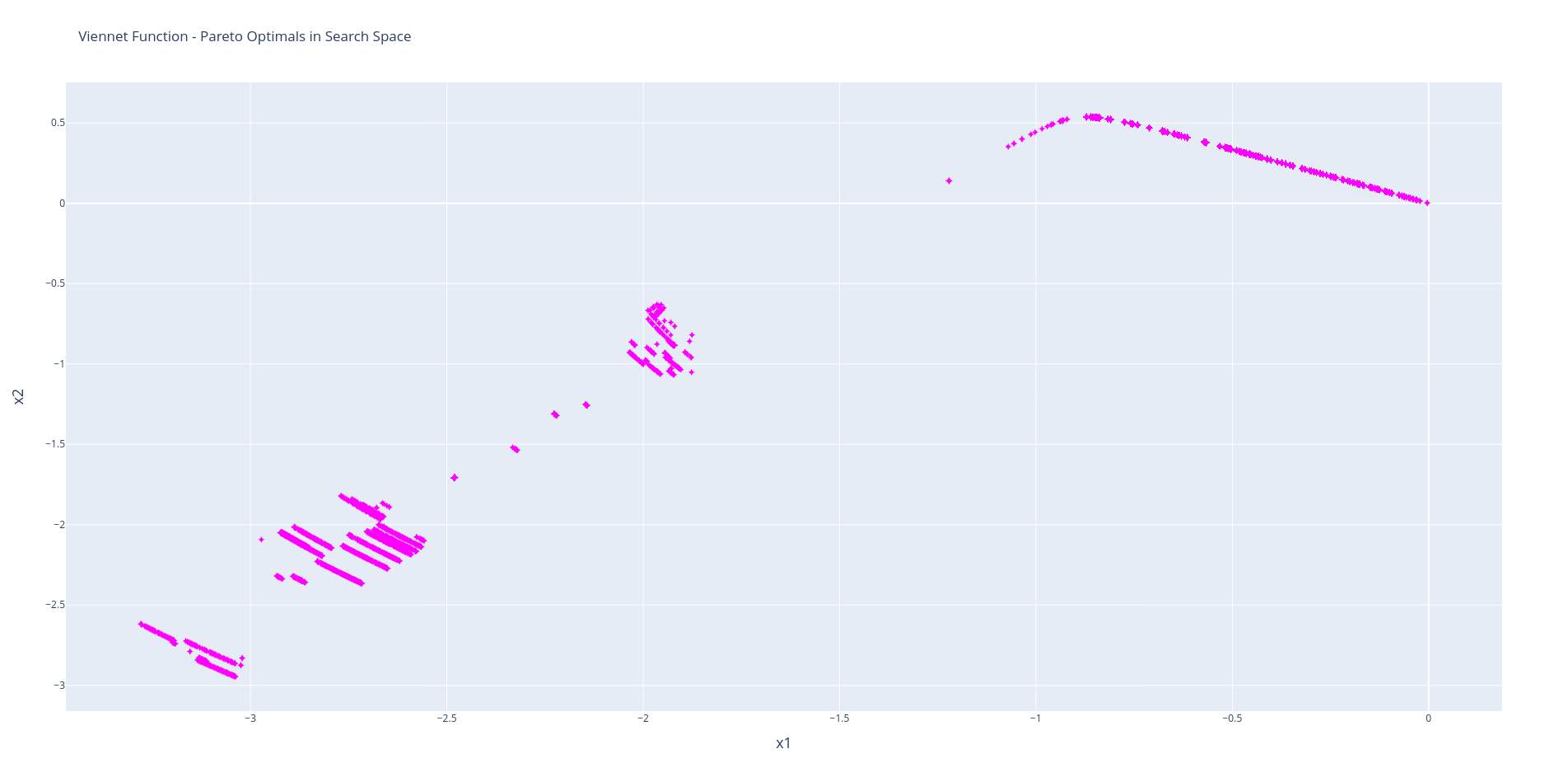}}
    \subcaptionbox{BT$_{\rm base}$-\ref{eq:my_LP_explicit}}{\includegraphics[trim={0.5cm 1cm 5.5cm 3.5cm},clip,width=0.49\textwidth,height=0.165\textheight]{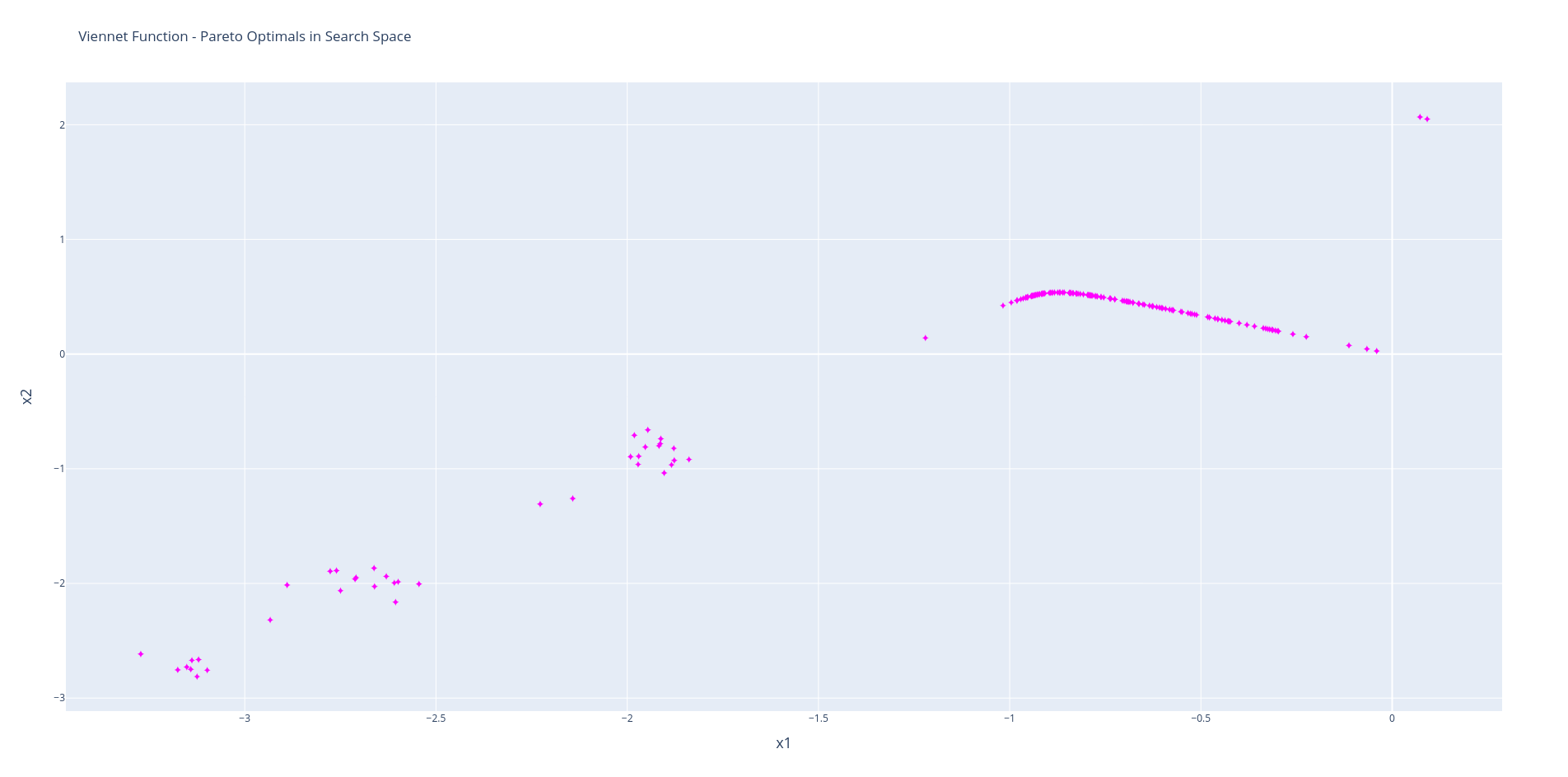}}
    \subcaptionbox{\ref{eq:seq_bcktrck_theory}-\ref{eq:my_LP_explicit}}{\includegraphics[trim={0.5cm 1cm 5.5cm 3.5cm},clip,width=0.49\textwidth,height=0.165\textheight]{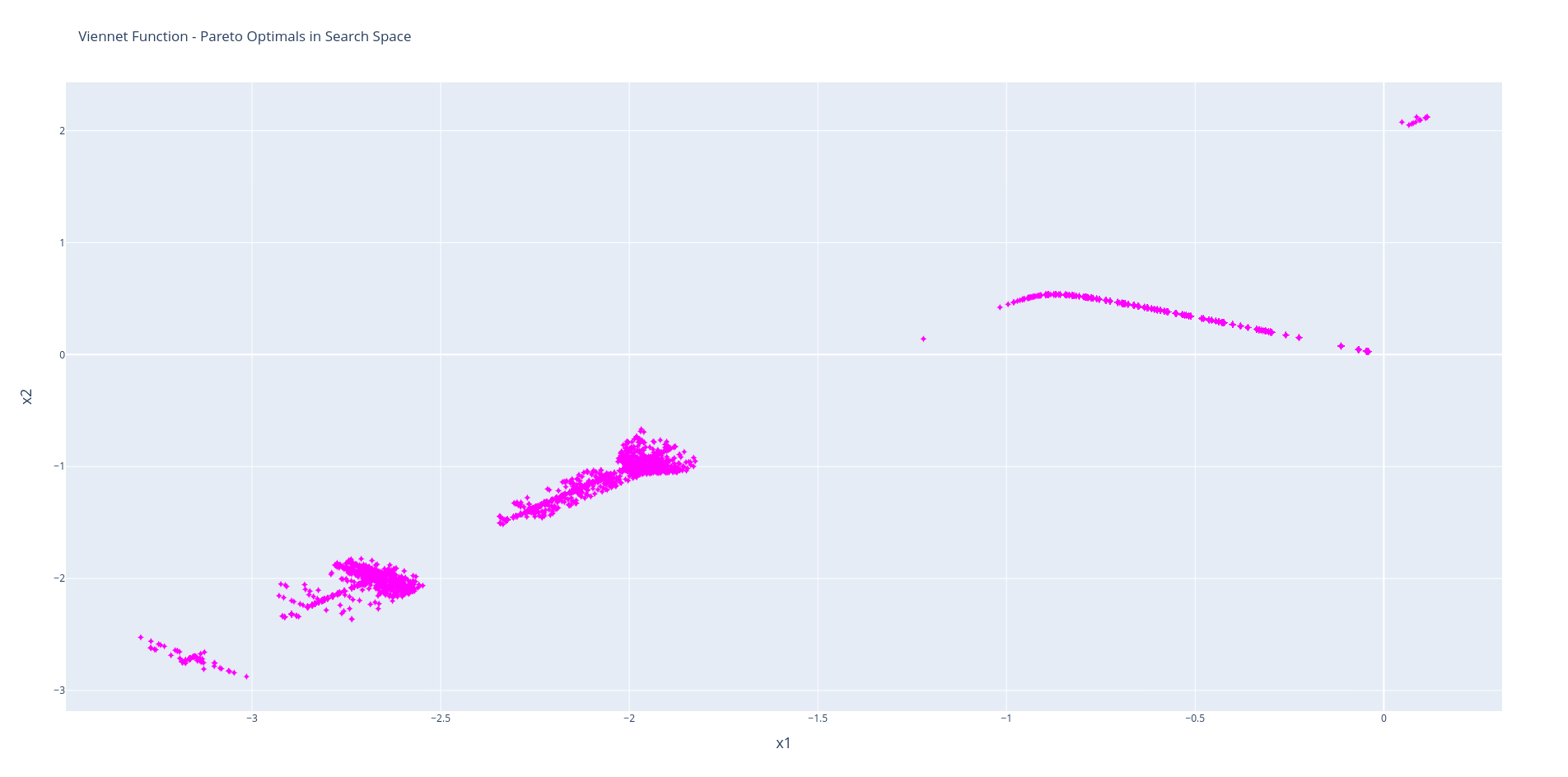}}
    \caption{Viennet test case. Non-dominated points among all the outputs returned by the MGD algorithms; i.e., the points in $\bigcup_{j=1}^N \widetilde{C}_j$ (see \eqref{eq:nondom_j}).}
    \label{fig:V_set}
\end{figure}

\subsection{Analyses' Summary}\label{sec:exp_summary}

Given the analyses of the results of the numerical experiments in this section, we summarize the properties of \eqref{eq:seq_bcktrck_theory} and \eqref{eq:my_LP_explicit} in the following.

Concerning the new backtracking strategy \eqref{eq:seq_bcktrck_theory} and its implementation \Cref{alg:my_mgd_storage}, we observe that it has only advantages and no drawbacks with respect to BT$_{\rm base}$ (see \Cref{alg:mgd_classic}). Indeed, independently of the LP problem used for computing the directions, this new backtracking strategy always improves the probability of a sequence reaching the Pareto set of the problem, reducing the risk of stopping in (or near to) Pareto critical points that are not global Pareto optimals.

Concerning the novel LP problem proposed \eqref{eq:my_LP_explicit}, due to its construction, we observe that it generates sequences of different behavior with respect to the ones generated by \eqref{eq:fliege2000}. This different behavior of the sequences does not show evident advantages if $\v{x}\in\R^n$ is not a Pareto critical or if $m=2$ and $\v{x}$ is Pareto critical; actually, sometimes its performances are even slightly smaller (see \Cref{tab:global_pareto_percentage}). Nonetheless, \eqref{eq:my_LP_explicit} demonstrate crucial importance if used together with \eqref{eq:seq_bcktrck_theory} on problems characterized by large and diffused ``static regions'' like the Viennet problem and exemplified by \Cref{fig:comparison_lemmacases}c. Indeed, in these situations, the theoretical properties of \eqref{eq:my_LP_explicit} together with the ones of \eqref{eq:seq_bcktrck_theory} guarantee the building of sequences that more probably reach a good approximation of a global Pareto optimal; on the contrary, using \eqref{eq:fliege2000} this probability is not necessarily increased too much.

Summarizing, we can say that MGD methods \ref{eq:seq_bcktrck_theory}-\ref{eq:fliege2000} and \ref{eq:seq_bcktrck_theory}-\ref{eq:my_LP_explicit} are always preferable for solving a MOO problem with respect to BT$_{\rm base}$-\ref{eq:fliege2000} and BT$_{\rm base}$-\ref{eq:my_LP_explicit}, respectively. Concerning \eqref{eq:fliege2000} and \eqref{eq:fliege2000}, there are no reasons to prefer one or another, unless there is the suspect that the MOO problem is characterized by large regions of Pareto critical points in its domain of the type described in \Cref{lem:my_LP} - item 2.3 (e.g., like Viennet problem, see \Cref{fig:V_dom_critregions}).

\section{Conclusion}\label{sec:conclusion}

In this paper, we introduced \eqref{eq:my_LP_explicit}, a novel LP problem specifically designed for computing directions in MGD methods. This new LP problem is derived from \eqref{eq:fliege2000}, introduced by Fliege and Svaiter in \cite{Fliege2000}. Through rigorous theoretical analysis (see \Cref{lem:my_LP}), we demonstrated that our proposed LP formulation surely returns non-null directions if the point $\v{x}$ of the sequence is Pareto critical but under particular conditions; namely, when at $\v{x}$ there is at least one non-ascent direction that is a descent direction for at least one objective function. This is one of the main properties that distinguishes the new LP problem from the one in \cite{Fliege2000}, which can return null directions anytime the point of the sequence we consider is a Pareto critical point.

Additionally, we developed a new backtracking strategy for MGD methods (see \eqref{eq:seq_bcktrck_theory}). This strategy is characterized by the acceptance of a new point $\v{x}^{(k+1)}=\v{x}^{(k)} + \widehat{\eta}\,\v{p}^{*\,(k)}$ at the last backtracking step if it is non-dominated by $\v{x}^{(k)}$, even if the Armijo condition is not satisfied for all objectives. Furthermore, we introduced a ``storing property'' within this strategy (see \Cref{alg:my_mgd_storage}), ensuring that all points $\v{x}^{(k)}$ which do not dominate $\v{x}^{(k+1)}$, and vice-versa, are stored. This innovation aims to improve the efficiency and robustness of the backtracking process, improving the probability of any sequence to reach the Pareto set. We provided theoretical proof of the convergence properties for a MGD method that incorporates our new backtracking strategy. 

To validate our theoretical findings, we conducted numerical experiments to evaluate the performance of the new methods compared to the baseline method taken from \cite{Fliege2000}. Our experiments revealed several advantages of using the new backtracking strategy, showing consistent improvements in the performance of the MGD methods.

While the new LP problem demonstrated good behavior, it did not present significant advantages over the previous formulation, except when paired with the new backtracking strategy in MOO problems characterized by large, static regions. This specific combination exhibited notable benefits, reinforcing the value of the new LP problem.

Future work will focus on extending our new methods to constrained MOO and on applying them to real-world problems. This will enable further validation of their practical utility and potential for broader adoption in various application domains.


\appendix

\section{Strictly Decreasing Backtracking for MGD}\label{sec:strdec_bcktrck}

In this appendix section, we report the pseudocode of the standard MGD algorithm used in \Cref{sec:experiments}; i.e., the algorithm that implements a backtracking strategy looking for a decreasing sequence for all the objectives. Actually, this algorithm is equivalent to a simplified version of \Cref{alg:my_mgd}, where the algorithm stops if the Armijo condition \eqref{eq:armijo_moo_t} is not satisfied for all the objectives, for all $t = 0,\ldots ,\Theta$.

\begin{alg}\label{alg:mgd_classic}
    Let us consider the unconstrained MOO problem \eqref{eq:MOOprob}. Then, we define the following MGD algorithm for the implementation of the classic descent method based on a backtracking strategy looking for a decreasing sequence for all the objectives.
    
    \begin{description}
    
    \item[Data:] $\v{x}^{(0)}\in\R^n$ starting point for \eqref{eq:iterative_descent}; $c_1, \alpha\in (0, 1)$ parameters for the Armijo condition; $\eta_0$ starting value for the step length; $\Theta\in\N$ maximum number of backtracking steps; $K\in\N$ maximum number of iterations; $\mathcal{P}$ sub-problem for computing $\v{p}^{*\,(k)}$ at each iteration.
    
    \item[Procedure:] \quad
        \begin{algorithmic}[1]
        \For{$k=0,\ldots ,(K-1)$}
            \State $\v{p}^{*(k)}\gets$ solution of $\mathcal{P}$, defined by $\v{x}^{(k)}$
            \State $\eta^{(k)}\gets\eta_0$
            \For{$t=0,\ldots ,(\Theta-1)$}
                \If{\eqref{eq:armijo_moo_t} is true for each $i=1,\ldots ,m$}
                    \State break
                \Else
                    \State $\eta^{(k)}\gets \eta^{(k)}\alpha$
                \EndIf
            \EndFor
            \If{$\eta^{(k)}=\eta_0 \alpha^{\Theta}$ and there is $i\in\{1,\ldots ,m\}$ s.t. \eqref{eq:armijo_moo_t} is false}
                \State break
            \Else
                \State $\v{x}^{(k+1)}\gets \v{x}^{(k)} + \eta^{(k)} \v{p}^{*(k)}$
            \EndIf
        \EndFor
        \State $\widehat{\v{x}}\gets \v{x}^{(k)}$
        \State \Return: $\widehat{\v{x}}$
        \end{algorithmic}
    \end{description}
\end{alg}

\section{Multi-Objective Test Problems}\label{sec:moo_test_probs}

In this appendix section, we report the formulations of the three MOO test problems used in the numerical experiments of \Cref{sec:experiments}.
\begin{itemize}    
    \item \textbf{Fonseca-Fleming} \cite{Fonseca1995,Fonseca1998}\textbf{.} The objective functions of the MOO problem are
    \begin{equation*}
        \begin{aligned}
        f_1 &= 1 - e^{-\sum{i=1}^n (x_i - 1/\sqrt{n})^2}\\
        f_2 &= 1 - e^{-\sum{i=1}^n (x_i + 1/\sqrt{n})^2}
    \end{aligned}\qquad.
    \end{equation*}
    The starting points for the MGD methods are sampled with random uniform distribution $\mathcal{U}([-2, 2]^n)$, $n=3$. The maximum number of steps used for the MGD methods is $K=250$.

    \item \textbf{Kursawe} \cite{Kursawe1990}\textbf{.} The objective functions of the MOO problem are
    \begin{equation*}
    \begin{aligned}
        f_1 &= \sum_{i=1}^2 -10 \, e^{-0.2 \sqrt{x_i^2+x_{i+1}^2}}\\
        f_2 &= \sum_{i=1}^3 (|x_i|^{0.8} + 5\sin(x_i^3))
    \end{aligned}\qquad.
    \end{equation*}
    The starting points for the MGD methods are sampled with random uniform distribution $\mathcal{U}([-1.5, 0.5]^n)$, $n=3$. The maximum number of steps used for the MGD methods is $K=1500$.

    \item \textbf{Viennet} \cite{Viennet1996}\textbf{.} The objective functions of the MOO problem are
    \begin{equation*}
    \begin{aligned}
        f_1 &= 0.5(x_1^2+x_2^2) + \sin(x_1^2 + x_2^2)\\
        f_2 &= \frac{(3x_1 - 2x_2 + 4)^2}{8} + \frac{(x_1 + x_2 + 1)^2}{27} + 15\\
        f_3 &= \frac{1}{x_1^2 + x_2^2 + 1} -1.1 \, e^{-(x_1^2 + x_2^2)}
    \end{aligned}\qquad.
    \end{equation*}
    The starting points for the MGD methods are sampled with random uniform distribution $\mathcal{U}([-3, 1.5]^n)$, $n=2$. The maximum number of steps used for the MGD methods is $K=7500$.
\end{itemize}



\section*{Acknowledgements}

This study was carried out within the FAIR-Future Artificial Intelligence Research and received funding from the European Union Next-GenerationEU (PIANO NAZIONALE DI RIPRESA E RESILIENZA (PNRR)–MISSIONE 4 COMPONENTE 2, INVESTIMENTO 1.3---D.D. 1555 11/10/2022, PE00000013). This manuscript reflects only the authors’ views and opinions; neither the European Union nor the European Commission can be considered responsible for them.

\subsubsection*{Code Availability:}
The code for MGD methods illustrated in this paper is available at: \url{https://github.com/Fra0013To/MGD}.


 \bibliographystyle{elsarticle-num} 
 \bibliography{References/aiPapers, References/DellaPapers, References/optimizationPapers}





\end{document}